\documentclass[twoside, a4paper, 12pt, bibliography=totoc]{scrartcl}
\usepackage[shortcuts]{extdash}
\usepackage{graphicx}
\usepackage{xcolor}
\usepackage{pifont}
\usepackage[a4paper]{anysize}
\usepackage[round,colon]{natbib}
\usepackage{times}
\usepackage[hyphens]{url}
\usepackage[german,english]{babel}
\usepackage{moreverb}
\usepackage{amsmath}
\usepackage{amsfonts}
\usepackage{amssymb}
\usepackage{amsthm}
\usepackage{thmtools}
\PassOptionsToPackage{dvipdfmx}{graphicx}
\usepackage{graphicx}
\usepackage{environ}
\usepackage{bm}
\usepackage[utf8]{inputenc}
\usepackage[T1]{fontenc}
\usepackage{hyperref}
\usepackage{breakurl}
\usepackage{nameref}
\marginsize{3.0cm}{3.0cm}{2.0cm}{2.0cm}
\usepackage[right]{showlabels}

\usepackage{mathtools}
\definecolor{blue4}{HTML}{00008E}
\definecolor{red2}{HTML}{CE0000}

\declaretheoremstyle[
  headfont=\color{blue4}\normalfont\bfseries,
  bodyfont=\color{blue4}\normalfont\itshape,
]{colored}
\declaretheorem[
  style=colored,
  name=Lemma,
]{lemma}
\newcommand{\trarxiv}[2]{#2}
\begin{document}
\pagenumbering{gobble}

% changing line labels in the moreverb package (listing env.)
\def\listinglabel#1{\llap{\tiny\ttfamily\the#1}\hskip\listingoffset\relax}

\newcommand{\trtitle}{Derivation of Symmetric PCA Learning Rules from
  a Novel Objective Function}
\newcommand{\tryear}{2019}

%===========================================================================
%\titlehead{\includegraphics[width=\textwidth]{logomittext-en_flat_mod.eps}}
\trarxiv{%
  \titlehead{%
    \includegraphics[width=8cm]{Faculty-of-Technology_vektorisiert.eps}}
}{%
  \titlehead{\hspace*{1cm}}
}
\subject{\hspace*{1cm}}
\title{%
  \vspace*{-3cm}
  \trtitle%
}
\trarxiv{%
  \author{\textsf{Ralf Möller}}
}{%
  \author{%
    Ralf Möller\\
    Computer Engineering Group, Faculty of Technology\\
    Bielefeld University, Bielefeld, Germany\\
    \url{www.ti.uni-bielefeld.de}
  }
}
\trarxiv{%
  \date{\normalsize\textsf{\tryear\\[5mm]version of \today}}
}{%
  \date{\hspace*{1cm}}
}
\maketitle
%===========================================================================

%\input /home/moeller/tex/mathcmds-1-3.tex
%===========================================================================
%
% mathcmds.tex --
%
%
% Ralf Moeller <moeller@mpipf-muenchen.mpg.de>
%
%   Copyright (C) 2002
%   Cognitive Robotics Group
%   Max Planck Institute for Psychological Research
%   Munich
%
% 1.0 / 10. Jul 02 (rm)
% - from scratch
% 1.1 / 24. Oct 04 (rm)
% - pmat, rank, sgn
% 1.2 / 25. Oct 04 (rm)
% - pihalf
% 1.3 /  8. Feb 05 (rm)
% - fixed bug in \matN
% 1.4 / 16. Jan 06 (rm)
% - new: grad
% 1.5 / 22. Jan 06 (rm)
% - gradv, df, matGamma, matBeta
% 1.6 / 25. Apr 06 (rm)
% - ddfs, ddfsq
% 1.7 /  7. Mar 07 (rm)
% - vecmu, vecnu, cov
% 1.8 / 12. Feb 08 (rm)
% - matPhi
% 1.9 / 20. May 08 (rm)
% - atantwo operator
% 1.10 /  7. Oct 08 (rm)
% - vecunull, vecvnull
% 1.11 /  9. Oct 08 (rm)
% - matUnull, matVnull, matSnull
% 1.12 / 26. Feb 09 (rm)
% - asin operator
% - eps shorthand for varepsilon
% 1.13 /  2. Mar 09 (rm)
% - acos operator
% 1.14 /  4. Aug 09 (rm)
% - atan operator
% 1.15 /  3. Aug 10 (rm)
% - vecone
% 1.16 /  2. Nov 15 (rm)
% - \bf -> \mathbf
% 1.17 / 19. Jan 17 (rm)
% - mathcmds1-2!
% - vecxdot, corrections to other dot commands
% 1.18 / 20. May 19 (rm)
% - matAnull
% - \bar -> \bar{}
% - \rm -> \textrm
% 1.19 / 27. May 19 (rm)
% - \betanull
%
%===========================================================================

% sums etc.
\newcommand{\summe}[2]{\sum\limits_{#1}^{#2}}
\newcommand{\produkt}[2]{\prod\limits_{#1}^{#2}}
\newcommand{\sumin}{\summe{i=1}{n}}
\newcommand{\sumjn}{\summe{j=1}{n}}
\newcommand{\sumkn}{\summe{k=1}{n}}
\newcommand{\sumim}{\summe{i=1}{m}}
\newcommand{\sumjm}{\summe{j=1}{m}}
\newcommand{\prodin}{\produkt{i=1}{n}}
\newcommand{\prodim}{\produkt{i=1}{m}}
\newcommand{\prodjm}{\produkt{j=1}{m}}
% matrices
\newcommand{\matA}{\mathbf{A}}
\newcommand{\matAnull}{\bar{\mathbf{A}}}
\newcommand{\matB}{\mathbf{B}}
\newcommand{\matBnull}{\bar{\mathbf{B}}}
\newcommand{\matC}{\mathbf{C}}
\newcommand{\matD}{\mathbf{D}}
\newcommand{\matDnull}{\bar{\mathbf{D}}}
\newcommand{\matE}{\mathbf{E}}
\newcommand{\matF}{\mathbf{F}}
\newcommand{\matG}{\mathbf{G}}
\newcommand{\matH}{\mathbf{H}}
\newcommand{\matI}{\mathbf{I}}
\newcommand{\matJ}{\mathbf{J}}
\newcommand{\matK}{\mathbf{K}}
\newcommand{\matL}{\mathbf{L}}
\newcommand{\matM}{\mathbf{M}}
\newcommand{\matMnull}{\bar{\mathbf{M}}}
\newcommand{\matN}{\mathbf{N}}
\newcommand{\matO}{\mathbf{O}}
\newcommand{\matP}{\mathbf{P}}
\newcommand{\matPhi}{\boldsymbol{\Phi}}
\newcommand{\matQ}{\mathbf{Q}}
\newcommand{\matR}{\mathbf{R}}
\newcommand{\matS}{\mathbf{S}}
\newcommand{\matSnull}{\bar{\mathbf{S}}}
\newcommand{\matT}{\mathbf{T}}
\newcommand{\matTheta}{\boldsymbol{\Theta}}
\newcommand{\matU}{\mathbf{U}}
\newcommand{\matUnull}{\bar{\mathbf{U}}}
\newcommand{\matV}{\mathbf{V}}
\newcommand{\matVnull}{\bar{\mathbf{V}}}
\newcommand{\matW}{\mathbf{W}}
\newcommand{\matWnull}{\bar{\mathbf{W}}}
\newcommand{\matX}{\mathbf{X}}
\newcommand{\matXnull}{\bar{\mathbf{X}}}
\newcommand{\matY}{\mathbf{Y}}
\newcommand{\matZ}{\mathbf{Z}}
% special
\newcommand{\matWdot}{{\dot{\mathbf{W}}}}
\newcommand{\matWe}{{\mathbf{W}^e}}
\newcommand{\matBeta}{\mathbf{B}}
\newcommand{\matDelta}{{\mathbf\Delta}}
\newcommand{\matGamma}{{\mathbf\Gamma}}
\newcommand{\matLambda}{{\mathbf\Lambda}}
\newcommand{\matOmega}{{\mathbf\Omega}}
\newcommand{\matXi}{{\mathbf\Xi}}
\newcommand{\matLambdanull}{\bar{\mathbf{\Lambda}}}
\newcommand{\matNull}{\mathbf{0}}
% vectors
\newcommand{\veca}{\mathbf{a}}
\newcommand{\vecalpha}{\boldsymbol{\alpha}}
\newcommand{\vecb}{\mathbf{b}}
\newcommand{\vecc}{\mathbf{c}}
\newcommand{\vecd}{\mathbf{d}}
\newcommand{\vece}{\mathbf{e}}
\newcommand{\veceta}{\boldsymbol{\eta}}
\newcommand{\vecf}{\mathbf{f}}
\newcommand{\vecg}{\mathbf{g}}
\newcommand{\vech}{\mathbf{h}}
\newcommand{\veci}{\mathbf{i}}
\newcommand{\vecj}{\mathbf{j}}
\newcommand{\veck}{\mathbf{k}}
\newcommand{\vecl}{\mathbf{l}}
\newcommand{\vecm}{\mathbf{m}}
\newcommand{\vecmnull}{\bar{\mathbf{m}}}
\newcommand{\vecmu}{\boldsymbol{\mu}}
\newcommand{\vecn}{\mathbf{n}}
\newcommand{\vecnu}{\boldsymbol{\nu}}
\newcommand{\veco}{\mathbf{o}}
\newcommand{\vecp}{\mathbf{p}}
\newcommand{\vecphi}{\boldsymbol{\varphi}}
\newcommand{\vecq}{\mathbf{q}}
\newcommand{\vecr}{\mathbf{r}}
\newcommand{\vecs}{\mathbf{s}}
\newcommand{\vect}{\mathbf{t}}
\newcommand{\vecu}{\mathbf{u}}
\newcommand{\vecunull}{\bar{\mathbf{u}}}
\newcommand{\vecv}{\mathbf{v}}
\newcommand{\vecvnull}{\bar{\mathbf{v}}}
\newcommand{\vecx}{\mathbf{x}}
\newcommand{\vecxdot}{\dot{\mathbf{x}}}
\newcommand{\vecxxi}{\boldsymbol{\xi}}
\newcommand{\vecy}{\mathbf{y}}
\newcommand{\vecz}{\mathbf{z}}
\newcommand{\veczeta}{\boldsymbol{\zeta}}
\newcommand{\vecw}{\mathbf{w}}
\newcommand{\vecwe}{{\mathbf{w}^e}}
\newcommand{\vecwei}{{\mathbf{w}^e_i}}
\newcommand{\vecwek}{{\mathbf{w}^e_k}}
\newcommand{\vecwdot}{\dot{\mathbf{w}}}
\newcommand{\vecwnull}{\bar{\mathbf{w}}}
\newcommand{\vecdwdt}{\frac{d\vecw}{dt}}
\newcommand{\vecnull}{\mathbf{0}}
\newcommand{\vecone}{\mathbf{1}}
\newcommand{\vecvdot}{\dot{\mathbf{v}}}
\newcommand{\vecomega}{\boldsymbol{\omega}}
% norms
\newcommand{\norm}[1]{\|#1\|}
\newcommand{\normF}[1]{\|#1\|_F}
\newcommand{\normm}{\norm{\vecm}}
\newcommand{\normw}{\norm{\vecw}}
\newcommand{\normx}{\norm{\vecx}}
\newcommand{\xp}[1]{\langle #1\rangle}
\newcommand{\xpvecm}{\xp{\vecm}}
\newcommand{\normxpm}{\norm{\xpvecm}}
% other
\newcommand{\betanull}{\bar{\beta}}
\newcommand{\lambdanull}{\bar{\lambda}}
\newcommand{\Wnull}{\bar{W}}
\newcommand{\half}{\frac{1}{2}}
\newcommand{\third}{\frac{1}{3}}
\newcommand{\quarter}{\frac{1}{4}}
\newcommand{\order}[1]{{\cal O}(#1)}
\newcommand{\ddt}{\frac{d}{dt}}
\newcommand{\vecstk}[1]{\left(\begin{array}{c}#1\end{array}\right)}
\newcommand{\blkstk}[1]{\left(\begin{array}{c|c}#1\end{array}\right)}
\newcommand{\pmat}[1]{\begin{pmatrix}#1\end{pmatrix}}
%
% mathops (like log or limsup)
%\def\tr{\mathop{\textrm{tr}}\nolimits}
%\def\diag{\mathop{\textrm{diag}}\nolimits}
%\newcommand{\Diag}[2]{\mathop{\textrm{diag}}\limits_{#1}^{#2}}
%\def\cov{\mathop{\textrm{cov}}\nolimits}
%\def\var{\mathop{\textrm{var}}\nolimits}
% new version:  5. Jul 19 (rm)
\newcommand{\tr}{\operatorname{tr}\nolimits}
\newcommand{\dg}{\operatorname{dg}\nolimits}
\newcommand{\diag}{\operatorname{diag}\nolimits}
\newcommand{\Diag}[2]{\operatorname*{diag}\limits_{#1}^{#2}}
\newcommand{\cov}{\operatorname{cov}\nolimits}
\newcommand{\var}{\operatorname{var}\nolimits}
\newcommand{\sign}{\operatorname{sign}\nolimits}

\newcommand{\rank}{\operatorname{rank}}
\newcommand{\sgn}{\operatorname{sgn}}
\newcommand{\atantwo}{\operatorname{atan2}}
\newcommand{\asin}{\operatorname{asin}}
\newcommand{\acos}{\operatorname{acos}}
\newcommand{\atan}{\operatorname{atan}}
\newcommand{\adj}{\operatorname{adj}}

% partial derivative
\newcommand{\ddf}[2]{\frac{\partial #1}{\partial #2}}
\newcommand{\df}[1]{\ddf{}{#1}}
% ... second order
\newcommand{\ddfs}[3]{\frac{\partial^2 #1}{\partial #2\partial #3}}
\newcommand{\ddfsq}[2]{\frac{\partial^2 #1}{\partial {#2}^2}}
\newcommand{\pihalf}{\frac{\pi}{2}}

% \grad
\newcommand{\grad}{\nabla}
\newcommand{\gradv}{\boldsymbol{\nabla}}

% \eps
\newcommand{\eps}{\varepsilon}

%\newcommand{\mattW}{\tilde{\matW}}
%\newcommand{\vectw}{\tilde{\vecw}}
%\newcommand{\mattU}{\tilde{\matU}}
%\newcommand{\vectu}{\tilde{\vecu}}
%\newcommand{\mattV}{\tilde{\matV}}
%\newcommand{\vectv}{\tilde{\vecv}}

%===========================================================================

\trarxiv{%
}{ %
  \renewcommand{\showlabelsetlabel}[1]{}
}

%===========================================================================

\trarxiv{%
  \newcommand{\lemmaref}[1]{\textbf{\color{red2}Lemma~#1}}
  \newcommand{\lemmasep}{\vspace*{10mm}}
}{%
  \newcommand{\lemmaref}[1]{\textbf{Lemma~#1}}
  \newcommand{\lemmasep}[1]{}
}

%===========================================================================

\trarxiv{%
  \newcommand{\todo}[1]{\textbf{\color{orange}$\bullet$~TODO: #1}\\}
  \newcommand{\todocomment}[1]{\textbf{\color{orange}(TODO: #1)}}
}{%
  \newcommand{\todo}[1]{}
  \newcommand{\todocomment}[1]{}
}

%===========================================================================
%\vspace*{-1cm}
\trarxiv{\newpage}{\vspace*{-1cm}}

\begin{abstract}
\trarxiv{\noindent\sloppy \textbf{Abstract}\\[0.5cm]}{}
\noindent Neural learning rules for principal component / subspace
analysis (PCA / PSA) can be derived by maximizing an objective
function (summed variance of the projection on the subspace axes)
under an orthonormality constraint. For a subspace with a single axis,
the optimization produces the principal eigenvector of the data
covariance matrix. Hierarchical learning rules with deflation
procedures can then be used to extract multiple eigenvectors. However,
for a subspace with multiple axes, the optimization leads to PSA
learning rules which only converge to axes spanning the principal
subspace but not to the principal eigenvectors. A modified objective
function with distinct weight factors had to be introduced produce PCA
learning rules. Optimization of the objective function for multiple
axes leads to symmetric learning rules which do not require deflation
procedures. For the PCA case, the estimated principal eigenvectors are
ordered (w.r.t. the corresponding eigenvalues) depending on the order
of the weight factors.

Here we introduce an alternative objective function where it is not
necessary to introduce fixed weight factors; instead, the alternative
objective function uses squared summands. Optimization leads to
symmetric PCA learning rules which converge to the principal
eigenvectors, but without imposing an order. In place of the diagonal
matrices with fixed weight factors, variable diagonal matrices appear
in the learning rules. We analyze this alternative approach by
determining the fixed points of the constrained optimization. The
behavior of the constrained objective function at the fixed points is
analyzed which confirms both the PCA behavior and the fact that no
order is imposed. Different ways to derive learning rules from the
optimization of the objective function are presented. The role of the
terms in the learning rules obtained from these derivations is
explored.

\trarxiv{%
  \vspace*{1cm}
  Please cite as: Ralf Möller. {\em \trtitle}. Technical Report, Computer
  Engineering, Faculty of Technology, Bielefeld University, \tryear,
  version of \today, \url{www.ti.uni-bielefeld.de}.
}{%
}
\end{abstract}

%===========================================================================

%\thispagestyle{empty}
\newpage
%\thispagestyle{empty}
%\begin{footnotesize}
\tableofcontents
%\end{footnotesize}
%\thispagestyle{empty}
\newpage
\pagenumbering{arabic}

% ===========================================================================

\sloppypar
% split eqnarray over multiple pages
% https://github.com/davidar/blog/wiki/2008-07-26-splitting-eqnarray-over-multiple-pages
\allowdisplaybreaks

\newpage

% TODO
% - it might be clearer to define ``orthogonal similarity transformation''
%   at the beginning
% - use ``critical point'' instead of ``fixed point'' when we refer to
%   obj. fct.: seems to be handled consistently in the text

\section{Introduction}
%=====================

Neural network approaches to Principal Component Analysis (PCA) or
Principal Subspace Analysis (PSA) have received continuous attention
since the initial contributions by \cite{nn_Oja82}, \cite{nn_Oja89},
\cite{nn_Sanger89}, \cite{nn_Oja92a}, \cite{nn_Oja92},
\cite{nn_Oja92b}, and \cite{nn_Xu93}. A recent overview is given in
the comprehensive textbook by \cite{nn_Kong17}; an influential early
textbook was published by \cite{nn_Diamantaras96}.

With respect to network structure, there are two different
approaches. {\em Hierarchical networks} are chains of multiple
single-component principal component analyzers. Deflation
\cite[]{nn_Sanger89} is used to remove the projection onto the
estimated principal eigenvector (corresponding to the largest
eigenvalue) from the data, such that the next unit will estimate the
next principal eigenvector (with the second-largest eigenvalue), and
so on. In {\em symmetrical networks}, all units see the same input and
compete to represent the principal eigenvectors or principal subspace
axes. An order with respect to the corresponding eigenvalues is
imposed in some of these learning rules by the introduction of
distinct weight factors (but this not necessary as will be shown in
this work).

The operation of both hierarchical and symmetrical networks is
ultimately determined by the objective function from with they are
derived. Traditionally, the variance of the projection on the weight
vectors (subspace axes) is maximized; this is identical to the
minimization of the reconstruction error \cite[see
  e.g.][p.45]{nn_Diamantaras96}. If a {\em single} subspace axis is
determined, maximization of the projected variance leads to PCA
learning rules where the subspace axis converges to the principal
eigenvector (since a 1D subspace is confined to its axis). This forms
the basis of hierarchical PCA networks. However, if the {\em sum} of
the projected variances on {\em multiple} subspace axis is maximized
(resulting in symmetric rules), the subspace axes only converge
towards the principal subspace, i.e. they span the same subspace as
the principal eigenvectors, but do not necessarily coincide with them
\cite[]{nn_Oja89,nn_Xu93}. Weight factors with pairwise different
values had to be introduced into the objective function to break the
symmetry such that the network converges towards the principal
eigenvectors \cite[]{nn_Oja92,nn_Oja92b,nn_Xu93}. If symmetric rules
derived from weighed objective functions are written in matrix form,
these weight factors appear as fixed diagonal matrices. The order
of the weight factors in these matrices determines the order of the
eigenvectors estimated by the network, thus the networks are
ultimately not fully symmetric. Moreover, the chosen values and range
of the fixed weight factors may affect the convergence speed, thus it
may be necessary to use different sets for different data
distributions.

Our goal in this work was to produce {\em fully} symmetric learning
rules which nevertheless converge towards the principal
eigenvectors. We approach this at the root of the methods by
suggesting an alternative objective function. This objective function
resembles the sum of projected variances, but uses squared summands
instead. Interestingly, the learning rules derived from this novel
objective function also contain diagonal matrices in the same location
as the fixed weight factor matrices, but these matrices depend on the
covariance matrix of the data and on the estimated axes. We will show
that squaring the terms in the objective function leads to additional
fixed points which do not coincide with the principal
eigenvectors. However, as we will also demonstrate (at least
implicitly via an analysis of the behavior the objective function),
these fixed points are local minima or saddle points and will thus be
avoided by the learning rule. The network will therefore converge
towards the principal eigenvectors which are found at local maxima of
the alternative objective function.

After recapitulating the traditional objective function in section
\ref{sec_tradof} and introducing the alternative objective in section
\ref{sec_newof}, we look at four special cases in section
\ref{sec_special_cases}. Three of these relate to the traditional
objective function. Of these three, one special case is related to PSA
rules, the other two are weighted versions which show PCA
behavior. The fourth special case relates to the novel objective
function. In most parts of the work, we derive results for all four
special cases.

In section \ref{sec_fp_constrained_opt}, we derive fixed points for
the optimization under orthonormality constraints of the weight
matrices (containing the estimated subspace axes). We use the
Lagrange-multiplier method to express the constrained optimization
problem. The Lagrange multipliers are isolated from the
equations. However, this requires a non-equivalent transformation
which presumably affects the solution sets of the resulting
fixed-point equations. There are two different ways to re-insert and
thus eliminate the Lagrange multipliers, resulting in two groups of
fixed-point equations, one showing 'uninteresting' PSA behavior, the
other 'interesting' PCA behavior (except for the case derived from the
non-weighted traditional objective function).

After an analysis of the overall fixed points in section
\ref{sec_overall_fp} --- which can coarsely differentiate between PCA
and PSA rules --- we proceed by determining the fixed points of all
special cases in the two groups in section \ref{sec_fp}. An analysis
of the solution space shows that three learning rules from the first
group contain spurious solutions which results in PSA behavior. In
contrast, the corresponding three learning rules of the second group
exhibit no spurious solutions and show PCA behavior. An exception is
the case derived from the novel objective function. Since the
equations contain a variable diagonal matrix instead of fixed
weight-factor matrices, additional fixed points appear where the
weight vectors are not principal eigenvectors.

Section \ref{sec_behavior} analyzes the behavior of the constrained
objective functions at the critical points. Here we use techniques
from the treatment of Stiefel manifolds to determine which critical
points of the objective function are maxima, minima, or saddle
points. For one case derived from the weighted traditional objective
function we can confirm its well-known PCA behavior. For the case
derived from the novel objective function we find that only critical
points in the principal eigenvectors are maxima, while all other
critical points are either minima or saddle points. It can be expected
(but has not been established formally) that the corresponding
learning rules will therefore avoid the minima and saddle points and
converge towards the maxima and therefore towards principal
eigenvectors.

In section \ref{sec_derivation} we derive learning rules by three
approaches. In the {\em short form}, we just turn the fixed-point
equations into an ordinary differential equations. In the {\em long
  form}, we turn the Lagrange multipliers from fixed-point versions
into variables, insert them into the objective function, and determine
the gradient. The third approach derives learning rules from gradients
on Stiefel manifolds. The presumed role played by the different terms
appearing in all learning rules is explored. In this work, we only
analyze {\em averaged} learning rules which explicitly contain the
covariance matrix of the data $\matC = E\{\vecx \vecx^T\}$ where
$\vecx$ is a zero-mean data vector. {\em Online learning rules} can be
derived by informally approximating $\matC \approx \vecx \vecx^T$
\cite[see e.g.][]{own_Moeller04a} and using decaying learning rates,
but this has so far not been explored.

%############################################################################
%############################################################################
%############################################################################

\section{Abbreviations}
%======================

\begin{description}
\item[PSA] Principal Subspace Analysis; refers to learning rules which
  converge to weight vectors spanning the principal subspace
\item[PCA] Principal Component Analysis; refers to learning rules
  where the weight vectors converge to the eigenvectors of the
  covariance matrix
\end{description}

%############################################################################
%############################################################################
%############################################################################

\section{Notation}
%=================

Matrix and vector notation: expressions $(\matA)_{ij} = \matA_{ij} =
A_{ij}$ exchangeably denote element $(i,j)$ of matrix
$\matA$. Expression $(a_{ij})_{ij}$ denotes a matrix with elements
$a_{ij}$ at row $i$ and column $j$. Expression $a_j$ denotes element
$j$ of vector $\veca$. Expression $(\veca_j)_i$ denotes element $i$ of
vector $\veca_j$.

$\delta_{ij}$ is Kronecker's delta. $\matI$ is the identity matrix,
sometimes with dimension $n$ indicated as $\matI_n$. $\matNull$ is a
zero matrix, sometimes with dimensions $n, m$ indicated as
$\matNull_{n,m}$, or a zero vector, sometimes with dimension $n$
indicated as $\vecnull_n$. In cases where the dimensions of null
matrices should be obvious, we just write $\matNull$ for null matrices
of different sizes, even if they appear in the same equation.

$\matXi$ is a diagonal sign matrix where the diagonal elements are
$\xi_i = \pm 1$, sometimes with dimension $n$ indicated as
$\matXi_n$. Note that $\matXi^T \matXi = \matXi \matXi^T = \matI$
(thus $\matXi$ is orthogonal) and $\matXi^T \matD \matXi = \matD$ if
$\matD$ is diagonal.

$n$ denotes the dimension of the problem, $m$ the number of
eigenvector estimates ($1 \leq m \leq n$). $\matC$ is the $n \times n$
covariance matrix $\matC = E\{\vecx\vecx^T\}$ where $\vecx$ is the
zero-mean data vector ($E\{\vecx\} = \vecnull$). The $n \times m$
weight matrix $\matW$ contains the $m$ weight vectors $\vecw_j$ in its
columns. The $n \times n$ matrix $\matV$ contains the $n$ eigenvectors
of $\matC$ in its columns, denoted by $\vecv_i$. $\matV$ is
orthogonal: $\matV^T\matV = \matV\matV^T=\matI$. The $n \times n$
diagonal matrix $\matLambda$ contains the eigenvalues $\lambda_i$ on
the main diagonal. We assume the following order of the eigenvalues
\begin{equation}\label{eq_order_lambda}
\lambda_1 > \ldots > \lambda_n > 0.
\end{equation}
The spectral decomposition of $\matC$ is (see \lemmaref{\ref{lemma_spectral}}):
\begin{equation}\label{eq_C_spectral}
  \matC
  = \matV \matLambda \matV^T
  = \summe{i = 1}{n} \lambda_i \vecv_i \vecv_i^T.
\end{equation}

Some learning rules use pairwise different, strictly positive, fixed
coefficients $\theta_j$ in their objective functions ($\theta_i \neq
\theta_j$ for $i \neq j$), combined in a diagonal $m \times m$ matrix
$\matTheta$. Other learning rules use fixed diagonal matrices with
pairwise different, strictly positive entries $\varOmega_j$ in their
weight vector constraints ($\varOmega_i \neq \varOmega_j$ for $i \neq
j$), combined in a diagonal $m \times m$ matrix $\matOmega$. Without
loss of generality, we assume that the coefficients are sorted:
\begin{eqnarray}
\label{eq_order_theta}
&&\theta_1 > \ldots > \theta_n > 0\\
\label{eq_order_omega}
&&\varOmega_1 > \ldots > \varOmega_n > 0.
\end{eqnarray}

In some derivations, the weight vectors are projected into the space
of the eigenvectors:
\begin{eqnarray}
  \veca_j &=& \matV^T \vecw_j\\
  \matA &=& \matV^T \matW \label{eq_AVW}\\
  \vecw_j &=& \matV \veca_j = \summe{i=1}{n} (\veca_j)_i \vecv_i\\
  \matW &=& \matV \matA. \label{eq_WVA}
\end{eqnarray}
If $\vecw_i^T \vecw_j = \delta_{ij} \varOmega_i$, we see that $\veca_i^T
\veca_j = \vecw_i^T \matV \matV^T \vecw_j = \vecw_i^T \vecw_j =
\delta_{ij} \varOmega_i$, thus
\begin{eqnarray}
  \veca_i^T \veca_j &=& \delta_{ij} \varOmega_i\\
  \matA^T \matA &=& \matOmega.
  \label{eq_ATA}
\end{eqnarray}
In this case we have $\|\vecw_j\| = \|\veca_j\| = \varOmega_j^\half$.

Note that in the transformation of several equations we multiply by
orthogonal matrices from the left or from the right which is an
equivalent transformation (and thus invertible): If a matrix
expression $\matF$ (size $n \times n$) is left-multiplied by an
orthogonal matrix $\matV$ (size $n \times n$, $\matV^T\matV =
\matV\matV^T = \matI$), we get $\matV \matF$. Again left-multiplying
by $\matV^{-1} = \matV^T$ gives $\matV^T \matV \matF = \matI \matF =
\matF$, thus the transformation can be inverted. The same holds for
left-multiplication with $\matV^T$ since also $\matV \matV^T \matF =
\matI \matF = \matF$. The argument can also be applied to
multiplication from the right.

However, multiplying by a semi-orthogonal $n \times m$ matrix $\matA$
with $m < n$ (which is defined by $\matA^T \matA = \matI$, but
generally not $\matA \matA^T = \matI$) is a non-equivalent
transformation (which may introduce spurious solutions): If $\matF$ is
a matrix expression of size $n \times m$, left-multiplication by
$\matA^T$ gives $\matA^T \matF$. This transformation can't be inverted
since $\matA^{-1}$ does not exist. Trying a left-multiplication with
$\matA$ as for an orthogonal matrix doesn't help since it would give
$\matA \matA^T \matF$ which is generally not identical to $\matF$
since $\matA \matA^T = \matI$ does not generally hold. In this work,
we are forced to left-multiply by a semi-orthogonal matrix to isolate
Lagrange multipliers in section \ref{sec_first_variant}. This seems to
introduce spurious solutions into the set of fixed points (section
\ref{sec_fp}).

The matrix $\diag_{i=1}^n\{x_i\}$ is a diagonal matrix of dimension
$n$ with diagonal elements $x_i$. The diagonalization operator
$\dg\{\matX\}$ applied to a square matrix $\matX$ of dimension $n$
produces a diagonal matrix of dimension $n$ which has the same
diagonal elements as $\matX$. We have
\begin{equation}
  \Diag{i=1}{n} \{\matX_{ii}\}
  =
  \Diag{i=1}{n} \{\vece_i^T\matX\vece_i\}
  =
  \dg \{\matX\}
\end{equation}
and
\begin{equation}
  (\dg \{\matX\})_{ij}
  =
  \matX_{ij} \delta_{ij}
  =
  \matX_{ii} \delta_{ij}.
\end{equation}
%

%############################################################################
%############################################################################
%############################################################################

\section{Objective Functions for PCA and PSA}
%============================================

\subsection{Traditional Objective Function}\label{sec_tradof}
%------------------------------------------

Learning rules for principal component analysis, PCA, or principal
subspace analysis, PSA, are often derived by a constrained
minimization of the (weighted) mean-square reconstruction error or,
equivalently, by the constrained maximization of the (weighted)
variance of the projection
\begin{equation}\label{eq_objfct}
  J(\matW)
  = \half \summe{j=1}{m} \theta_j E\{(\vecw_j^T \vecx)^2\}
  = \half \summe{j=1}{m} \theta_j \vecw_j^T \matC \vecw_j
  = \half \tr\{\matW^T \matC \matW \matTheta\}
\end{equation}
(note that \lemmaref{\ref{lemma_tr_AD}} was used in the transition
$\sum_j \vecw_j^T \matC \vecw_j \theta_j = \sum_j (\matW^T \matC
\matW)_{jj} \theta_j = \tr\{\matW^T \matC \matW \matTheta\}$) under
the constraint
\begin{equation}\label{eq_objfct_constraint}
  \vecw_i^T \vecw_j = \varOmega_i \delta_{ij}
  \;\;\mbox{or}\;\;
  \matW^T\matW = \matOmega
\end{equation}
\cite[see e.g.][chapter 3]{nn_Diamantaras96}. In the cases we
consider, we have $\matTheta = \matI_m$ or $\matOmega = \matI_m$
(where ``or'' is non-exclusive).

Learning rules like Oja's Subspace Rule \cite[]{nn_Oja89} (PSA), Oja's
Weighted Algorithm \cite[]{nn_Oja92,nn_Oja92b}, Xu's LMSER Rule
\cite[]{nn_Xu93} (PSA) and its weighted version\footnote{To be
  precise: we managed to derive rule (15a) and a rule {\em similar} to
  rule (15b) described by \cite{nn_Xu93}.}  (PCA), can be derived from
a Lagrange-multiplier framework applied to this objective function
(see section \ref{sec_derivation}). Note that for $m=1$, all rules are
PCA rules.

In the analysis and derivation of these rules, we need the derivative
of the objective function,
\begin{eqnarray}
  \ddf{J}{\vecw_l}
  &=&
  \half \summe{j=1}{m} \df{\vecw_l} (\theta_j \vecw_j^T \matC \vecw_j)\\
  &=&
  \summe{j=1}{m} \theta_j \matC \vecw_j \delta_{jl}\\
  &=&
  \theta_l \matC \vecw_l.
\end{eqnarray}
In the derivations below, we use the abbreviations
\begin{eqnarray}
  \label{eq_tradof_vecm}
  \vecm_l
  &=& \ddf{J}{\vecw_l}
  = \theta_l \matC \vecw_l\\
  \label{eq_tradof_vecmnull}
  \vecmnull_l
  &=& \left.\ddf{J}{\vecw_l}\right|_{\matWnull}
  =  \theta_l \matC \vecwnull_l.\\
  \label{eq_tradof_matM}
  \matM
  &=& \ddf{J}{\matW}
  = \matC \matW \matTheta\\
  \label{eq_tradof_matMnull}
  \matMnull
  &=& \left.\ddf{J}{\matW}\right|_{\matWnull}
  = \matC \matWnull \matTheta.
\end{eqnarray}
Furthermore, we also need the second derivative
\begin{equation}
  \ddf{\vecm_j}{\vecw_l}
  =
  \underbrace{\theta_j \matC}_{\matH_j} \delta_{jl}.
\end{equation}

\subsection{Novel Objective Function}\label{sec_newof}
%------------------------------------

Learning rules derived from the traditional objective function
(\ref{eq_objfct}) are PSA (not PCA) rules unless diagonal matrices
with pairwise different elements are introduced either into their
objective function ($\matTheta$) or into their weight vector
constraints ($\matOmega$); the only exception is the case $m=1$ which
leads to PCA rules. The novel objective function introduced below can
produce true PCA rules without the need of those diagonal
matrices. Actually, as we will see below, diagonal matrices appear
``naturally'' in the learning rules derived from this novel objective
function.

The novel objective function is defined as
\begin{equation}\label{eq_objfct_new}
  J(\matW)
  = \quarter \summe{j=1}{m} (\vecw_j^T \matC \vecw_j)^2
\end{equation}
under the constraint
\begin{equation}\label{eq_objfct_new_constraint}
  \vecw_i^T \vecw_j = \delta_{ij}
  \;\;\mbox{or}\;\;
  \matW^T\matW = \matOmega = \matI_m.
\end{equation}
With respect to the traditional objective function (\ref{eq_objfct}),
we can describe the modification by replacing $\theta_j \coloneqq
\vecw_j^T \matC \vecw_j$.

In the analysis and derivation of learning rules, we need the
derivative of the objective function,
\begin{eqnarray}
  \ddf{J}{\vecw_l}
  &=&
  \half \summe{j=1}{m}
  \left[
    (\vecw_j^T \matC \vecw_j)
    \df{\vecw_l} (\vecw_j^T \matC \vecw_j)
    \right]\\
  &=&
  \summe{j=1}{m}
  \left[
    (\vecw_j^T \matC \vecw_j)
    (\matC \vecw_j \delta_{jl})
    \right]\\
  &=&
  (\vecw_l^T \matC \vecw_l) (\matC \vecw_l).
\end{eqnarray}
We use the following abbreviations:
\begin{eqnarray}
  \label{eq_newof_m}
  \vecm_l
  &=& \ddf{J}{\vecw_l}
  = (\vecw_l^T \matC \vecw_l) \matC \vecw_l\\
  \label{eq_newof_mnull}
  \vecmnull_l
  &=& \left.\ddf{J}{\vecw_l}\right|_{\matWnull}
  = (\vecwnull_l^T \matC \vecwnull_l) \matC \vecwnull_l\\
  \label{eq_newof_matM}
  \matM
  &=& \ddf{J}{\matW}
  = \matC \matW \Diag{j=1}{m}\{\vecw_j^T \matC \vecw_j\}\\
  \label{eq_newof_matMnull}
  \matMnull
  &=& \left.\ddf{J}{\matW}\right|_{\matWnull}
  = \matC \matWnull \Diag{j=1}{m}\{\vecwnull_j^T \matC \vecwnull_j\}.
\end{eqnarray}
We also need the second derivative
\begin{eqnarray}
  \ddf{\vecm_j}{\vecw_l}
  &=&
  \underbrace{[\matC\vecw_j\vecw_j^T\matC
      + (\vecw_j^T\matC\vecw_j)\matC]}_{\matH_j} \delta_{jl}
\end{eqnarray}
which was obtained by applying \lemmaref{\ref{lemma_scalar_vector_deriv}}.

\subsection{Special Cases}\label{sec_special_cases}
%-------------------------

In the sections below, we analyze different special cases of $\matM$
from equations (\ref{eq_tradof_matM}) and (\ref{eq_newof_matM}) and of
the diagonal matrices with pairwise different entries $\matTheta$
(influencing the objective function) and $\matOmega$ (influencing the
constraint). In our nomenclature, the specifier indicates the
objective function by ``T'' for traditional (\ref{eq_objfct}) and by
``N'' for novel (\ref{eq_objfct_new}), a weighted objective function
by ``wJ'', and a weighted constraint by ``wC'':
\begin{eqnarray}
  \mbox{T:} &&
  \matTheta = \matI_m,
  \matW^T\matW = \matOmega = \matI_m,
  \matM = \matC \matW\\
  \mbox{TwJ:} &&
  \matTheta = \Diag{j=1}{m}\{\theta_j\},
  \matW^T \matW = \matOmega = \matI_m,
  \matM = \matC \matW \matTheta\\
  \mbox{TwC:} &&
  \matTheta = \matI_m,
  \matW^T \matW = \matOmega = \Diag{j=1}{m}\{\varOmega_j\},
  \matM = \matC \matW\\
  \mbox{N:} &&
  \matW^T\matW = \matOmega = \matI_m,
  \matM = \matC \matW \Diag{j=1}{m}\{\vecw_j^T \matC \vecw_j\}
\end{eqnarray}
%

%############################################################################
%############################################################################
%############################################################################

\section{Fixed-Point Equations from Constrained Optimization}
%============================================================
\label{sec_fp_constrained_opt}

Two different objective functions $J(\matW)$ were introduced in
section \ref{sec_tradof} and \ref{sec_newof} together with their
constraints. Constrained optimization by the Lagrange-multiplier
method starts by defining the modified objective function $J^*(\matW)
= J(\matW) + C(\matB,\matW)$ where $C$ is the constraint term which
includes the Lagrange multipliers $\matB$. The choice of the
constraint term directly influences which type of learning rule is
derived. \citet[their eqn.~(8), modified here] {nn_Chatterjee00} use a
constraint term with a triangular matrix of Lagrange multipliers to
ensure orthonormality of the weight vectors
\begin{equation}\label{eq_constraint_nonsymm}
  C(\matBeta,\matW) =
  \half \summe{j=1}{m} \summe{k=1}{i}
  \beta_{jk} (\vecw_j^T \vecw_k - \delta_{ij}).
\end{equation}
This leads to a non-symmetrical learning rule: In a chain of neurons,
each neuron chooses a weight vector which is orthogonal to the weight
vectors of all previous neurons in the chain. Even for special case T
(see section \ref{sec_special_cases}), this would result in a PCA
rule. Our goal is to derive symmetric learning rules where all neurons
behave in the same way. One possible constraint would be
\begin{equation}\label{eq_constraint_symm_beta}
  C(\matBeta,\matW) =
  \half \summe{j=1}{m} \summe{k=1}{m}
  \beta_{jk} (\vecw_j^T \vecw_k - \varOmega_j \delta_{ij})
\end{equation}
(note the different upper index in the second sum; we also introduced
$\varOmega_j$). This constraint term is sufficient to derive the fixed
points and the ``short form'' of the learning rules (see section
\ref{sec_derivation_short}), but leads to problems with the derivation
of the the ``long form'' (see section \ref{sec_derivation_long}) since
there we need to insert a term for the Lagrange multipliers into the
modified objective function. Moreover, one would expect a symmetry
constraint on $\matBeta$ since the second factor is identical if $j$
and $k$ are exchanged. We therefore introduce the following constraint
term
\begin{equation}\label{eq_constraint_symm_betastar}
  C(\matBeta,\matW) =
  \half \summe{j=1}{m} \summe{k=1}{m}
  \half (\beta_{jk} + \beta_{kj}) (\vecw_j^T \vecw_k - \varOmega_j \delta_{jk})
\end{equation}
where the symmetry is now ensured by the Lagrange multipliers $\half
(\beta_{jk} + \beta_{kj})$. We also managed to derive the ``long
form'' of the learning rules from this constraint term.

Therefore our modified objective function\footnote{I'm grateful to
  Axel Könies for his useful advice on an earlier version of this
  section.} becomes
\begin{equation}
  J^*
  =
  J
  + \half \summe{j=1}{m} \summe{k=1}{m}
  \half (\beta_{jk} + \beta_{kj}) (\vecw_j^T \vecw_k - \varOmega_j \delta_{jk})
  \label{eq_Jstar_gen}
\end{equation}
where $\beta^*_{jk} = \half (\beta_{jk} + \beta_{kj})$ are the
Lagrange multipliers. Note that we do {\em not} assume symmetry
($\beta_{jk} = \beta_{kj}$) of the matrix $\matB$. We obtain the
derivatives with respect to the weight vectors $\vecw_l$ ($l =
1,\ldots,m$)
\begin{eqnarray}
  \ddf{J^*}{\vecw_l}
  &=&
  \ddf{J}{\vecw_l}
  + \half \df{\vecw_l} \left[\summe{j=1}{m} \summe{k=1}{m}
    \half (\beta_{jk} + \beta_{kj})
    (\vecw_j^T \vecw_k - \varOmega_j \delta_{jk})\right]\\
  &=&
  \ddf{J}{\vecw_l}
  + \half \summe{j=1}{m} \summe{k=1}{m}
  \half (\beta_{jk} + \beta_{kj})
  (\delta_{jl} \vecw_k + \delta_{kl} \vecw_j)\\
  &=&
  \ddf{J}{\vecw_l}
  + \half\left(
  \summe{k=1}{m} \half (\beta_{lk} + \beta_{kl}) \vecw_k +
  \summe{j=1}{m} \half (\beta_{jl} + \beta_{lj}) \vecw_j
  \right)\\
  &=&
  \ddf{J}{\vecw_l}
  +
  \summe{j=1}{m} \half (\beta_{jl} + \beta_{lj}) \vecw_j
\end{eqnarray}
which can be written as
\begin{equation}
  \ddf{J^*}{\vecw_l} = \vecm_l + \half
  \summe{j=1}{m} (\beta_{jl} + \beta_{lj}) \vecw_j.
\end{equation}
Note that we would arrive at the same equation by using constraint
term (\ref{eq_constraint_symm_beta}). In matrix form, this equation
becomes
\begin{equation}
  \ddf{J^*}{\matW} = \matM + \half \matW (\matB + \matB^T).
\end{equation}
In the fixed point we have (using $\vecmnull_l$ and $\matMnull$ from
section \ref{sec_tradof} and \ref{sec_newof})
\begin{eqnarray}
  \label{eq_fp_pca_vec}
  \vecmnull_l + \half \summe{j=1}{m} (\beta_{jl} + \beta_{lj}) \vecwnull_j
  &=&
  \vecnull\\
  \label{eq_fp_pca_vec_constraint}
  \vecwnull_i^T \vecwnull_j^T &=& \varOmega_j \delta_{ij}
\end{eqnarray}
or, in matrix form
\begin{eqnarray}
  \label{eq_fp_pca_mat}
  \matMnull + \half \matWnull (\matBnull + \matBnull^T) &=& \matNull\\
  \label{eq_fp_pca_mat_constraint}
  \matWnull^T \matWnull &=& \matOmega.
\end{eqnarray}
Interestingly, at this point we have two ways to proceed which
ultimately lead to different sets of fixed points and different
learning rules. In the first variant, we do {\em not} exploit the
symmetry of $\matB + \matB^T$. We will show in section \ref{sec_fp}
that the fixed points are weight vectors which span the same subspace
as $m$ of the eigenvectors (leading to PSA rules). In the second
variant, we exploit the symmetry and show in section \ref{sec_fp} that
the fixed points coincide with $m$ of the eigenvectors (leading to PCA
rules).

\subsection{First variant}\label{sec_first_variant}
%-------------------------

The following analysis resorts to a non-equivalent transformation
which may introduce spurious fixed points in addition to the correct
fixed points. This step is required to isolate (and ultimately
eliminate) the Lagrange multipliers $\matBnull$. We use the same
transformation as \cite{nn_Chatterjee00}: We left-multiply
(\ref{eq_fp_pca_mat}) by $\matWnull^T$, apply the constraint
(\ref{eq_fp_pca_mat_constraint}), and obtain
\begin{eqnarray}
  \matNull &=& \matWnull^T \matMnull
  + \half \matWnull^T \matWnull (\matBnull + \matBnull^T)\\
  \label{eq_multiplier_first}
  \matNull &=& \matWnull^T \matMnull
  + \half \matOmega (\matBnull + \matBnull^T)\\
  \label{eq_half_B_BT_first}
  \half (\matBnull + \matBnull^T) &=& -\matOmega^{-1} \matWnull^T \matMnull.
\end{eqnarray}
If we insert this into (\ref{eq_fp_pca_mat}), we get the
fixed-point equation
\begin{equation}\label{eq_fp_pca_first_solution}
  \matMnull - \matWnull \matOmega^{-1} \matWnull^T \matMnull = \matNull.
\end{equation}
We already see at this point that we can factor out $\matMnull$,
leading to $(\matI_n - \matWnull \matOmega^{-1} \matWnull^T) \matMnull
= \matNull$. If we, for example, insert $\matMnull = \matC \matWnull
\matTheta$ for special case TwJ (section \ref{sec_special_cases}), we
can eliminate $\matTheta$ from the fixed-point equation. However,
$\matTheta$ was introduced by \cite{nn_Xu93} to break the symmetry of
subspace rules and thus produce PCA rules. So we expect that fixed
points derived from (\ref{eq_fp_pca_first_solution}) are weight
vectors spanning the same subspace as $m$ of the eigenvectors.

From the generic fixed-point equation (\ref{eq_fp_pca_first_solution}),
we derive specific equations for different special cases from section
\ref{sec_special_cases} (adding ``1'' to the specifier to indicate the
first variant):
\begin{eqnarray}
  \mbox{T1: } &&
  \matC \matWnull - \matWnull \matWnull^T \matC \matWnull
  = \matNull\\
  \mbox{TwJ1: } &&
  \matC \matWnull \matTheta - \matWnull \matWnull^T \matC \matWnull \matTheta
  = \matNull\\
  \mbox{TwC1: } &&
  \matC \matWnull - \matWnull \matOmega^{-1} \matWnull^T \matC \matWnull
  = \matNull\\
  \mbox{N1: } &&
  \matC \matWnull \matDnull - \matWnull \matWnull^T \matC \matWnull \matDnull
  = \matNull
  \;\;\mbox{with}\;\;
  \matDnull = \Diag{j=1}{m}\{\vecwnull_j^T \matC \vecwnull_j\}.
\end{eqnarray}

\subsection{Second variant}\label{sec_second_variant}
%--------------------------

In the second variant we exploit the obvious symmetry of $\matB +
\matB^T$ which turns (\ref{eq_half_B_BT_first}) into
\begin{equation}\label{eq_half_B_BT_second}
  \half (\matBnull + \matBnull^T) = -\matMnull^T \matWnull \matOmega^{-1}.
\end{equation}
If we insert this into (\ref{eq_fp_pca_mat}), we get the fixed-point
equation
\begin{equation}\label{eq_fp_pca_second_solution}
  \matMnull - \matWnull \matMnull^T \matWnull \matOmega^{-1} = \matNull.
\end{equation}
This equation seems to be more ``interesting'' than
(\ref{eq_fp_pca_first_solution}), since $\matMnull$ appears transposed
and surrounded by other terms. Note that also the second variant is
based on the non-equivalent transformation.

This leads to the following special cases from section
\ref{sec_special_cases} (adding ``2'' to the specifier to indicate the
second variant):
\begin{eqnarray}
  \mbox{T2: } &&
  \matC \matWnull - \matWnull \matWnull^T \matC \matWnull
  = \matNull\\
  \mbox{TwJ2: } &&
  \matC \matWnull \matTheta - \matWnull \matTheta \matWnull^T \matC \matWnull
  = \matNull\\
  \mbox{TwC2: } &&
  \matC \matWnull - \matWnull \matWnull^T \matC \matWnull \matOmega^{-1}
  = \matNull\\
  \mbox{N2: } &&
  \matC \matWnull \matDnull - \matWnull \matDnull \matWnull^T \matC \matWnull
  = \matNull
  \;\;\mbox{with}\;\;
  \matDnull = \Diag{j=1}{m}\{\vecwnull_j^T \matC \vecwnull_j\}.
\end{eqnarray}
We see that the equations of case T1 and T2 coincide. In the
following, we refer to this equation as T.

Note that the equations used to determine the Lagrange multipliers
(\ref{eq_half_B_BT_first}, \ref{eq_half_B_BT_second}) can now be
re-inserted into the modified objective function (\ref{eq_Jstar_gen})
which allows us to derive the ``long form'' of the learning rules (see
section \ref{sec_derivation_long}).

\subsection{Discussion}\label{sec_variants_discussion}
%----------------------

Currently we cannot explain why we obtain two different sets of
fixed-point equations, i.e. the first and second variant above, and
why the first variant leads to ``uninteresting'' solutions (which is
confirmed below in section \ref{sec_fp}). It is also unclear whether
this results from the non-equivalent transformation or from some other
property. It seems that the symmetry property of the Lagrange
multipliers has to be explicitly utilized --- which is the case for
the second variant but not the first --- to arrive at ``interesting''
solutions.

%#############################################################################
%#############################################################################
%#############################################################################

\section{Overall Fixed-Point Analysis}\label{sec_overall_fp}
%=====================================

The following overall fixed-point analysis is taken in modified form
from \citet[p.374]{nn_Oja92}. It reveals whether the fixed points are
true eigenvectors of the covariance matrix. Starting point is equation
(\ref{eq_multiplier_first}) where we replace $\matBnull^* \coloneqq
\half (\matBnull + \matBnull^T)$:
\begin{equation}\label{eq_overall_first}
  \matWnull^T \matMnull = -\matOmega \matBnull^*
\end{equation}
Note that this analysis is affected by the non-equivalent
transformation (see sections \ref{sec_first_variant} and
\ref{sec_second_variant}). It is presently unclear how this affects
the statements derived below.

Please also note that the overall fixed-point analysis explicitly
makes use of constraint (\ref{eq_fp_pca_mat_constraint}). However, the
learning rules exhibit additional fixed points which violate this
constraint (e.g. $\matWnull = \matNull$). The implications of this
difference are presently unclear.

\subsection{Fixed-Point Analysis of Special Case T}
%--------------------------------------------------

For special case T1, we have $\matWnull^T \matWnull = \matOmega =
\matI_m$ and $\matMnull = \matC \matWnull$ which turns
(\ref{eq_overall_first}) into
\begin{equation}
  \matWnull^T \matC \matWnull = -\matBnull^*.
\end{equation}
Transposing this equation leads to the same left-hand side and
$-\matBnull^{*T}$ on the right-hand side. Except of $\matBnull^* =
\matBnull^{*T}$ (which is obvious), we cannot derive further
statements on $\matBnull^*$, particularly not that it is diagonal (as
in the other cases which lead to PCA rules).

\subsection{Fixed-Point Analysis of Special Case TwJ}
%----------------------------------------------------

For special case TwJ, we have
\begin{equation}
  \matWnull^T \matWnull = \matOmega = \matI_m,\quad
  \matMnull = \matC \matWnull \matTheta
\end{equation}
which turns (\ref{eq_overall_first}) into
\begin{equation}
  \matWnull^T  \matC \matWnull = -\matBnull^* \matTheta^{-1}.
\end{equation}
Transposing this equation leads to the same left-hand side, so we can
conclude for the right-hand sides that $\matBnull^* \matTheta^{-1} =
\matTheta^{-1} \matBnull^{*T}$. Since $\matBnull^* = \matBnull^{*T}$,
we get $\matBnull^* \matTheta^{-1} = \matTheta^{-1}
\matBnull^*$. According to \lemmaref{\ref{lemma_commute_diag}},
$\matBnull^*$ is a diagonal matrix since $\matTheta^{-1}$ is diagonal
and has pairwise different entries.

For special case TwJ, equation (\ref{eq_fp_pca_mat}) becomes
\begin{equation}
\matC \matWnull = -\matWnull \matBnull^* \matTheta^{-1}.
\end{equation}
Since we now know that $\matBnull^* \matTheta^{-1}$ is diagonal, this
equation has solutions in the eigenvectors and eigenvalues of $\matC$:
The matrix $\matWnull$ contains $m$ distinct eigenvectors from
$\matV$, and $-\matBnull^* \matTheta^{-1}$ contains the corresponding
eigenvalues from $\matLambda$ on the diagonal. Therefore special case
TwJ should be related to true PCA learning rules.

\subsection{Fixed-Point Analysis of Special Case TwC}
%----------------------------------------------------

For special case TwC, we have
\begin{equation}
  \matWnull^T \matWnull = \matOmega = \Diag{j=1}{m}\{\varOmega_j\},\quad
  \matMnull = \matC \matWnull
\end{equation}
which turns (\ref{eq_overall_first}) into
\begin{equation}
  \matWnull^T \matC \matWnull = -\matOmega \matBnull^*.
\end{equation}
As for case TwJ, we can conclude that also special case TwC should be
related to true PCA learning rules.

\subsection{Fixed-Point Analysis of Special Case N}
%--------------------------------------------------

For special case N, we have
\begin{equation}
  \matWnull^T \matWnull = \matOmega = \matI_m,\quad
  \matMnull = \matC \matWnull \matDnull,\quad
  \matDnull = \Diag{j=1}{m}\{\vecwnull_j^T \matC \vecwnull_j\}
\end{equation}
which turns (\ref{eq_overall_first}) into
\begin{equation}
  \matWnull^T \matC \matWnull \matDnull = -\matBnull^*.
\end{equation}
We can invert $\matDnull$ since according to the Rayleigh-Ritz Theorem
\cite[][sec.~4.2.2]{nn_Horn99}, assumption (\ref{eq_order_lambda}),
and the weight length constraint above, all diagonal entries must be
strictly positive. We get
\begin{equation}
  \matWnull^T \matC \matWnull = -\matBnull^* \matDnull^{-1}.
\end{equation}
In a similar way as for special case TwJ, transposition leads to
$\matBnull^* \matDnull^{-1} = \matDnull^{-1} \matBnull^*$. If
$\matDnull$ contains pairwise different entries, we can draw the same
conclusion as for special case TwJ: We expect true PCA learning
rules. However, $\matDnull$ does not necessarily contain pairwise
different entries, and in this case we cannot conclude that
$\matBnull^*$ is diagonal. The treatment of this case in \ref{sec_N2}
shows that additional fixed points actually exist, and only the
analysis of the behavior of the constrained objective function in
section \ref{sec_N} reveals that they are not maxima. So probably the
overall fixed-point analysis in this section cannot produce further
insights for special case N.

%#############################################################################
%#############################################################################
%#############################################################################

\section{Fixed Points of Objective Functions}\label{sec_fp}
%============================================

\subsection{Introduction}
%------------------------

In this part, we determine the fixed points of special cases T, TwJ1,
TwC1, N1, TwJ2, TwC2, and N2. We analyze whether the fixed points
could contain spurious solutions which are possibly introduced by the
non-equivalent transformation described in section
\ref{sec_first_variant}.\footnote{For this analysis we have to
  distinguish between the exploration of the space of solutions and
  the stability of the fixed points. While we may see a maximum in the
  space of solutions, the same point may be a saddle or a minimum when
  we leave the space of solutions and consider the entire Stiefel
  manifold.}

\subsection{Special Case T}
%--------------------------

\subsubsection{Solution of Special Case T}
%'''''''''''''''''''''''''''''''''''''''''

The fixed-point equation of special case T
\begin{equation}\label{eq_T}
\matC \matWnull - \matWnull \matWnull^T \matC \matWnull = \matNull
\end{equation}
can be transformed into the space of the eigenvectors by applying
(\ref{eq_WVA}) and expressing $\matC$ by its eigenvalues and
eigenvectors using (\ref{eq_C_spectral}):
\begin{eqnarray}
  \matNull
  &=&
  \matC \matWnull
  - \matWnull \matWnull^T \matC \matWnull\\
  \matNull
  &=&
  \matC \matV \matAnull
  - \matV \matAnull \matAnull^T \matV^T \matC \matV \matAnull\\
  \matNull
  &=&
  \matV\matLambda \matAnull
  - \matV \matAnull \matAnull^T \matLambda \matAnull\\
  \matNull
  &=&
  \matLambda \matAnull
  - \matAnull \matAnull^T \matLambda \matAnull.
  \label{eq_FP_A}
\end{eqnarray}
The $n \times m$ matrix $\matAnull$ can be expressed by a singular
value decomposition (see e.g. \citet[sec.~2.5.2]{nn_Golub96},
\citet[sec.~3.3]{nn_Diamantaras96})
\begin{equation}\label{eq_svd}
\matAnull = \matQ \pmat{\matDelta_m\\\matNull_{n-m,m}}\matR_m^T
\end{equation}
where $\matQ$ is an orthogonal $n \times n$ matrix, $\matDelta_m$ is a
diagonal matrix with elements $\Delta_1 \geq \ldots \geq \Delta_\mu >
0$ and $\Delta_{\mu+1} = \ldots = \Delta_m = 0$, and $\matR_m$ is an
orthogonal $m \times m$ matrix.

Before we proceed, we look at a special case of (\ref{eq_svd}). For
$\matDelta_m = \matI_m$, the matrix $\matAnull$ (and thus also
$\matWnull$) is semi-orthogonal:
\begin{eqnarray}
  \matAnull^T \matAnull
  &=&
  \matR_m \pmat{\matI_m & \matNull_{m,n-m}}
  \underbrace{\matQ^T \matQ}_{\matI_m}
  \pmat{\matI_m\\\matNull_{n-m,m}}\matR_m^T\\
  &=&
  \matR_m
  \underbrace{\pmat{\matI_m & \matNull_{m,n-m}} \pmat{\matI_m\\\matNull_{n-m,m}}}_{\matI_m}
  \matR_m^T\\
  &=&
  \matR_m^T \matR_m = \matI_m.
\end{eqnarray}
In this case, $\matAnull$ can be expressed by 
\begin{equation}
  \label{eq_semiortho}
  \matAnull = \matQ' \pmat{\matI_m\\\matNull_{n-m,m}}
\end{equation}
where $\matQ'$ is an orthogonal $n \times n$ matrix. This is obvious,
but can also be derived as follows:
\begin{eqnarray}
  \matAnull
  &=& \matQ \pmat{\matDelta_m\\\matNull_{n-m,m}}\matR_m^T
  = \matQ \pmat{\matI_m\\\matNull_{n-m,m}} \matR_m^T\\
  &=& \pmat{\matQ_L & \matQ_R} \pmat{\matI_m\\\matNull_{n-m,m}} \matR_m^T
  = \matQ_L \matR_m^T\\
  &=& \matQ'_L
  = \pmat{\matQ'_L & \matQ'_R} \pmat{\matI_m\\\matNull_{n-m,m}}
  = \matQ' \pmat{\matI_m\\\matNull_{n-m,m}}.
\end{eqnarray}
Note that $\matQ_L$, $\matQ_R$, $\matQ'_L$, and $\matQ'_R$ are
semi-orthogonal. For the special cases other than T, we will use
(\ref{eq_semiortho}) for a simplified analysis where we assume that
the fixed points are semi-orthogonal matrices lying on the constraint
$\matWnull^T\matW = \matI$ (with the exception of
$\matWnull^T\matWnull = \matOmega$ for special case TwC). Note,
however, that all fixed-point equations have at least the solution
$\matWnull = \matNull$ which is not semi-orthogonal. So far we didn't
succeed in finding other solutions which are not semi-orthogonal for
these special cases. This treatment is presently only available for
the simplest special case T (see below).

We return to special case T. If we insert the expression
(\ref{eq_svd}) of $\matAnull$ into the fixed-point equation
(\ref{eq_FP_A}), we get
\begin{eqnarray}
  \matNull
  &=&
  \matLambda \matAnull
  - \matAnull \matAnull^T \matLambda \matAnull\\
  \matNull
  &=&
  \matLambda \matQ \pmat{\matDelta_m\\ \matNull_{n-m,m}} \matR_m^T\\
  &-&
  \matQ \pmat{\matDelta_m\\ \matNull_{n-m,m}}
  \matR_m^T \matR_m
  \pmat{\matDelta_m & \matNull_{m,n-m}} \matQ^T
  \matLambda
  \matQ \pmat{\matDelta_m\\ \matNull_{n-m,m}} \matR_m^T\\
  \matNull
  &=&
  \matQ^T \matLambda \matQ \pmat{\matDelta_m\\ \matNull_{n-m,m}}
  - \pmat{\matDelta^2_m & \matNull_{m,n-m}\\ \matNull_{n-m,m} & \matNull_{n-m,n-m}}
  \matQ^T \matLambda \matQ \pmat{\matDelta_m\\ \matNull_{n-m,m}}
  \label{eq_FP_Q}
\end{eqnarray}
where we right-multiplied by $\matR_m$ and left-multiplied by $\matQ^T$
in the last step. If we write
\begin{equation}\label{eq_M_def}
\matM := \matQ^T \matLambda \matQ = \pmat{\matS & \matT^T\\ \matT & \matU}
\end{equation}
where $\matS$ is an $m \times m$ matrix (and the sizes of $\matT$ and
$\matU$ are chosen accordingly), equation (\ref{eq_FP_Q}) turns into
\begin{eqnarray}
  \pmat{\matS\matDelta_m\\ \matT\matDelta_m}
  - \pmat{\matDelta^2_m\matS\matDelta_m\\ \matNull}
  = \pmat{\matNull\\ \matNull}.
\end{eqnarray}
We analyze the upper matrix equation. For column $i$ and with
$\Delta_i = (\matDelta_m)_{ii}$, we obtain
\begin{equation}
\vecs_i \Delta_i - \matDelta^2_m \vecs_i \Delta_i = \vecnull.
\end{equation}
If $\Delta_i = 0$ (last $m-\mu$ elements, see equation
(\ref{eq_svd})), column $\vecs_i$ can be chosen freely. For $\Delta_i
\neq 0$, we look at element $(i, i)$ of $\matS$. After dividing by
$\Delta_i$, we get
\begin{equation}
(1 - \Delta^2_i) s_{ii} = 0.
\end{equation}
If $\Delta_i \neq 1$ (first $\mu$ elements, see equation
(\ref{eq_svd}); moreover $\Delta_i > 0$ assumed by the SVD), we would
get $s_{ii} = 0$. This is a contradiction since we know from
\lemmaref{\ref{lemma_RTDR_diag}} and assumption
(\ref{eq_order_lambda}) that $\lambda_1 \geq
(\matQ^T\matLambda\matQ)_{ii} \geq \lambda_n > 0$. We conclude that in
this case $\Delta_i = 1$. Therefore the first $\mu \leq m$ diagonal
elements of $\matDelta_m$ are $1$, all other diagonal elements are
zero, which leads to
\begin{eqnarray}
  \matDelta_m
  &=& \pmat{
      \matI_\mu          & \matNull_{\mu,m-\mu}\\
      \matNull_{m-\mu,\mu} & \matNull_{m-\mu,m-\mu}}\\
  \pmat{\matDelta_m\\ \matNull_{n-m,m}}
  &=& \pmat{
      \matI_\mu          & \matNull_{\mu,m-\mu}\\
      \matNull_{n-\mu,\mu} & \matNull_{n-\mu,m-\mu}}\\
  \pmat{
    \matDelta^2_m   & \matNull_{m,n-m}\\
    \matNull_{n-m,m} & \matNull_{n-m,n-m}}
  &=&
  \pmat{
    \matI_\mu          & \matNull_{\mu,n-\mu}\\
    \matNull_{n-\mu,\mu} & \matNull_{n-\mu,n-\mu}}.
\end{eqnarray}
We continue with (\ref{eq_FP_Q}). We omit the last $m-\mu$ zero
columns and obtain
\begin{equation}
  \matNull
  =
  \matQ^T \matLambda \matQ \pmat{\matI_\mu\\ \matNull_{n-\mu,\mu}}
  - \pmat{
    \matI_\mu          & \matNull_{\mu,n-\mu}\\
    \matNull_{n-\mu,\mu} & \matNull_{n-\mu,n-\mu}}
  \matQ^T \matLambda \matQ \pmat{\matI_\mu\\ \matNull_{n-\mu,\mu}}.
  \label{eq_FP_Q_mod}
\end{equation}
If we write
\begin{equation}\label{eq_M_def_mod}
\matM := \matQ^T \matLambda \matQ = \pmat{\matS & \matT^T\\ \matT & \matU}
\end{equation}
where $\matS$ is now a $\mu \times \mu$ matrix (and the sizes of
$\matT$ and $\matU$ are chosen accordingly), equation
(\ref{eq_FP_Q_mod}) turns into
\begin{eqnarray}
  \pmat{\matS\\ \matT}
  - \pmat{\matS\\ \matNull}
  = \pmat{\matNull\\ \matNull}.
\end{eqnarray}
which leads to $\matT = \matNull$, thus $\matM$ has a block-diagonal
shape. Equation (\ref{eq_M_def_mod}) is a similarity transformation of
$\matLambda$ into $\matM$, so $\matM$ also has the (pairwise
different) eigenvalues $\lambda_1,\ldots,\lambda_n$, albeit in
different order (expressed by their order in the matrix $\matLambda^*$
used in the spectral decomposition below). Applying
\lemmaref{\ref{lemma_blockdiag_evec}} in reverse direction, we can
{\em construct} any block-diagonal matrix $\matM$ from
\begin{equation}\label{eq_matM_decomp}
  \matM = \matE \matLambda^* \matE^T\;\mbox{with}\;
  \matE = \pmat{\matX & \matNull\\ \matNull & \matY}
\end{equation}
where $\matX$ (dimension $\mu$) and $\matY$ (dimension $n - \mu$) are
orthogonal matrices forming the matrix of eigenvectors $\matE$ of
$\matM$. In the construction of $\matM$, the first $\mu$ eigenvalues
on the diagonal of $\matLambda^*$ are assigned to $\matS$, the
remaining $n - \mu$ eigenvalues to $\matU$. In our case,
$\matLambda^*$ can be any permutation of $\matLambda$, thus
$\matLambda^* = \matP^T \matLambda \matP$ where $\matP$ is a
permutation matrix. We obtain $\matM = \matE \matP^T \matLambda \matP
\matE^T$, and with $\matM = \matQ^T \matLambda \matQ$ from equation
(\ref{eq_M_def}) we can conclude that
\begin{equation}\label{eq_matQ_T}
  \matQ = \matXi \matP \matE^T
\end{equation}
according to \lemmaref{\ref{lemma_ortho_diag_ortho}}.

We summarize
\begin{eqnarray}\label{eq_matAnull_T}
  \matAnull
  &=&
  \matQ
  \pmat{
    \matI_\mu          & \matNull_{\mu,m-\mu}\\
    \matNull_{n-\mu,\mu} & \matNull_{n-\mu,m-\mu}}
  \matR_m^T\\
  &=&
  \matXi \matP \matE^T
  \pmat{
    \matI_\mu          & \matNull_{\mu,m-\mu}\\
    \matNull_{n-\mu,\mu} & \matNull_{n-\mu,m-\mu}}
  \matR_m^T\\
  &=&
  \matXi \matP
  \pmat{
    \matX^T & \matNull\\
    \matNull & \matY^T}
  \pmat{
    \matI_\mu          & \matNull_{\mu,m-\mu}\\
    \matNull_{n-\mu,\mu} & \matNull_{n-\mu,m-\mu}}
  \matR_m^T\\
  &=&
  \label{eq_matAnull_T_res}
  \matXi \matP
  \pmat{
    \matR_\mu^T        & \matNull_{\mu,m-\mu}\\
    \matNull_{n-\mu,\mu} & \matNull_{n-\mu,m-\mu}}
  \matR_m^T
\end{eqnarray}
where $\matR_\mu^T = \matX^T$ is an orthogonal matrix.

With (\ref{eq_WVA}) we obtain
\begin{equation}
  \matWnull
  =
  \matV \matXi \matP
  \pmat{
    \matR_\mu^T        & \matNull_{\mu,m-\mu}\\
    \matNull_{n-\mu,\mu} & \matNull_{n-\mu,m-\mu}}
  \matR_m^T  
\end{equation}
If we denote by $\matV' = \matV \matXi$ the matrix of eigenvectors
with the signs of the eigenvectors in its columns arbitrarily chosen
(\lemmaref{\ref{lemma_ev_xi}}), we obtain
\begin{equation}\label{eq_fp_T}
  \matWnull
  =
  \matV' \matP 
  \pmat{
    \matR_\mu^T        & \matNull_{\mu,m-\mu}\\
    \matNull_{n-\mu,\mu} & \matNull_{n-\mu,m-\mu}}
  \matR_m^T.
\end{equation}
We get the result that the fixed points of Oja's subspace rule
\cite[]{nn_Oja89}
\begin{equation}\label{eq_oja_subspace}
\tau \matWdot = \matC \matW - \matW \matW^T \matC \matW
\end{equation}
are the following:
\begin{itemize}
\item For $\mu = 0$, we get $\matWnull = \matNull$ which is also an
  obvious fixed point of (\ref{eq_oja_subspace}).
\item For $0 < \mu < m$, we get a solution where, in an intermediate
  matrix, $\mu$ columns are arbitrary selections of $\mu$ eigenvectors
  which can then be arbitrarily rotated within the subspace they span,
  and the remaining $m-\mu$ columns are zero vectors. This
  intermediate matrix can again be arbitrarily rotated in the subspace
  spanned by it, finally giving $\matWnull$.
\item For $\mu = m$, we get the solution
  \begin{equation}
    \label{eq_fp_T_full}
  \matWnull
  =
  \matV' \matP 
  \pmat{
    \matR'^T_m\\
    \matNull_{n-m,m}}
  \matR_m^T  
  =
  \matV' \matP
    \pmat{
    \matR''^T_m\\
    \matNull_{n-m,m}},
  \end{equation}
  i.e. arbitrary selections of $m$ eigenvectors which can then be
  arbitrarily rotated within the subspace they span.
\end{itemize}
We also see that the solution (\ref{eq_fp_T_full}) (but not the
other cases) fulfills the constraint of special case T. Since
$\matR''_m$, $\matV'$, and $\matP$ are orthogonal we get
\begin{equation}
  \matWnull^T \matWnull
  =
  \pmat{\matR''_m & \matNull_{m,n-m}} \matP^T \matV'^T
  \matV' \matP \pmat{\matR''^T_m\\ \matNull_{n-m,m}}
  =
  \matI_m.
\end{equation}
It is interesting to observe that for $\mu < m$, the constraint
$\matWnull^T\matWnull = \matI_m$ is not fulfilled since the matrix
$\matWnull$ does not have the maximal rank. This violation may have
been introduced by the non-equivalent transformation used to eliminate
the Lagrange multipliers.

\subsubsection{Analysis of Special Case T}
%'''''''''''''''''''''''''''''''''''''''''

We only analyze the case $\mu = m$. According to equation
(\ref{eq_fp_T_full}), the putative fixed points of equation
(\ref{eq_T}) are given by
\begin{equation}
  \matWnull = \matV \matP \pmat{\matR\\ \matNull}
  \label{eq_fp_oja_mod}
\end{equation}
where we replaced $\matV \coloneqq \matV'$, $\matR \coloneqq
\matR''^T_m$, $\matNull \coloneqq \matNull_{n-m,m}$, and below also
$\matNull^T \coloneqq \matNull_{m,n-m}$. Since the non-equivalent
transformation (section \ref{sec_first_variant} and
\ref{sec_second_variant}) may introduce spurious solutions, we have to
check whether all solutions of (\ref{eq_fp_oja_mod}) are actually
valid. Validity can be checked by inserting the solution into the
objective function:
\begin{eqnarray}
  J(\matWnull)
  &=&
  \half \tr\left\{\matWnull^T\matC\matWnull\right\}\\
  &=&
  \half \tr\left\{\matWnull^T\matV\matLambda\matV^T\matWnull\right\}\\
  &=& \half \tr\left\{
  \pmat{\matR^T & \matNull^T} \matP^T \matV^T \matV \matLambda
  \matV^T \matV \matP \pmat{\matR\\\matNull}
  \right\}\\
  &=& \half \tr\left\{
  \pmat{\matR^T & \matNull^T} \matP^T \matLambda \matP \pmat{\matR\\\matNull}
  \right\}\\
  &=& \half \tr\left\{
  \pmat{\matR^T & \matNull^T} \matLambda^* \pmat{\matR\\\matNull}
  \right\}\\
  &=& \half \tr\left\{
  \pmat{\matR^T & \matNull^T}
  \pmat{\hat{\matLambda}^* & \matNull\\ \matNull & \check{\matLambda}^*}
  \pmat{\matR\\\matNull}
  \right\}\\
  &=& \half \tr\left\{
  \matR^T
  \hat{\matLambda}^*
  \matR
  \right\}\\
  &=& \half \tr\left\{
  \matR
  \matR^T
  \hat{\matLambda}^*
  \right\}\\
  &=& \half \tr\{
  \hat{\matLambda}^*
  \}
\end{eqnarray}
We see that $J(\matWnull)$ only depends on $\matP$ (implicitly
contained in $\hat{\matLambda}^*$), but not on $\matR$. Different
choices of $\matP$ produce different isolated solutions
$\matWnull$. In each of these solutions, changes of $\matR$ have no
effect on $J(\matWnull)$, confirming that any choice of $\matR$ leads
to a valid solution. Therefore the solution (\ref{eq_fp_T_full})
does not include spurious solutions.

\subsection{Special Case TwJ1}
%-----------------------------

\subsubsection{Solution of Special Case TwJ1}
%''''''''''''''''''''''''''''''''''''''''''''

The fixed-point equation of special case TwJ1
\begin{equation}\label{eq_TwJ1}
\matC \matWnull \matTheta - \matWnull \matWnull^T \matC \matWnull
\matTheta = \matNull
\end{equation}
can be transformed into the fixed-point equation of special case T
(\ref{eq_T}) by right-multiplying by $\matTheta^{-1}$ (assuming all
$\theta_j$ are non-zero):
\begin{equation}
\matC \matWnull - \matWnull \matWnull^T \matC \matWnull = \matNull.
\end{equation}
Therefore, equation (\ref{eq_TwJ1}) of special case TwJ1 has the same
solution (\ref{eq_fp_T_full}) for case $\mu = m$ as equation
(\ref{eq_T}) of special case T:
\begin{equation}\label{eq_fp_TwJ1}
\matWnull = \matV \matP \pmat{\matR\\ \matNull_{n-m,m}}.
\end{equation}
The constraint of special case TwJ1, $\matWnull^T \matWnull =
\matI_m$, is fulfilled which an be shown as for special case T (case
$\mu = m$). This not surprising, since for $\mu = m$, the ansatz for
$\matAnull$ is a semi-orthogonal matrix, see equation
(\ref{eq_semiortho}).

\subsubsection{Analysis of Special Case TwJ1}
%''''''''''''''''''''''''''''''''''''''''''''

Similar to special case T, we insert the putative solution
(\ref{eq_fp_TwJ1}) of special case TwJ1 into the corresponding, now
weighted objective function:
\begin{eqnarray}
  J(\matWnull)
  &=&
  \half \tr\left\{\matWnull^T\matC\matWnull\matTheta\right\}\\
  &=&
  \half \tr\left\{\matWnull^T\matV\matLambda\matV^T\matWnull\matTheta\right\}\\
  &=& \half \tr\left\{
  \pmat{\matR^T & \matNull^T} \matP^T \matV^T \matV \matLambda
  \matV^T \matV \matP \pmat{\matR\\\matNull}
  \matTheta\right\}\\
  &=& \half \tr\left\{
  \pmat{\matR^T & \matNull^T} \matP^T \matLambda \matP \pmat{\matR\\\matNull}
  \matTheta\right\}\\
  &=& \half \tr\left\{
  \pmat{\matR^T & \matNull^T} \matLambda^* \pmat{\matR\\\matNull}
  \matTheta\right\}\\
  &=& \half \tr\left\{
  \pmat{\matR^T & \matNull^T}
  \pmat{\hat{\matLambda}^* & \matNull\\ \matNull & \check{\matLambda}^*}
  \pmat{\matR\\\matNull}
  \matTheta\right\}\\
  &=& \half \tr\left\{
  \matR^T
  \hat{\matLambda}^*
  \matR
  \matTheta\right\}\\
  &=& \half
  \summe{i=1}{m}
  (\matR^T\hat{\matLambda}^*\matR)_{ii} \theta_i
\end{eqnarray}
where we applied \lemmaref{\ref{lemma_tr_AD}} to switch from trace to
sums. According to \lemmaref{\ref{lemma_RTDR_ii_b_i}}, the last
expression is maximal for $\matR = \matXi$. We see that, in this case,
it is not possible to chose an arbitrary $\matR$ in each of the
discrete solutions determined by $\matP$. The objective function is
maximized for a specific choice of $\matR$, namely $\matR =
\matXi$. Other solutions of $\matR$ are spurious solutions since
$\matR$ can be modified such that the
objective function becomes larger.

\subsection{Special Case TwC1}
%-----------------------------

\subsubsection{Solution of Special Case TwC1}
%''''''''''''''''''''''''''''''''''''''''''''

The fixed-point equation of special case TwC1
\begin{equation}\label{eq_TwC1}
  \matC \matWnull - \matWnull \matOmega^{-1} \matWnull^T \matC \matWnull
  = \matNull
\end{equation}
can be treated in a similar way as for special case T. However,
instead of expressing $\matAnull$ by equation (\ref{eq_semiortho}),
the weighted constraint leads to
\begin{equation}\label{eq_matAnull_TwC1}
  \matAnull
  = \matQ \pmat{\matOmega^\half\\\matNull_{n-m,m}}
  = \matQ \pmat{\matI_m\\\matNull_{n-m,m}} \matOmega^\half.
\end{equation}
That the weighted constrained is fulfilled can be shown by
\begin{equation}
  \matA^T\matA =
  \matOmega^\half \pmat{\matI_m & \matNull} \matQ
  \matQ^T \pmat{\matI_m\\ \matNull} \matOmega^\half
  = \matOmega.
\end{equation}
We proceed with
\begin{eqnarray}
  \matNull
  &=&
  \matC \matWnull
  - \matWnull \matOmega^{-1} \matWnull^T \matC \matWnull\\
  \matNull
  &=&
  \matC \matV \matAnull
  - \matV \matAnull \matOmega^{-1} \matAnull^T \matV^T \matC \matV \matAnull\\
  \matNull
  &=&
  \matLambda \matAnull
  - \matAnull \matOmega^{-1} \matAnull^T \matLambda \matAnull\\
  \matNull
  &=&
  \matLambda \matQ \pmat{\matI_m\\ \matNull} \matOmega^\half
  - \matQ \pmat{\matI_m\\ \matNull}
  \matOmega^\half \matOmega^{-1} \matOmega^\half
  \pmat{\matI_m & \matNull^T}
  \matQ^T \matLambda \matQ \pmat{\matI_m\\ \matNull} \matOmega^\half\\
  \matNull
  &=&
  \label{eq_TwC1_tmp1}
  \matQ^T \matLambda \matQ \pmat{\matI_m\\ \matNull} \matOmega^\half
  - \pmat{\matI_m\\ \matNull} \pmat{\matI_m & \matNull^T}
  \matQ^T \matLambda \matQ \pmat{\matI_m\\ \matNull} \matOmega^\half\\
  \matNull
  &=&
  \pmat{\matS & \matT^T\\\matT & \matU} \pmat{\matI_m\\ \matNull}
  - \pmat{\matI_m\\ \matNull} \pmat{\matI_m & \matNull^T}
  \pmat{\matS & \matT^T\\\matT & \matU} \pmat{\matI_m\\ \matNull}\\
  \pmat{\matNull\\ \matNull}
  &=&
  \pmat{\matS\\ \matT}
  - \pmat{\matS\\ \matNull}
\end{eqnarray}
where we right-multiplied (\ref{eq_TwC1_tmp1}) by
$\matOmega^{-\half}$. We can now continue as in the special case T and
get
\begin{equation}\label{eq_fp_TwC1}
\matWnull = \matV \matP \pmat{\matR\\ \matNull_{n-m,m}} \matOmega^\half.
\end{equation}
That the constraint is fulfilled can be verified by
\begin{equation}
  \matWnull^T \matWnull
  =
  \matOmega^\half \pmat{\matR^T & \matNull_{m,n-m}} \matP^T \matV'^T
  \matV' \matP \pmat{\matR\\ \matNull_{n-m,m}} \matOmega^\half
  =
  \matOmega
\end{equation}
which is not surprising since the ansatz for $\matAnull$ was chosen
this way.

\subsubsection{Analysis of Special Case TwC1}
%''''''''''''''''''''''''''''''''''''''''''''

Similar to special case TwJ1, we insert the putative solution
(\ref{eq_fp_TwC1}) into the corresponding objective function
\begin{eqnarray}
  J(\matWnull)
  &=&
  \half \tr\left\{\matWnull^T\matC\matWnull\right\}\\
  &=&
  \half \tr\left\{\matWnull^T\matV\matLambda\matV^T\matWnull\right\}\\
  &=& \half \tr\left\{
  \matOmega^\half
  \pmat{\matR^T & \matNull^T} \matP^T \matV^T \matV \matLambda
  \matV^T \matV \matP \pmat{\matR\\\matNull}
  \matOmega^\half\right\}\\
  &=& \half \tr\left\{
  \pmat{\matR^T & \matNull^T} \matP^T \matLambda \matP \pmat{\matR\\\matNull}
  \matOmega\right\}\\
  &=& \half \tr\left\{
  \pmat{\matR^T & \matNull^T} \matLambda^* \pmat{\matR\\\matNull}
  \matOmega\right\}\\
  &=& \half \tr\left\{
  \pmat{\matR^T & \matNull^T}
  \pmat{\hat{\matLambda}^* & \matNull\\ \matNull & \check{\matLambda}^*}
  \pmat{\matR\\\matNull}
  \matOmega\right\}\\
  &=& \half \tr\left\{
  \matR^T
  \hat{\matLambda}^*
  \matR
  \matOmega\right\}\\
  &=& \half
  \summe{i=1}{m}
  (\matR^T\hat{\matLambda}^*\matR)_{ii} \varOmega_i
\end{eqnarray}
where we applied \lemmaref{\ref{lemma_tr_AD}} to switch from trace to
sums. As in the case TwJ1, we can apply
\lemmaref{\ref{lemma_RTDR_ii_b_i}} to the last expression which leads
to the conclusion that (\ref{eq_fp_TwC1}) contains spurious solutions.

\subsection{Special Case N1}
%---------------------------

\subsubsection{Solution of Special Case N1}
%''''''''''''''''''''''''''''''''''''''''''

The fixed-point equation of special case N1
\begin{equation}\label{eq_N1}
\matC \matWnull \matDnull - \matWnull \matWnull^T \matC \matWnull \matDnull
= \matNull
\;\;\mbox{with}\;\;
\matDnull = \Diag{j=1}{m}\{\vecwnull_j^T \matC \vecwnull_j\}
\end{equation}
can be transformed into the fixed-point equation of special case T
(\ref{eq_T}) by right-multiplying by $\matDnull^{-1}$ (similar to the
treatment of case TwJ1):
\begin{equation}
\matC \matWnull - \matWnull \matWnull^T \matC \matWnull = \matNull.
\end{equation}
Note that $\matDnull$ can be inverted since all diagonal elements are
strictly positive and thus non-zero. This in turn can be shown by the
Rayleigh-Ritz Theorem \cite[][sec.~4.2.2]{nn_Horn99}: Assuming the
order of eigenvalues (\ref{eq_order_lambda}), it is guaranteed
that
\begin{equation}\label{eq_rayleigh_ritz}
  \lambda_1 \vecwnull_j^T \vecwnull_j
  \geq
  \vecwnull_j^T \matC \vecwnull_j
  \geq
  \lambda_n \vecwnull_j^T \vecwnull_j.
\end{equation}
Note that we assume $\matA$ to be semi-orthogonal according to
equation (\ref{eq_semiortho}), thus $\matWnull^T\matWnull =
\matA^T\matV^T\matV\matA = \matI$ and therefore
$\vecwnull_j^T\vecwnull_j = 1$ for all $j$, which leads to
\begin{equation}\label{eq_rayleigh_ritz_semiortho}
  \lambda_1
  \geq
  \vecwnull_j^T \matC \vecwnull_j
  \geq
  \lambda_n
  > 0
\end{equation}
where (\ref{eq_order_lambda}) was considered. Therefore, equation
(\ref{eq_TwJ1}) of special case N1 has the same solution
(\ref{eq_fp_T_full}) for case $\mu = m$ as equation (\ref{eq_T}) of
special case T:
\begin{equation}\label{eq_fp_N1}
\matWnull = \matV \matP \pmat{\matR\\ \matNull_{n-m,m}}.
\end{equation}
The constraint of special case N1, $\matWnull^T \matWnull = \matI_m$
is fulfilled which can be shown as for special case T (case $\mu = m$,
semi-orthogonal ansatz for $\matAnull$).

\subsubsection{Analysis of Special Case N1}
%''''''''''''''''''''''''''''''''''''''''''
\label{sec_analysis_N1}

We insert the putative solution (\ref{eq_TwJ1}) into the novel
objective function (\ref{eq_objfct_new}) to test for spurious
solutions. Here we need to express the solution for each individual
vector $\vecwnull_j$:
\begin{equation}\label{eq_fp_N1_vec}
\vecwnull_j = \matV \matP \pmat{\vecr_j\\ \vecnull_{n-m}}
\end{equation}
where $\vecr_j$ is column $j$ of $\matR$. (In the following, for
brevity we write $\vecnull_{n-m} = \vecnull$.)
\begin{eqnarray}
  J(\matWnull)
  &=&
  \quarter \summe{j=1}{m}
  (\vecwnull_j^T \matC \vecwnull_j)^2\\
  &=&
  \quarter \summe{j=1}{m}
  (\vecwnull_j^T \matV \matLambda \matV^T \vecwnull_j)^2\\
  &=&
  \quarter \summe{j=1}{m}
  \left[
    \pmat{\vecr_j^T & \vecnull^T} \matP^T \matV^T \matV \matLambda
    \matV^T \matV \matP \pmat{\vecr_j\\\vecnull}
  \right]^2\\
  &=&
  \quarter \summe{j=1}{m}
  \left[
    \pmat{\vecr_j^T & \vecnull^T} \matP^T \matLambda \matP
    \pmat{\vecr_j\\\vecnull}
  \right]^2\\
  &=&
  \quarter \summe{j=1}{m}
  \left[
    \pmat{\vecr_j^T & \vecnull^T}
    \matLambda^*
    \pmat{\vecr_j\\\vecnull}
  \right]^2\\
  &=&
  \quarter \summe{j=1}{m}
  \left[
    \pmat{\vecr_j^T & \vecnull^T} 
    \pmat{\hat{\matLambda}^* & \matNull\\ \matNull & \check{\matLambda}^*}
    \pmat{\vecr_j\\\vecnull}
  \right]^2\\
  &=&
  \quarter \summe{j=1}{m}
  \left(
    \vecr_j^T \hat{\matLambda}^* \vecr_j
  \right)^2\\
  &=&
  \quarter \summe{j=1}{m} \left(\matR^T \hat{\matLambda}^* \matR\right)_{jj}^2
\end{eqnarray}
where we applied \lemmaref{\ref{lemma_RTDR_ii}} in the last
step. According to \lemmaref{\ref{lemma_RTDR_ii_sqr}},
\begin{equation}
\summe{j=1}{m} \left(\matR^T \hat{\matLambda}^* \matR\right)_{jj}^2
\end{equation}
is maximal only if $\matR^T \hat{\matLambda}^* \matR$ is diagonal,
which holds only if $\matR = \matXi \matP'$ where $\matP'$ is a
permutation matrix.

We see that, in this case, it is not possible to chose an arbitrary
$\matR$ in each of the discrete solutions determined by $\matP$. The
objective function is maximized for a specific choice of $\matR$,
namely $\matR = \matXi \matP'$. Other solutions of $\matR$ are
spurious solutions since $\matR$ can be modified such that the
objective function becomes larger.

If we insert the solution $\matR = \matXi \matP'$ into
(\ref{eq_fp_N1}), we get (with matrix sizes indicated)
\begin{eqnarray}
  \matWnull
  &=&
  \matV \matP_n \pmat{\matR_m\\ \matNull}\\
  &=&
  \matV \matP_n \pmat{\matXi_m \matP'_m\\ \matNull}\\
  &=&
  \matV \matP_n \pmat{\matXi_m\\ \matNull} \matP'_m.  
\end{eqnarray}
We form a larger diagonal sign matrix $\hat{\matXi}_n$ which contains
$\matXi_m$
\begin{equation}
\hat{\matXi}_n = \pmat{\matXi_m & \matNull\\ \matNull & \check{\matXi}_{n-m}}
\end{equation}
and continue
\begin{eqnarray}
  \matWnull
  &=&
  \matV \matP_n \hat{\matXi}_n \pmat{\matI_m\\ \matNull} \matP'_m \\
  &=&
  \matV \hat{\matXi}^*_n \matP_n \pmat{\matI_m\\ \matNull} \matP'_m\\
  &=&
  \matV' \matP_n \pmat{\matP'_m\\ \matNull}
\end{eqnarray}
where we applied \lemmaref{\ref{lemma_DP}} to swap permutation and
sign matrices, and integrated the diagonal sign matrix into
$\matV'$. We form a larger permutation matrix $\hat{\matP}'_n$ which
contains $\matP'_m$
\begin{equation}
\hat{\matP}'_n = \pmat{\matP'_m & \matNull\\ \matNull & \check{\matP}'_{n-m}}
\end{equation}
and continue
\begin{eqnarray}
  \matWnull
  &=&
  \matV' \matP_n \hat{\matP}'_n \pmat{\matI_m\\ \matNull}\\
  \label{eq_N1_solution_max}
  &=&
  \matV' \matP''_n \pmat{\matI_m\\ \matNull}.
\end{eqnarray}
We see that this solution only comprises an arbitrary selection of
eigenvectors (with arbitrary sign) but no rotation, so the solution
(\ref{eq_fp_N1}) of the fixed-point equation contains spurious
solutions.

\subsection{Special Case TwJ2}
%-----------------------------

\subsubsection{Solution of Special Case TwJ2}\label{sec_TwJ2_solution}
%''''''''''''''''''''''''''''''''''''''''''''

The fixed-point equation of special case TwJ2
\begin{equation}\label{eq_TwJ2}
  \matC \matWnull \matTheta - \matWnull \matTheta \matWnull^T \matC \matWnull
  = \matNull
\end{equation}
can be treated in a similar way as for special case T, expressing
$\matAnull$ by equation (\ref{eq_semiortho}):
\begin{eqnarray}
  \matNull
  &=&
  \matC \matWnull \matTheta
  - \matWnull \matTheta \matWnull^T \matC \matWnull\\
  \matNull
  &=&
  \matC \matV \matAnull \matTheta
  - \matV \matAnull \matTheta \matAnull^T \matV^T \matC \matV \matAnull\\
  \matNull
  &=&
  \matV^T \matC \matV \matAnull \matTheta
  - \matV^T \matV \matAnull \matTheta \matAnull^T \matV^T \matC\matV \matAnull\\
  \matNull
  &=&
  \label{eq_TwJ2_intermediate}
  \matLambda \matAnull \matTheta
  - \matAnull \matTheta \matAnull^T \matLambda \matAnull\\
  \matNull
  &=&
  \matLambda \matQ \pmat{\matI_m\\ \matNull} \matTheta
  - \matQ \pmat{\matI_m\\ \matNull}
  \matTheta
  \pmat{\matI_m & \matNull^T}
  \matQ^T \matLambda \matQ \pmat{\matI_m\\ \matNull}\\
  \matNull
  &=&
  \matQ^T \matLambda \matQ \pmat{\matI_m\\ \matNull} \matTheta
  - \pmat{\matI_m\\ \matNull} \matTheta \pmat{\matI_m & \matNull^T}
  \matQ^T \matLambda \matQ \pmat{\matI_m\\ \matNull}\\
  \matNull
  &=&
  \pmat{\matS & \matT^T\\\matT & \matU} \pmat{\matI_m\\ \matNull} \matTheta
  - \pmat{\matI_m\\ \matNull} \matTheta \pmat{\matI_m & \matNull^T}
  \pmat{\matS & \matT^T\\\matT & \matU} \pmat{\matI_m\\ \matNull}\\
  \pmat{\matNull\\ \matNull}
  &=&
  \pmat{\matS \matTheta \\ \matT \matTheta}
  - \pmat{\matTheta \matS\\ \matNull}
\end{eqnarray}
We get $\matT\matTheta = \matNull$ and therefore, since $\matTheta$ is
invertible, $\matT = \matNull$, thus $\matM = \matQ^T \matLambda
\matQ$ is block-diagonal, and
\begin{equation}\label{eq_TwJ2_ThetaS_STheta}
\matTheta \matS = \matS \matTheta.
\end{equation}
According to \lemmaref{\ref{lemma_commute_diag}}, $\matS$ is a
diagonal matrix.

We can express the block-diagonal matrix $\matM$ by its spectral
decomposition (\ref{eq_matM_decomp})
\begin{eqnarray}
  \matM
  &=&
  \matE \matLambda^* \matE^T\\
  &=&
  \pmat{\matX & \matNull\\ \matNull & \matY}
  \pmat{\hat{\matLambda}^* & \matNull\\ \matNull & \check{\matLambda}^*}
  \pmat{\matX^T & \matNull\\ \matNull & \matY^T}\\
  &=&
  \pmat{\matX\hat{\matLambda}^*\matX^T & \matNull\\
    \matNull & \matY\check{\matLambda}^*\matY^T}\\
  &=&
  \pmat{\matS & \matNull\\ \matNull & \matU}
\end{eqnarray}
which leads to
\begin{equation}\label{eq_TwJ2_S_XLXT}
  \matS = \matX\hat{\matLambda}^*\matX^T.
\end{equation}
Since this is an orthogonal transformation, and since $\matS$ is
diagonal, $\matS$ shares the eigenvalues of $\hat{\matLambda}^*$,
i.e. the diagonal elements in $\hat{\matLambda}^*$ also appear as the
diagonal elements of $\matS$, albeit in permuted order:
\begin{equation}
\matS = \matP'^T \hat{\matLambda}^* \matP' = \matX\hat{\matLambda}^*\matX^T.
\end{equation}
According to \lemmaref{\ref{lemma_ortho_diag_ortho}}, we get $\matX =
\matXi' \matP'^T$. Note that the solution for $\matQ$ is the same as
(\ref{eq_matQ_T}) for special case T. We insert this into the solution
for $\matAnull$; see equation (\ref{eq_matAnull_T}):
\begin{eqnarray}
  \matAnull
  &=&
  \matQ \pmat{\matI_m\\ \matNull}\\
  &=&
  \matXi \matP \pmat{\matX^T & \matNull\\ \matNull & \matY^T}
  \pmat{\matI_m\\ \matNull}\\
  &=&\label{eq_T_TwJ2_matAnull}
  \matXi \matP \pmat{\matX^T\\ \matNull}\\
  &=&
  \matXi \matP \pmat{\matP'\matXi'\\ \matNull}.
\end{eqnarray}
We form larger matrices which contain $\matP'$ and $\matXi'$
(temporarily indicating matrix sizes)
\begin{eqnarray}
  \hat{\matXi}'_n
  &=&
  \pmat{\matXi'_m & \matNull\\ \matNull & \check{\matXi}'_{n-m}}\\
  \hat{\matP}'_n
  &=&
  \pmat{\matP'_m & \matNull\\ \matNull & \check{\matP}'_{n-m}}
\end{eqnarray}
and continue with
\begin{eqnarray}
  \matAnull
  &=&
  \matXi \matP \hat{\matP}'_n \hat{\matXi}'_n \pmat{\matI_m\\ \matNull}\\
  &=&
  \matXi \matP'' \hat{\matXi}'_n \pmat{\matI_m\\ \matNull}\\
  &=&
  \matXi \hat{\matXi}'^*_n \matP'' \pmat{\matI_m\\ \matNull}\\
  &=&
  \label{eq_TwJ2_solution_matA}
  \matXi'' \matP'' \pmat{\matI_m\\ \matNull}
\end{eqnarray}
where we fused the permutation matrices, applied \lemmaref{\ref{lemma_DP}}
to swap permutation and sign matrices, and fused the sign matrices. We get
\begin{eqnarray}
  \matWnull
  &=&
  \matV \matXi'' \matP'' \pmat{\matI_m\\ \matNull_{n-m,m}}\\
  \label{eq_fp_TwJ2}
  &=&
  \matV' \matP'' \pmat{\matI_m\\ \matNull_{n-m,m}}.
\end{eqnarray}
This confirms the known result \cite[]{nn_Xu93} that the fixed points
of Xu's rule (15a)
\begin{equation}
\tau \matWdot = \matC \matW \matTheta - \matW \matTheta \matW^T \matC \matW
\end{equation}
are arbitrary selections of the eigenvectors of $\matC$ (with
arbitrary signs).

We also see that the solution (\ref{eq_fp_TwJ2}) fulfills the
constraint of special case TwJ2 (not surprising, since the ansatz for
$\matAnull$ was a semi-orthogonal matrix). Since $\matV'$ and
$\matP''$ are orthogonal we get
\begin{equation}
  \matWnull^T \matWnull
  =
  \pmat{\matI_m & \matNull_{m,n-m}} \matP''^T \matV'^T
  \matV' \matP'' \pmat{\matI_m\\ \matNull_{n-m,m}}
  =
  \matI_m.
\end{equation}

\subsubsection{Test of Solution of Special Case TwJ2}
%''''''''''''''''''''''''''''''''''''''''''''''''''''

To test our solution, we insert the solution for $\matA$
(\ref{eq_TwJ2_solution_matA}) (omitting primes) into
(\ref{eq_TwJ2_intermediate}). We left-multiply by orthogonal matrices
several times:
\begin{eqnarray}
  \matNull
  &=&
  \matLambda \matAnull \matTheta
  - \matAnull \matTheta \matAnull^T \matLambda \matAnull\\
   \matNull
  &=&
   \matLambda
   \matXi \matP \pmat{\matI_m\\\matNull}
   \matTheta
   -
   \matXi \matP \pmat{\matI_m\\\matNull}
   \matTheta
   \pmat{\matI_m&\matNull} \matP^T \matXi^T
   \matLambda
   \matXi \matP \pmat{\matI_m\\\matNull}\\
   \matNull
  &=&
   \matLambda
   \matXi \matP \pmat{\matI_m\\\matNull}
   \matTheta
   -
   \matXi \matP \pmat{\matI_m\\\matNull}
   \matTheta
   \pmat{\matI_m&\matNull} \matP^T
   \matLambda
   \matP \pmat{\matI_m\\\matNull}\\   
   \matNull
  &=&
   \matXi^T \matLambda
   \matXi \matP \pmat{\matI_m\\\matNull}
   \matTheta
   -
   \matXi^T \matXi \matP \pmat{\matI_m\\\matNull}
   \matTheta
   \pmat{\matI_m&\matNull}
   \matLambda^*
   \pmat{\matI_m\\\matNull}\\   
   \matNull
   &=&
   \matP^T \matLambda
   \matP \pmat{\matI_m\\\matNull}
   \matTheta
   -
   \matP^T \matP \pmat{\matI_m\\\matNull}
   \matTheta
   \hat{\matLambda}^*\\   
   \matNull
   &=&
   \matLambda^*
   \pmat{\matI_m\\\matNull}
   \matTheta
   -
   \pmat{\matI_m\\\matNull}
   \matTheta
   \hat{\matLambda}^*\\
   \pmat{\matNull\\\matNull}
   &=&
   \pmat{\hat{\matLambda}^*\matTheta\\\matNull}
   -
   \pmat{\matTheta\hat{\matLambda}^*\\\matNull}
\end{eqnarray}
where we applied \lemmaref{\ref{lemma_mult_INull}} in the last
step. Since diagonal matrices are exchangeable within a product, this
confirms our solution.

\subsubsection{Analysis of Special Case TwJ2}
%''''''''''''''''''''''''''''''''''''''''''''

Since (\ref{eq_fp_TwJ2}) does not contain any parameters which can
change continuously, it is clear that inserting this solution into the
objective function (\ref{eq_objfct}) will confirm that there are no
spurious solutions. For the sake of completeness, we nevertheless
include this step (note that we replaced $\matV \coloneqq \matV'$ and
$\matP \coloneqq \matP''_n$):
\begin{eqnarray}
  J
  &=& \half\tr\{\matWnull^T\matC\matWnull\matTheta\}\\
  &=& \half\tr\left\{
  \pmat{\matI_m & \matNull} \matP^T \matV^T \matC \matV \matP
  \pmat{\matI_m\\ \matNull} \matTheta
  \right\}\\
  &=& \half\tr\left\{
  \pmat{\matI_m & \matNull} \matP^T \matLambda \matP
  \pmat{\matI_m\\ \matNull} \matTheta
  \right\}\\
  &=& \half\tr\left\{
  \pmat{\matI_m & \matNull} \matLambda^*
  \pmat{\matI_m\\ \matNull} \matTheta
  \right\}\\
  &=& \half\tr\left\{
  \pmat{\matI_m & \matNull}
  \pmat{\hat{\matLambda}^* & \matNull\\ \matNull & \check{\matLambda}^*}  
  \pmat{\matI_m\\ \matNull} \matTheta
  \right\}\\
  &=& \half\tr\{\hat{\matLambda}^*\matTheta\}.
\end{eqnarray}
Matrix $\hat{\matLambda}^*$ depends on the permutation $\matP$ which
defines different isolated fixed points. According to
\lemmaref{\ref{lemma_sorted_permute}} and considering the sorting
condition (\ref{eq_order_theta}), the {\em global} maximum would be
achieved if $\hat{\matLambda}^*$ contains the $m$ largest eigenvalues
sorted as in $\matLambda$. Which fixed points are {\em local} maxima
needs to be analyzed separately.

\subsection{Special Case TwC2}
%-----------------------------

\subsubsection{Solution of Special Case TwC2}
%''''''''''''''''''''''''''''''''''''''''''''

The ansatz for $\matAnull$ is the same as for case TwC1, see
(\ref{eq_matAnull_TwC1}), here for brevity with $\matNull \coloneqq
\matNull_{n-m,m}$:
\begin{equation}\label{eq_matAnull_TwC2}
  \matAnull
  = \matQ \pmat{\matI_m\\\matNull} \matOmega^\half.
\end{equation}
We insert this into the fixed-point equation:
\begin{eqnarray}
  \matNull
  &=&
  \matC \matWnull
  -
  \matWnull \matWnull^T \matC \matWnull \matOmega^{-1}\\
  \matNull
  &=&
  \matC \matV \matAnull
  -
  \matV \matAnull \matAnull^T \matV^T \matC \matV \matAnull \matOmega^{-1}\\
  \matNull
  &=&
  \matV \matLambda \matAnull
  -
  \matV \matAnull \matAnull^T \matLambda \matAnull \matOmega^{-1}\\
  \matNull
  &=&
  \matLambda \matAnull
  -
  \matAnull \matAnull^T \matLambda \matAnull \matOmega^{-1}\\
  \matNull
  &=&
  \matLambda \matQ \pmat{\matI_m\\\matNull} \matOmega^\half
  -
  \matQ \pmat{\matI_m\\\matNull} \matOmega^\half
  \matOmega^\half \pmat{\matI_m & \matNull^T} \matQ^T
  \matLambda \matQ \pmat{\matI_m\\\matNull} \matOmega^\half
  \matOmega^{-1}\\
  \matNull
  &=&
  \matLambda \matQ \pmat{\matI_m\\\matNull} \matOmega
  -
  \matQ \pmat{\matI_m\\\matNull} \matOmega \pmat{\matI_m & \matNull^T} \matQ^T
  \matLambda \matQ \pmat{\matI_m\\\matNull}\\
  \matNull
  &=&
  \matQ^T \matLambda \matQ \pmat{\matI_m\\\matNull} \matOmega
  -
  \pmat{\matI_m\\\matNull} \matOmega \pmat{\matI_m & \matNull^T} \matQ^T
  \matLambda \matQ \pmat{\matI_m\\\matNull}.
\end{eqnarray}
From this point on we can proceed as for special case TwJ2. We obtain
the fixed-point solution
\begin{eqnarray}
  \matAnull &=& \matXi \matP \pmat{\matI_m\\ \matNull} \matOmega^\half\\
  \label{eq_fp_TwC2}
  \matWnull &=& \matV \matP \pmat{\matI_m\\ \matNull} \matOmega^\half.
\end{eqnarray}
This confirms the result that Oja's weighted subspace rule
\cite[]{nn_Oja92a,nn_Oja92,nn_Oja92b}
\begin{equation}
\tau \matWdot = \matC \matW - \matW \matW^T \matC \matW \matOmega^{-1} 
\end{equation}
has fixed points at permutations of scaled eigenvectors $\vecv_j$,
with the vector length determined by the coefficients $\varOmega_j^\half$.
The constraint is fulfilled:
\begin{equation}
  \matWnull^T \matWnull
  =
  \matOmega^\half \pmat{\matI_m & \matNull} \matP^T \matV^T
  \matV \matP \pmat{\matI_m\\ \matNull} \matOmega^\half
  =
  \matOmega.
\end{equation}

\subsubsection{Analysis of Special Case TwC2}
%''''''''''''''''''''''''''''''''''''''''''''

The analysis is analogous to special case TwJ2. We insert
(\ref{eq_fp_TwC2}) into (\ref{eq_objfct}). We skip some steps which
are the same as for TwJ2 (except with $\matOmega$ instead of
$\matTheta$):
\begin{eqnarray}
  J
  &=& \half\tr\{\matWnull^T\matC\matWnull\}\\
  &=& \half\tr\left\{
  \matOmega^\half\pmat{\matI_m & \matNull} \matP^T \matV^T \matC \matV \matP
  \pmat{\matI_m\\ \matNull} \matOmega^\half
  \right\}\\
  &=& \half\tr\left\{
  \pmat{\matI_m & \matNull} \matP^T \matLambda \matP
  \pmat{\matI_m\\ \matNull} \matOmega
  \right\}\\
  &=& \half\tr\{\hat{\matLambda}^*\matOmega\}.
\end{eqnarray}
The insights are the same as for special case TwJ2.

\subsection{Special Case N2}\label{sec_N2}
%---------------------------

\subsubsection{Solution of Special Case N2}\label{sec_N2_solution}
%''''''''''''''''''''''''''''''''''''''''''

The fixed-point equation of special case N2
\begin{equation}\label{eq_N2}
  \matC \matWnull \matDnull - \matWnull \matDnull \matWnull^T \matC \matWnull
  = \matNull
  \;\;\mbox{with}\;\;
  \matDnull = \Diag{j=1}{m}\{\vecwnull_j^T \matC \vecwnull_j\}.
\end{equation}
can be treated in a similar way as for special case TwJ2, expressing
$\matAnull$ by equation (\ref{eq_semiortho}):
\begin{eqnarray}
  \matNull
  &=&
  \matC \matWnull \matDnull
  - \matWnull \matDnull \matWnull^T \matC \matWnull\\
  \matNull
  &=&
  \matC \matV \matAnull \matDnull
  - \matV \matAnull \matDnull \matAnull^T \matV^T \matC \matV \matAnull\\
   \matNull
  &=&
  \matV^T \matC \matV \matAnull \matDnull
  - \matV^T\matV \matAnull \matDnull \matAnull^T \matV^T \matC \matV \matAnull\\
  \label{eq_N2_intermediate}
  \matNull
  &=&
  \matLambda \matAnull \matDnull
  - \matAnull \matDnull \matAnull^T \matLambda \matAnull\\
  \matNull
  &=&
  \matLambda \matQ \pmat{\matI_m\\ \matNull} \matDnull
  - \matQ \pmat{\matI_m\\ \matNull}
  \matDnull
  \pmat{\matI_m & \matNull^T}
  \matQ^T \matLambda \matQ \pmat{\matI_m\\ \matNull}\\
  \matNull
  &=&
  \matQ^T \matLambda \matQ \pmat{\matI_m\\ \matNull} \matDnull
  - \pmat{\matI_m\\ \matNull} \matDnull \pmat{\matI_m & \matNull^T}
  \matQ^T \matLambda \matQ \pmat{\matI_m\\ \matNull}\\
  \matNull
  &=&
  \pmat{\matS & \matT^T\\\matT & \matU} \pmat{\matI_m\\ \matNull} \matDnull
  - \pmat{\matI_m\\ \matNull} \matDnull \pmat{\matI_m & \matNull^T}
  \pmat{\matS & \matT^T\\\matT & \matU} \pmat{\matI_m\\ \matNull}\\
  \label{eq_constraint_N2}
  \pmat{\matNull\\ \matNull}
  &=&
  \pmat{\matS \matDnull \\ \matT \matDnull}
  - \pmat{\matDnull \matS\\ \matNull}.
\end{eqnarray}
We obtain
\begin{equation}\label{eq_N2_SD_DS}
\matS\matDnull = \matDnull\matS.
\end{equation}
If all entries of $\matDnull$ are pairwise different, we can take the
same path as for special case TwJ2 (where this property is assumed for
$\matTheta$) starting from (\ref{eq_TwJ2_ThetaS_STheta}). In this case
we can conclude from (\ref{eq_N2_SD_DS}) that $\matS$ is
diagonal. This leads to solution (\ref{eq_fp_TwJ2}).

However, if entries of $\matDnull$ coincide, we have to take a
different path. First, we show that $\matM$ from definition
(\ref{eq_M_def}) is block-diagonal: From equation
(\ref{eq_constraint_N2}) we get $\matT\matDnull = \matNull$. For our
choice of $\matAnull$ as a semi-orthogonal matrix, the matrix
$\matWnull$ contains unit vectors $\vecwnull_j$ in its
columns. According to the Rayleigh-Ritz Theorem
(\ref{eq_rayleigh_ritz}) and under assumption
(\ref{eq_order_lambda}), we can conclude that $\matDnull$ contains
strictly positive and therefore non-zero elements, thus $\matDnull$ is
invertible. Multiplication of $\matT\matDnull = \matNull$ by
$\matDnull^{-1}$ then leads to $\matT = \matNull$, thus $\matM$ is
block-diagonal. Therefore $\matS$ can be expressed by
(\ref{eq_TwJ2_S_XLXT}) as in case TwJ2:
\begin{equation}\label{eq_N2_S_XLXT}
  \matS = \matX\hat{\matLambda}^*\matX^T.
\end{equation}
Note that according to \lemmaref{\ref{lemma_blockdiag_evec}}, $\matX$
is an orthogonal matrix. Since the orthogonal similarity
transformation preserves eigenvalues, $\matS$ has pairwise different
eigenvalues as does $\hat{\matLambda}^*$. $\matS$ is also symmetric.

We continue from (\ref{eq_N2_SD_DS}). To fulfill the prerequisites of
\lemmaref{\ref{lemma_commute_blockdiag}}, we interpret $\matDnull$ as
the permuted version of a diagonal matrix $\matDnull^*$ where
identical diagonal elements are contiguous, i.e.
\begin{equation}\label{eq_N2_matDnull_star}
\matDnull = \matP^{*T} \matDnull^* \matP^*
\end{equation}
where
\begin{equation}
  \matDnull^*
  =
  \pmat{
    \overline{d}^{*'}_1\matI_1 & \matNull & \ldots & \matNull\\
    \matNull & \overline{d}^{*'}_2\matI_2 & \ldots & \matNull\\
    \vdots   & \vdots   & \ddots & \vdots\\
    \matNull & \matNull & \ldots & \overline{d}^{*'}_k \matI_k}.
\end{equation}
This gives
\begin{eqnarray}
  \matS \matDnull &=& \matDnull \matS\\
  \matS \matP^{*T} \matDnull^* \matP^* &=& \matP^{*T} \matDnull^* \matP^* \matS\\
  \matP^* \matS \matP^{*T} \matDnull^* &=& \matDnull^* \matP^* \matS \matP^{*T}\\
  \label{eq_N2_SsDs_DsSs}
  \matS^* \matDnull^* &=& \matDnull^* \matS^*
\end{eqnarray}
where $\matS^* = \matP^* \matS \matP^{*T}$. By applying
\lemmaref{\ref{lemma_commute_blockdiag}} to (\ref{eq_N2_SsDs_DsSs}) we
know that $\matS^*$ is block-diagonal (the contiguous order of
diagonal elements in $\matDnull^*$ as presumed by the lemma is given),
and according to \lemmaref{\ref{lemma_diagonalization_blockdiag}} we
can express $\matS^*$ as
\begin{equation}
\matS^* = \matU^* \matDelta \matU^{*T}
\end{equation}
where $\matU^*$ is a block-diagonal orthogonal matrix and $\matDelta$ is
diagonal. Therefore
\begin{eqnarray}
\matP^* \matS \matP^{*T} &=& \matU^* \matDelta \matU^{*T}\\
\matS &=& \matP^{*T} \matU^* \matDelta \matU^{*T} \matP^*.
\end{eqnarray}
Combined with (\ref{eq_N2_S_XLXT}), this leads to
\begin{eqnarray}
  \matP^{*T} \matU^* \matDelta \matU^{*T} \matP^*
  &=&
  \matX \hat{\matLambda}^* \matX^T\\
  \matDelta
  &=&
  \matU^{*T} \matP^* \matX \hat{\matLambda}^* \matX^T \matP^{*T} \matU^*.
\end{eqnarray}
Both $\matDelta$ and $\hat{\matLambda}^*$ are diagonal, thus from
\lemmaref{\ref{lemma_RTDR_Ds}} we can conclude that
\begin{eqnarray}
  \matX^T \matP^{*T} \matU^* &=& \matXi' \matP'\\
  \matX^T &=& \matXi' \matP' \matU^{*T} \matP^*
\end{eqnarray}
where $\matP'$ is a permutation matrix and $\matXi'$ a diagonal sign
matrix. We insert this into equation (\ref{eq_T_TwJ2_matAnull}) which
holds for special case T, TwJ2, and N2:
\begin{eqnarray}
  \matAnull
  &=& \matXi \matP \pmat{\matX^T\\ \matNull}\\
  &=& \matXi \matP \pmat{\matXi' \matP' \matU^{*T} \matP^*\\ \matNull}.
\end{eqnarray}
We proceed in a similar way as for the analysis of special case N1
(section \ref{sec_analysis_N1}), indicating matrix sizes from now on:
\begin{eqnarray}
  \matAnull
  &=&
  \matXi_n \matP_n \pmat{\matXi'_m \matP'_m \matU^{*T}_m \matP^*\\ \matNull}\\
  &=&
  \matXi_n \matP_n
  \underbrace{
    \pmat{\matXi'_m & \matNull\\ \matNull & \check{\matXi}'_{n-m}}}_{\hat{\matXi}'_n}
  \underbrace{
    \pmat{\matP'_m & \matNull\\ \matNull & \check{\matP}'_{n-m}}}_{\hat{\matP}'_n}
  \pmat{\matU^{*T}_m \matP^*\\ \matNull}\\
  &=&
  \matXi_n \matP_n \hat{\matXi}'_n \hat{\matP}'_n
  \pmat{\matU^{*T}_m \matP^*\\ \matNull}\\
  &=&
  \matXi_n \hat{\matXi}''_n \matP_n \hat{\matP}'_n
  \pmat{\matU^{*T}_m \matP^*\\ \matNull}\\
  \label{eq_N2_solution_matA}
  &=&
  \matXi'''_n \matP''_n \pmat{\matU^{*T}_m \matP^*\\ \matNull}
\end{eqnarray}
where we applied \lemmaref{\ref{lemma_DP}} and fused sign matrices and
permutation matrices, respectively.

We obtain (on the way omitting primes, leaving out some matrix sizes,
replacing $\matU^* \coloneqq \matU_m^*$, and integrating signs into
eigenvectors)
\begin{equation}\label{eq_fp_N2}
  \matWnull
  = \matV \matP \pmat{\matU^{*T}\matP^*\\ \matNull}
  = \matV \matP \pmat{\matU\\ \matNull}.
\end{equation}
Note that this solution resembles (\ref{eq_fp_oja_mod}) for special
case T, but here $\matU^*$ is not arbitrary but is a block-diagonal
orthogonal matrix with a shape depending on the multiplicity of
entries in $\matDnull$. Post-multiplication by $\matP^*$ permutes the
rows of $\matU^{*T}$ and thus the columns of $\matU^*$.

For the special case of pairwise different elements in $\matDnull$,
the orthogonal blocks in $\matU^*$ are of size $1 \times 1$ and can
only be $\pm 1$, thus $\matU^* = \matXi^*$. The diagonal sign matrix
$\matXi^*$ and the permutation matrix $\matP^*$ can then be integrated
into $\matP$ and $\matV$ as described above such that we obtain
\begin{equation}\label{eq_fp_N2_special}
\matWnull = \matV \matP \pmat{\matI_m\\ \matNull}
\end{equation}
which relates to (\ref{eq_fp_N2}) by $\matU = \matI_m$.

The constraint $\matWnull^T\matWnull = \matI_m$ is fulfilled which can
be shown as for special case T.

\subsubsection{Constraint in Special Case N2}\label{sec_N2_constraint}
%'''''''''''''''''''''''''''''''''''''''''''

Besides $\matU^*$ being block-diagonal (depending on the multiplicity
of the diagonal elements of $\matDnull$), there is another constraint
on $\matU^*$. We look at
\begin{eqnarray}
  \matDnull
  &=&
  \Diag{j=1}{m}\{\vecwnull_j^T \matC \vecwnull_j\}\\
  &=&
  \Diag{j=1}{m}\{\vece_j^T \matWnull^T \matC \matWnull \vece_j\}\\
  &=&
  \dg\{\matWnull^T \matC \matWnull\}.
\end{eqnarray}
On the one hand, we have with (\ref{eq_fp_N2})
\begingroup
\setlength\arraycolsep{3pt}
\begin{align}
  &\matH\\
  &\coloneqq
  \matWnull^T \matC \matWnull\\
  &=
  \pmat{\matP^{*T}\matU^* & \matNull} \matP^T \matV^T
  \matV \matLambda \matV^T
  \matV \matP \pmat{\matU^{*T} \matP^*\\ \matNull}\\
  &=
  \pmat{\matP^{*T} \matU^* & \matNull}
  \matLambda^*
  \pmat{\matU^{*T} \matP^*\\ \matNull}\\
  &=
  \pmat{\matP^{*T} \matU^* & \matNull}
  \pmat{\hat{\matLambda}^* & \matNull\\ \matNull & \check{\matLambda}^*}
  \pmat{\matU^{*T} \matP^*\\ \matNull}\\
  &=
  \matP^{*T} \matU^* \hat{\matLambda}^* \matU^{*T} \matP^*\\
  &=
  \matP^{*T}
  \pmat{
    \matU^{*'}_1  & \matNull   & \ldots & \matNull\\
    \matNull   & \matU^{*'}_2  & \ldots & \matNull\\
    \vdots     & \vdots     & \ddots & \vdots\\
    \matNull   & \matNull   & \ldots & \matU^{*'}_k}
  \pmat{
    \hat{\matLambda}^*_1  & \matNull  & \ldots & \matNull\\
    \matNull  & \hat{\matLambda}^*_2  & \ldots & \matNull\\
    \vdots    & \vdots    & \ddots & \vdots\\
    \matNull  & \matNull  & \ldots & \hat{\matLambda}^*_k}  
  \pmat{
    \matU^{*'T}_1  & \matNull  & \ldots & \matNull\\
    \matNull  & \matU^{*'T}_2  & \ldots & \matNull\\
    \vdots    & \vdots    & \ddots & \vdots\\
    \matNull  & \matNull  & \ldots & \matU^{*'T}_k}
  \matP^*\\
  &=
  \matP^{*T}
  \pmat{
    \matU^{*'}_1 \hat{\matLambda}^*_1 \matU^{*'T}_1 & \matNull &\ldots &\matNull\\
    \matNull & \matU^{*'}_2 \hat{\matLambda}^*_2 \matU^{*'T}_2 &\ldots &\matNull\\
    \vdots    & \vdots   & \ddots & \vdots\\
    \matNull  & \matNull &\ldots &\matU^{*'}_k\hat{\matLambda}^*_k \matU^{*'T}_k}
  \matP^*\\
  &=
  \matP^{*T}
  \underbrace{\pmat{
    \matH^{*'}_1  & \matNull  & \ldots & \matNull\\
    \matNull  & \matH^{*'}_2  & \ldots & \matNull\\
    \vdots    & \vdots    & \ddots & \vdots\\
    \matNull  & \matNull  & \ldots & \matH^{*'}_k}}_{\matH^*}
  \matP^*\\
  \label{eq_H_Hs}
  &=
  \matP^{*T} \matH^* \matP^*
\end{align}
\endgroup
where
\begin{equation}
  \matH^{*'}_l = \matU^{*'}_l \hat{\matLambda}^*_l \matU^{*'T}_l,
  \qquad l = 1,\ldots,k.
\end{equation}
If we only look at the diagonal elements, we get with
\lemmaref{\ref{lemma_perm_dg}}
\begin{align}
  \matDnull
  &=
  \dg\{\matWnull^T\matC\matWnull\}\\
  &=
  \dg\{\matP^{*T}\matH^*\matP^*\}\\
  &=
  \matP^{*T} \dg\{\matH^*\} \matP^*.
 \end{align}
On the other hand, we assumed
\begin{eqnarray}
  \matDnull
  &=&
  \matP^{*T} \matD^* \matP^*\\
  &=&
  \matP^{*T}
  \pmat{
    \overline{d}^{*'}_1\matI_1 & \matNull & \ldots & \matNull\\
    \matNull & \overline{d}^{*'}_2\matI_2 & \ldots & \matNull\\
    \vdots   & \vdots   & \ddots & \vdots\\
    \matNull & \matNull & \ldots & \overline{d}^{*'}_k \matI_k}
  \matP^*
\end{eqnarray}
where we express the blocks of $\matD^*$ as
\begin{equation}
\matD^{*'}_l = \overline{d}^{*'}_l\matI_l \qquad l = 1,\ldots,k.
\end{equation}
This leads to the constraints
\begin{eqnarray}
  \matP^{*T} \dg\{\matH^*\} \matP^*
  &=&
  \matP^{*T} \matD^* \matP^*\\
  %-----
  \dg\{\matH^*\}
  &=&
  \matD^*\\
  %-----
  \dg\{\matH^{*'}_l\}
  &=&
  \matD^{*'}_l,\quad l=1,\ldots,k\\
  %-----
  \label{eq_N2_constraint_H}
  \dg\{\matH^{*'}_l\}
  &=&
  \overline{d}^{*'}_l \matI_l,\quad l=1,\ldots,k\\
  %-----
  \label{eq_N2_constraint_U}
  \dg\{\matU^{*'}_l \hat{\matLambda}^*_l \matU^{*'T}_l\}
  &=&
  \overline{d}^{*'}_l \matI_l,\quad l=1,\ldots,k
\end{eqnarray}
i.e. the diagonal elements of $\matH^{*'}_l = \matU^{*'}_l
\hat{\matLambda}^*_l \matU^{*'T}_l$ must be $\overline{d}^{*'}_l$, or,
put differently, $\matH^{*'}_l$ has the identical diagonal elements
$\overline{d}^{*'}_l$ (which are strictly positive). Note that
$\matH^{*'}_l$ can't be diagonal, since an orthogonal similarity
transformation between two diagonal matrices can only permute the
diagonal elements (\lemmaref{\ref{lemma_RTDR_Ds}}), but not change
their value (such that they would become identical to each other).

At the moment we can only say that it is possible to find matrices
$\matU^{*'}_l$ which fulfill this constraint. Hadamard matrices of
size $s_l$ (size of block $l$) with a factor of $1/\sqrt{s_l}$ are one
choice.\footnote{At \url{https://math.stackexchange.com/3590184},
  'user8675309' kindly provided an answer for the complex domain.}
However, Hadamard matrices are not available for all sizes and it is
apparently not even clear for which sizes they exist.\footnote{See
  \url{https://en.wikipedia.org/wiki/Hadamard_matrix}.} Further
insights on the set of solutions $\matU^{*'}_l$ are presently not
available.

As a side remark, the orthogonal similarity transformation preserves
the trace, so
\begin{equation}
  \tr\{\matU^{*'}_l \hat{\matLambda}^*_l \matU^{*'T}_l\}
  =
  \tr\{\hat{\matLambda}^*_l\}
  =
  \summe{j=1}{s_l} \hat{\lambda}^*_{l_j}
  =
  \tr\{\matH^{*'}_l\}
  =
  s_l \overline{d}^{*'}_l,\quad l=1,\ldots,k,
\end{equation}
where $s_l$ is the size of block $l$, and therefore
\begin{equation}
\overline{d}^{*'}_l = \frac{1}{s_l} \summe{j=1}{s_l} \hat{\lambda}^*_{l_j},
\quad l=1,\ldots,k.
\end{equation}

\subsubsection{Test of Solution of Special Case N2}\label{sec_N2_test}
%''''''''''''''''''''''''''''''''''''''''''''''''''

To test our solution, we insert the solution for $\matA$
(\ref{eq_N2_solution_matA}) (omitting primes) into
(\ref{eq_N2_intermediate}). We left- and right-multiply by orthogonal
matrices several times:
\begin{eqnarray}
  \matNull
  &=&
  \matLambda \matAnull \matDnull
  - \matAnull \matDnull \matAnull^T \matLambda \matAnull\\
  \matNull
  &=&
  \matLambda
  \matXi \matP \pmat{\matU^{*T}_m\matP^*\\\matNull}
  \matDnull\\
  &-&
  \matXi \matP \pmat{\matU^{*T}_m\matP^*\\\matNull}
  \matDnull
  \pmat{\matP^{*T}\matU^*_m&\matNull} \matP^T \matXi^T
  \matLambda
  \matXi \matP \pmat{\matU^{*T}_m\matP^*\\\matNull}\\
  \matNull
  &=&
  \matXi^T
  \matLambda
  \matXi \matP \pmat{\matU^{*T}_m\matP^*\\\matNull}
  \matDnull\\
  &-&
  \matXi^T \matXi \matP \pmat{\matU^{*T}_m\matP^*\\\matNull}
  \matDnull
  \pmat{\matP^{*T}\matU^*_m&\matNull} \matP^T
  \matLambda
  \matP \pmat{\matU^{*T}_m\matP^*\\\matNull}\\
  \matNull
  &=&
  \matLambda
  \matP \pmat{\matU^{*T}_m\matP^*\\\matNull}
  \matDnull\\
  &-&
  \matP \pmat{\matU^{*T}_m\matP^*\\\matNull}
  \matDnull
  \pmat{\matP^{*T}\matU^*_m&\matNull}
  \matLambda^*
  \pmat{\matU^{*T}_m\matP^*\\\matNull}\\
  \matNull
  &=&
  \matP^T \matLambda \matP
  \pmat{\matU^{*T}_m\matP^*\\\matNull}
  \matDnull\\
  &-&
  \matP^T \matP \pmat{\matU^{*T}_m\matP^*\\\matNull}
  \matDnull
  \pmat{\matP^{*T}\matU^*_m&\matNull}
  \matLambda^*
  \pmat{\matU^{*T}_m\matP^*\\\matNull}\\
  \matNull
  &=&
  \matLambda^*
  \pmat{\matU^{*T}_m\matP^*\\\matNull}
  \matDnull
  -
  \pmat{\matU^{*T}_m\matP^*\\\matNull}
  \matDnull
  \matP^{*T}\matU^*_m \hat{\matLambda}^* \matU^{*T}_m\matP^*\\
  \matNull
  &=&
  \hat{\matLambda}^*
  \matU^{*T}_m\matP^*
  \matDnull
  -
  \matU^{*T}_m\matP^*
  \matDnull
  \matP^{*T}\matU^*_m \hat{\matLambda}^* \matU^{*T}_m\matP^*\\
  \matNull
  &=&
  \hat{\matLambda}^*
  \matU^{*T}_m
  \underbrace{
  \matP^*
  \matDnull
  \matP^{*T}}_{\matDnull^*}
  \matU^*_m
  -
  \matU^{*T}_m
  \underbrace{
  \matP^*
  \matDnull
  \matP^{*T}
  }_{\matDnull^*}
  \matU^*_m \hat{\matLambda}^*\\
  \matNull
  &=&
  \hat{\matLambda}^*
  \matDnull^*
  -
  \matDnull^*
  \hat{\matLambda}^*
\end{eqnarray}
where we used (\ref{eq_N2_matDnull_star}) in the next-to-last step and
applied \lemmaref{\ref{lemma_blockdiag_diag}} in the last step; the
latter is possible since $\matDnull^*$ has the shape specified in
\lemmaref{\ref{lemma_commute_blockdiag}}. Since diagonal matrices are
exchangeable within a product, this confirms our solution.

Surprisingly, while the solution is only fulfilled if the
block-diagonal shape of $\matU^*_m$ is taken into account, the
additional constraint (\ref{eq_N2_constraint_U}) on $\matU^*_m$ from
section \ref{sec_N2_constraint} is not required.

\subsubsection{Analysis of Special Case N2}
%''''''''''''''''''''''''''''''''''''''''''

The analysis of special case N2 is similar to that of special case N1
(see section \ref{sec_analysis_N1}; we take a slightly different
path). From (\ref{eq_fp_N2}), we extract column
\begin{equation}\label{eq_fp_N2_vec}
  \vecwnull_j = \matV \matP \pmat{\matU\vece_j\\ \vecnull}
\end{equation}
where $\vece_j$ is the unit vector with element $1$ at position
$j$. We insert this into the objective function (\ref{eq_objfct_new}):
\begin{eqnarray}
  J(\matWnull)
  &=&
  \quarter \summe{j=1}{m}
  (\vecwnull_j^T \matC \vecwnull_j)^2\\
  &=&
  \quarter \summe{j=1}{m}
  (\vecwnull_j^T \matV \matLambda \matV^T \vecwnull_j)^2\\
  &=&
  \quarter \summe{j=1}{m}
  \left[
    \pmat{\vece_j^T \matU^T & \vecnull^T}
    \matP^T \matV^T \matV \matLambda
    \matV^T \matV \matP \pmat{\matU \vece_j\\\vecnull}
  \right]^2\\
  &=&
  \quarter \summe{j=1}{m}
  \left[
    \pmat{\vece_j^T \matU^T & \vecnull^T}
    \matP^T \matLambda \matP
    \pmat{\matU \vece_j\\\vecnull}
  \right]^2\\
  &=&
  \quarter \summe{j=1}{m}
  \left[
    \pmat{\vece_j^T \matU^T & \vecnull^T}
    \matLambda^*
    \pmat{\matU \vece_j\\\vecnull}
  \right]^2\\
  &=&
  \quarter \summe{j=1}{m}
  \left[
    \pmat{\vece_j^T \matU^T & \vecnull^T}
    \pmat{\hat{\matLambda}^* & \matNull\\ \matNull & \check{\matLambda}^*}  
    \pmat{\matU \vece_j\\\vecnull}
  \right]^2\\
  &=&
  \quarter \summe{j=1}{m}
  \left[
    \vece_j^T \matU^T \hat{\matLambda}^* \matU \vece_j
  \right]^2\\
  &=&
  \quarter \summe{j=1}{m} \left[(\matU^T \hat{\matLambda}^* \matU)_{jj}\right]^2.
\end{eqnarray}
According to \lemmaref{\ref{lemma_RTDR_ii_sqr}}, the maximum is
achieved if $\matU$ is a signed permutation matrix: $\matU = \matXi
\matP'$. In this case, $\matU^T \hat{\matLambda}^* \matU$ is
diagonal. We would end with the same equation as
(\ref{eq_N1_solution_max}) in case N1. It is presently not clear what
conclusions can be drawn from this analysis. First, in contrast to
case N1, matrix $\matU$ is constrained. It is not even clear whether
the maximum can be reached under this constraint. It is doubtful
whether this analysis indicates spurious solutions (as in case N1),
also since spurious solutions do not appear in TwJ2 and TwC2.

The analysis above shows that the {\em global} maximum of $J$ for
solution (\ref{eq_fp_N2_vec}) is achieved if $\hat{\matLambda}^*$
contains the largest eigenvalues from $\matLambda$, regardless of
their order (in contrast to special case TwJ2), thus there exist
several fixed points which reach the global maximum. Other fixed
points lead to lower values of $J$. Stability of the fixed points has
to be checked in a separate step.

%#############################################################################
%#############################################################################
%#############################################################################

\section{Behavior of Constrained Objective Functions at the Critical Points}
%===========================================================================
\label{sec_behavior}

\subsection{Introduction}\label{sec_behavior_intro}
%------------------------

In the following sections we perform a local analysis of the behavior
of the objective functions on the Stiefel manifold at its critical
points. We only look at special cases TwJ and N, since special case T
has no isolated critical points, and special case TwC uses a weighted
constraint instead of the Stiefel constraint.

This analysis is based on the following approach:
\begin{enumerate}
\item Let the critical point be denoted by $\matWnull$. We describe a
  step away from the critical point on a tangent direction $\matDelta$
  on the Stiefel manifold. This leads to a point $\matW' = \matWnull +
  \matDelta$.
\item The step on a tangent direction is completely parametrized by a
  skew-symmetric matrix $\matA$ ($m \times m$) and a matrix $\matB$
  ($(n-m) \times m$). With a matrix $\matWnull_\perp$ ($n \times
  (n-m)$) complementing $\matWnull$ to size $n \times n$, we can
  express the step as $\matDelta = \matWnull \matA + \matWnull_\perp
  \matB$ (\lemmaref{\ref{lemma_stiefel_tangent_para}}).
\item The point $\matW'$ reached by a step on the tangent space at
  $\matWnull$ can be exactly projected back onto the Stiefel manifold
  by $\matW = \matW' (\matW'^T \matW')^{-\half}$
  (\lemmaref{\ref{lemma_stiefel_svd}}).
\item For the local analysis, we approximate this for small steps
  $\matDelta$ by $\matW \approx \matW' - \half \matWnull \matDelta^T
  \matDelta$ (\lemmaref{\ref{lemma_stiefel_proj_appr}}).
\item If we insert the parametrized $\matDelta$ into the last
  equation, we obtain $ \matW \approx \matWnull (\matI + \matA -\half
  [\matA^T\matA + \matB^T\matB]) + \matW_\perp \matB$
  (\lemmaref{\ref{lemma_stiefel_proj_appr_tangent}}).
\item We analyze the change in the objective function $\Delta J =
  J(\matW) - J(\matWnull)$ on the manifold. We omit terms above second
  order in $\matA$ and $\matB$, since these are dominated by the
  lower-order terms for small steps.
\item We analyze which critical points are local maxima. At these
  critical points, we have $\Delta J < 0$ for {\em arbitrary} choices
  of $\matA$ and $\matB$ (if not both are zero). In one case we show
  that $\Delta J > 0$ for {\em specific} choices of $\matA$ and
  $\matB$ which proves the existence of a local minimum or a saddle
  point.
\end{enumerate}
What is presently unknown is the relation between the statements
derived from this analysis and the actual convergence behavior of the
different learning rules. We can only assume that the statements also
affect the behavior of the learning rules since these are derived from
gradients of the objective functions under consideration of the
Stiefel constraint. However, one can probably show that the weight
changes of all learning rules lie in the tangent space, so even if we
start on the Stiefel manifold, a learning step will lead to a point
away from the manifold. Some learning rules may have the ability to
move back towards the Stiefel manifold under some circumstances
(particularly small deviations from the manifold), other learning
rules may require explicit steps for the back-projection onto the
manifold (either exact or approximated, see above). While the relation
may be obvious for the ``short forms'' of the learning rules (section
\ref{sec_derivation_short}) due to their direct relationship to the
fixed-point equations, it may be more complex for the ``long forms''
(section \ref{sec_derivation_long}) which are derived by computing
another derivative.

\subsection{Preparatory Computations}
%------------------------------------

The following preparatory computations concern both cases analyzed
below. We explore the behavior of the objective functions in the
vicinity of the critical points
\begin{equation}\label{eq_putative_fp}
\matWnull = \matV \matP \pmat{\matU_m\\ \matNull}.
\end{equation}
(where $\matV$ contains all eigenvectors associated with eigenvalues
in $\matLambda$ in descending order, $\matU_m$ is an orthogonal
matrix, and $\matP$ a permutation matrix) where the objective function
could have a minimum, maximum, or a saddle point on the
manifold. These critical points are obtained for special case TwJ2
(\ref{eq_fp_TwJ2}) where $\matU_m = \matI_m$ and for special case N2
(\ref{eq_fp_N2}) where $\matU_m = \matU^{*T}\matP^*$ with
block-diagonal $\matU^*$ and the constraint from section
\ref{sec_N2_constraint}. Note that the fixed-point equation
(\ref{eq_fp_TwJ2}) for case TwJ2 coincides with the fixed-point
equation of the gradient of the traditional objective function
(\ref{eq_objfct}) on the Stiefel manifold in the canonical metric
(\ref{eq_ode_TSC}), and the fixed point equation for case N2 coincides
with the fixed-point equation of the gradient of the novel objective
function (\ref{eq_objfct_new}) on the Stiefel manifold in the
canonical metric (\ref{eq_ode_NSC}).

For the derivation we have to chose a matrix $\matWnull_\perp$ which
complements $\matWnull$ to an orthogonal matrix:
\begin{equation}
\matWnull_\perp = \matV \matP \pmat{\matNull\\ \matI_{n-m}}.
\end{equation}
Using equation (\ref{eq_stiefel_proj_appr_tangent}) from
\lemmaref{\ref{lemma_stiefel_proj_appr_tangent}} we compute a weight
matrix $\matW$ which lies approximately on the Stiefel manifold in the
vicinity of $\matWnull$, where $\matA$ (skew-symmetric of size $m
\times m$) and $\matB$ (of size $(n-m) \times m$) parametrize a small
deviation along a tangent direction:
\begin{eqnarray}
  \matW
  &\approx&
  \matWnull
  \underbrace{\left(\matI + \matA
    - \half [\matA^T \matA + \matB^T \matB]\right)}_{\matF}
  + \matWnull_\perp \matB\\
  &=&
  \matV \matP \pmat{\matU_m\\ \matNull} \matF
  + \matV \matP \pmat{\matNull\\ \matI_{n-m}} \matB\\
  &=&
  \label{eq_approxW_prep}
  \matV \matP
  \left\{
  \pmat{\matU_m\\ \matNull} \matF
  + \pmat{\matNull\\ \matI_{n-m}} \matB
  \right\}
\end{eqnarray}
where $\matF$ (of size $m \times m$) is
\begin{equation}\label{eq_matF}
  \matF = \matI + \matA - \half [\matA^T \matA + \matB^T \matB].
\end{equation}
The following expression appears for the value of the objective
functions at $\matWnull$ from (\ref{eq_putative_fp}):
\begin{eqnarray}
  \matH \coloneqq \matWnull^T \matC \matWnull
  &=&
  \pmat{\matU_m^T & \matNull}
  \matP^T \matV^T \matC \matV \matP
  \pmat{\matU_m\\ \matNull}\\
  &=&
  \pmat{\matU_m^T & \matNull}
  \matP^T \matLambda \matP
  \pmat{\matU_m\\ \matNull}\\
  &=&
  \pmat{\matU_m^T & \matNull}
  \matLambda^*
  \pmat{\matU_m\\ \matNull}\\
  &=&
   \pmat{\matU_m^T & \matNull}
  \pmat{\hat{\matLambda}^* & \matNull\\ \matNull & \check{\matLambda}^*}
  \pmat{\matU_m\\ \matNull}\\
  &=&
  \label{eq_second_term_prep}
  \matU_m^T \hat{\matLambda}^* \matU_m
\end{eqnarray}
where $\matLambda^*$ is a permuted version of $\matLambda$,
$\hat{\matLambda}^*$ is the upper left $m \times m$ portion of
$\matLambda^*$, and $\matH$ represents the transformed version of
$\hat{\matLambda}^*$. Note that $\matH$ is symmetric.

The following expression appears for the value of the objective
functions at the approximated $\matW$ from (\ref{eq_approxW_prep}):
\begin{align}
  \nonumber
  &\matW^T \matC \matW\\
  &=
  \left\{
  \pmat{\matU_m\\ \matNull} \matF + \pmat{\matNull\\ \matI_{n-m}} \matB
  \right\}^T
  \matP^T \matV^T \matC \matV \matP
  \left\{
  \pmat{\matU_m\\ \matNull} \matF
  + \pmat{\matNull\\ \matI_{n-m}} \matB
  \right\}\\
  &=
  \left\{
  \pmat{\matU_m\\ \matNull} \matF
  + \pmat{\matNull\\ \matI_{n-m}} \matB
  \right\}^T
  \matLambda^*
  \left\{
  \pmat{\matU_m\\ \matNull} \matF
  + \pmat{\matNull\\ \matI_{n-m}} \matB
  \right\}\\
  &=
  \matF^T \pmat{\matU_m^T & \matNull}
  \matLambda^*
  \pmat{\matU_m\\ \matNull} \matF
  +
  \matB^T \pmat{\matNull & \matI_{n-m}}
  \matLambda^*
  \pmat{\matNull\\ \matI_{n-m}} \matB\\
  &=
  \matF^T \matU_m^T \hat{\matLambda}^* \matU_m \matF
  +
  \matB^T \check{\matLambda}^* \matB\\
  \label{eq_first_term_prep}
  &=
  \matF^T \matH \matF
  +
  \matB^T \check{\matLambda}^* \matB  
\end{align}
where $\check{\matLambda}^*$ denotes the lower right $(n-m) \times
(n-m)$ portion of $\matLambda^*$.

For the first term we insert (\ref{eq_matF}):
\begin{align}
  \nonumber
  &\matF^T \matH \matF\\[5mm]
  %---
  &=
  \left(\matI + \matA^T - \half [\matA^T \matA + \matB^T \matB]\right)
  \matH
  \left(\matI + \matA - \half [\matA^T \matA + \matB^T \matB]\right)\\[5mm]
  %---
  \nonumber
  &=
  \matH
  + \matA^T \matH
  - \half (\matA^T \matA + \matB^T \matB) \matH\\
  \nonumber
  &+
  \matH \matA
  + \matA^T \matH \matA
  - \half (\matA^T \matA + \matB^T \matB) \matH \matA\\
  \nonumber
  &-
  \half \matH (\matA^T \matA + \matB^T \matB)
  - \half \matA^T \matH (\matA^T \matA + \matB^T \matB)\\
  &+
  \quarter (\matA^T \matA + \matB^T \matB)
  \matH
  (\matA^T \matA + \matB^T \matB)\\[5mm]
  %---
  \nonumber
  &\approx
  \matH
  + \matA^T \matH
  - \half (\matA^T \matA + \matB^T \matB) \matH\\
  \label{eq_FTLF_prep}
  &+
  \matH \matA
  + \matA^T \matH \matA
  - \half \matH (\matA^T \matA + \matB^T \matB)
\end{align}
where we omitted terms above second order (in $\matA$ and $\matB$) in
the last step.

\subsection{Special Case TwJ}
%----------------------------

We analyze the behavior of objective function (\ref{eq_objfct})
\begin{equation}
  J(\matW) = \half \tr\{\matW^T \matC \matW \matTheta\}
\end{equation}
in the vicinity of the critical points (\ref{eq_putative_fp}). For
special case TwJ, we have only critical points where $\matU_m =
\matI_m$ and therefore $\matH = \hat{\matLambda}^*$; see section
\ref{sec_TwJ2_solution}. For that we determine the change of the
objective function from $\matWnull$ to $\matW$:
\begin{eqnarray}
  \Delta J
  &=&
  J(\matW) - J(\matWnull)\\
  &=&
  \half \tr\{\matW^T \matC \matW \matTheta\}
  -
  \half \tr\{\matWnull^T \matC \matWnull \matTheta\}.
\end{eqnarray}
For the second term we obtain with (\ref{eq_second_term_prep}):
\begin{equation}
  \tr\{\matWnull^T \matC \matWnull \matTheta\}
  =
  \tr\{\hat{\matLambda}^* \matTheta\}.
\end{equation}
For the first term we get with (\ref{eq_first_term_prep}):
\begin{equation}
  \tr\{\matW^T \matC \matW \matTheta\}
  =
  \tr\{\matF^T \hat{\matLambda}^* \matF \matTheta\}
  +
  \tr\{\matB^T \check{\matLambda}^* \matB \matTheta\}.
\end{equation}
We use (\ref{eq_FTLF_prep}) to further process
\begin{align}
  \nonumber
  \tr&\left\{\matF^T \hat{\matLambda}^* \matF \matTheta\right\}\\
  %---
  \nonumber
  \approx
  \tr&\left\{
  \left[\matH + \matA^T \matH + \matH \matA
    - \half (\matA^T \matA + \matB^T \matB) \matH
    - \half \matH (\matA^T \matA + \matB^T \matB)
  + \matA^T \matH \matA\right] \matTheta\right\}\\
  %---
  \nonumber
  =
  \tr&\left\{
  \left[\hat{\matLambda}^*
  - (\matA^T \matA + \matB^T \matB) \hat{\matLambda}^*
  + \matA^T \hat{\matLambda}^* \matA\right] \matTheta\right\}\\
  %---
  =
  \tr&\left\{
  \hat{\matLambda}^* \matTheta
  - (\matA^T \matA + \matB^T \matB) \hat{\matLambda}^* \matTheta
  + \matA^T \hat{\matLambda}^* \matA \matTheta\right\}
\end{align}
where we replaced $\matH = \hat{\matLambda}^*$, applied the invariance
of the trace to cyclic rotation and transposition and exchanged the
order of diagonal matrices. The linear terms disappear since $\matA$
is skew-symmetric and $\matH$ diagonal
(\lemmaref{\ref{lemma_tr_AskewD}}), also confirming that we actually
are at a critical point.

We can now continue with
\begin{align}
  \Delta J
  &=
  \half \tr\{\matF^T\hat{\matLambda}^*\matF\matTheta
  + \matB^T\check{\matLambda}^*\matB\matTheta\}
  -\half \tr\{\hat{\matLambda}^*\matTheta\}\\
  &\approx
  \half \tr\left\{
  (\matA^T \hat{\matLambda}^* \matA \matTheta
  - \matA^T \matA \hat{\matLambda}^* \matTheta)
  +
  (\matB^T\check{\matLambda}^*\matB\matTheta
  - \matB^T \matB \hat{\matLambda}^* \matTheta)
  \right\}.
\end{align}
Since $\matA$ and $\matB$ are independent perturbations, we can
analyze the terms separately. Note, however, that the individual
elements of $\matA$ are not independent since $\matA$ is
skew-symmetric; this motivates the splitting of the terms below.

For the terms containing $\matA$ (of size $m \times m$) we get with
\lemmaref{\ref{lemma_tr_ATDAOmega}} and \lemmaref{\ref{lemma_tr_ATAD}}
\begin{align}
  \nonumber
  &\tr\left\{\matA^T \hat{\matLambda}^* \matA \matTheta
  - \matA^T \matA \hat{\matLambda}^* \matTheta\right\}\\
  &=
  \summe{i=1}{m}\summe{k=1}{m} A^2_{ki} \hat{\lambda}^*_k \theta_i
  - \summe{i=1}{m}\summe{k=1}{m} A^2_{ki} \hat{\lambda}^*_i \theta_i\\
  &=
  \summe{i=1}{m}\summe{k=1}{m}
  A^2_{ki} (\hat{\lambda}^*_k - \hat{\lambda}^*_i) \theta_i\\
  &=
  \underbrace{\summe{i=1}{m}\summe{k=i+1}{m}
    A^2_{ki} (\hat{\lambda}^*_k - \hat{\lambda}^*_i) \theta_i}_{i < k}
  +
  \underbrace{\summe{k=1}{m}\summe{i=k+1}{m}
    A^2_{ki} (\hat{\lambda}^*_k - \hat{\lambda}^*_i) \theta_i}_{i > k}\\
  &=
  \summe{i=1}{m}\summe{k=i+1}{m}
  A^2_{ki} (\hat{\lambda}^*_k - \hat{\lambda}^*_i) \theta_i
  +
  \summe{i=1}{m}\summe{k=i+1}{m}
    A^2_{ik} (\hat{\lambda}^*_i - \hat{\lambda}^*_k) \theta_k\\
  &=
  \summe{i=1}{m}\summe{k=i+1}{m}
  A^2_{ki} (\hat{\lambda}^*_k - \hat{\lambda}^*_i) \theta_i
  -
  \summe{i=1}{m}\summe{k=i+1}{m}
  A^2_{ki} (\hat{\lambda}^*_k - \hat{\lambda}^*_i) \theta_k\\
  &=
  \summe{i=1}{m}\summe{k=i+1}{m}
  A^2_{ki} (\hat{\lambda}^*_k - \hat{\lambda}^*_i) (\theta_i - \theta_k)
\end{align}
where we renamed indices and considered that $\matA$ is skew-symmetric
and therefore $A^2_{ki} = A^2_{ik}$ and $A_{ii} = 0$. Note that to
confirm a maximum at $\matWnull$, we have to show that $\Delta J < 0$
for arbitrary non-zero perturbations. Since this includes perturbations
where only a single element of $\matA$ is non-zero, every term of the
sum has to be negative. We see that each term is negative iff either
($\hat{\lambda}^*_i < \hat{\lambda}^*_k$ and $\theta_i < \theta_k$) or
($\hat{\lambda}^*_i > \hat{\lambda}^*_k$ and $\theta_i >
\theta_k$). Since we assumed in equation (\ref{eq_order_theta})
that $\theta_1 > \ldots > \theta_m$, we can conclude that each term is
negative iff $\hat{\lambda}^*_1 > \ldots > \hat{\lambda}^*_m$, thus
the permutation described by $\matP$ sorts the first $m$ eigenvectors
in descending order of their associated eigenvalues.

For the terms containing $\matB$ (of size $(n-m)\times m$) we get with
\lemmaref{\ref{lemma_tr_ATDAOmega}} and \lemmaref{\ref{lemma_tr_ATAD}}
\begin{align}\
  \nonumber
  &\tr\left\{
  \matB^T\check{\matLambda}^*\matB\matTheta
  -
  \matB^T \matB \hat{\matLambda}^* \matTheta\right\}\\
  &=
  \summe{i=1}{m} \summe{k=1}{n-m}
  B^2_{ki} \check{\lambda}^*_k \theta_i
  -
  \summe{i=1}{m} \summe{k=1}{n-m}
  B^2_{ki} \hat{\lambda}^*_i \theta_i\\
  &=
  \summe{i=1}{m} \summe{k=1}{n-m}
  B^2_{ki} (\check{\lambda}^*_k - \hat{\lambda}^*_i) \theta_i.
\end{align}
We see that each term is negative iff $\check{\lambda}^*_k <
\hat{\lambda}^*_i$ for all $k \in [1,n-m]$ and $i \in [1,m]$, thus the
permutation described by $\matP$ sorts the eigenvectors such that the
associated eigenvalues of the first $m$ eigenvectors are larger than
those of the last $n-m$ eigenvectors.

In summary, a local maximum is present if the two conditions above are
fulfilled. In all other cases, we have a saddle point or a local
minimum since we can find directions (choices of $\matA$ and $\matB$)
where $\Delta J > 0$.

\subsection{Special Case N}\label{sec_N}
%--------------------------

We analyze the behavior of the novel objective function
(\ref{eq_objfct_new})
\begin{equation}
J(\matW) = \quarter \summe{j=1}{m} (\vecw_j^T \matC \vecw_j)^2
\end{equation}
in the vicinity of the critical points (\ref{eq_putative_fp}). For
that we determine the change of the objective function from
$\matWnull$ to $\matW$:
\begin{eqnarray}
  \Delta J
  &=&
  J(\matW) - J(\matWnull)\\
  \label{eq_N_deltaJ}
  &=&
  \quarter \summe{j=1}{m} (\vecw_j^T \matC \vecw_j)^2
  -
  \quarter \summe{j=1}{m} (\vecwnull_j^T \matC \vecwnull_j)^2.
\end{eqnarray}
We use the following expressions:
\begin{align}
  \vecwnull_j &= \matWnull \vece_j\\
  \vecw_j &= \matW \vece_j
\end{align}
where $\vece_j$ is the $m$-element unit vector with a $1$-element at
position $j$.

At this point we separately analyze
\begin{enumerate}
\item the special case where $\matDnull$ has pairwise different
  elements and thus $\matU_m = \matI_m$ and $\matH =
  \hat{\matLambda}^*$, treated in section \ref{sec_N_special_case}, and
\item the general case where some elements in $\matDnull$ may coincide
  and thus $\matH = \matU_m^T \hat{\matLambda}^* \matU_m$, treated in
  section \ref{sec_N_general_case},
\end{enumerate}
see section \ref{sec_N2_solution}.

\subsubsection{Special Case: Pairwise Different Elements}
%''''''''''''''''''''''''''''''''''''''''''''''''''''''''
\label{sec_N_special_case}

We first look at the second term of (\ref{eq_N_deltaJ}) where we apply
(\ref{eq_second_term_prep}):
\begin{align}
  \nonumber
  &\vecwnull_j^T \matC \vecwnull_j\\
  &= \vece_j^T \matWnull^T \matC \matWnull \vece_j\\
  &= \vece_j^T \hat{\matLambda}^* \vece_j \\
  &= \hat{\lambda}^*_j.
\end{align}
For the squared expression we obtain:
\begin{equation}
  (\vecwnull_j^T \matC \vecwnull_j)^2 = \hat{\lambda}^{*^2}.
\end{equation}
We now look at the first term of (\ref{eq_N_deltaJ}) where we apply
(\ref{eq_first_term_prep}):
\begin{align}
  &\vecw_j^T \matC \vecw_j\\
  &= \vece_j^T \matW^T \matC \matW \vece_j\\
  &= \vece_j^T \matF^T \hat{\matLambda}^* \matF \vece_j
  +
  \vece_j^T \matB^T \check{\matLambda}^* \matB \vece_j.
\end{align}
For the squared expression we obtain:
\begin{align}
  \nonumber
  &(\vecw_j^T \matC \vecw_j)^2\\
  &=
  (\vece_j^T \matF^T \hat{\matLambda}^* \matF \vece_j
  +
  \vece_j^T \matB^T \check{\matLambda}^* \matB \vece_j)^2\\
  \label{eq_N_squared_terms1}
  &=
  (\vece_j^T \matF^T \hat{\matLambda}^* \matF \vece_j)^2
  +
  2 (\vece_j^T \matF^T \hat{\matLambda}^* \matF \vece_j)
  (\vece_j^T \matB^T \check{\matLambda}^* \matB \vece_j)
  +
  (\vece_j^T \matB^T \check{\matLambda}^* \matB \vece_j)^2.
\end{align}
The first expression of (\ref{eq_N_squared_terms1}) is further
processed using (\ref{eq_FTLF_prep}) by:
\begin{align}
  \nonumber
  & \vece_j^T \matF^T \hat{\matLambda}^* \matF \vece_j\\[5mm]
  %---
  \nonumber
  &\approx
  \vece_j^T \hat{\matLambda}^* \vece_j
  + \vece_j^T \matA^T \hat{\matLambda}^* \vece_j
  - \half \vece_j^T (\matA^T \matA + \matB^T \matB) \hat{\matLambda}^* \vece_j\\
  &+
  \vece_j^T\hat{\matLambda}^* \matA \vece_j
  + \vece_j^T \matA^T \hat{\matLambda}^* \matA \vece_j
  - \half \vece_j^T \hat{\matLambda}^*
  (\matA^T \matA + \matB^T \matB) \vece_j\\[5mm]
  &=
  \hat{\lambda}^*_j
  + 2 (\matA^T \hat{\matLambda}^*)_{jj}
  + (\matA^T \hat{\matLambda}^* \matA)_{jj}
  - [(\matA^T \matA + \matB^T \matB) \hat{\matLambda}^*]_{jj}\\[5mm]
  %---
  \label{eq_N_first_term1}
  &=
  \hat{\lambda}^*_j
  + (\matA^T \hat{\matLambda}^* \matA)_{jj}
  - (\matA^T \matA + \matB^T \matB)_{jj} \hat{\lambda}^*_j
\end{align}
where, in the last step, the second, linear term disappeared due to
\lemmaref{\ref{lemma_AskewD_ii}} (confirming a critical point); we
also applied \lemmaref{\ref{lemma_AD_ii}} and
\lemmaref{\ref{lemma_ATD_DA_ii}}. We continue by squaring this
expression, again only including terms up to second order:
\begin{align}
  \nonumber
  & (\vece_j^T \matF^T \hat{\matLambda}^* \matF \vece_j)^2\\[5mm]
  %---
  & \approx
  \hat{\lambda}^{*^2}_j
  + 2 \hat{\lambda}^*_j (\matA^T \hat{\matLambda}^* \matA)_{jj}
  - 2 \hat{\lambda}^{*^2}_j (\matA^T \matA + \matB^T \matB)_{jj}.
\end{align}

For the second expression of (\ref{eq_N_squared_terms1}) we get with
(\ref{eq_N_first_term1}):
\begin{align}
  \nonumber
  &2 (\vece_j^T \matF^T \hat{\matLambda}^* \matF \vece_j)
  (\vece_j^T \matB^T \check{\matLambda}^* \matB \vece_j)\\[5mm]
  %---
  &\approx
  2[\hat{\lambda}^*_j
  + (\matA^T \hat{\matLambda}^* \matA)_{jj}
  - (\matA^T \matA + \matB^T \matB)_{jj} \hat{\lambda}^*_j]
  \cdot
  (\matB^T \check{\matLambda}^* \matB)_{jj}\\[5mm]
  %---
  &\approx
  2\hat{\lambda}^*_j (\matB^T \check{\matLambda}^* \matB)_{jj}
\end{align}
where we omitted all terms above second order in the last step.

The third expression of (\ref{eq_N_squared_terms1}) disappears since it
only includes fourth-order terms.

We insert all expressions and obtain (with $\matA$ of size $m\times m$
and $\matB$ of size $(n - m)\times m$, and applying (\ref{eq_ATDA_ij})
and (\ref{eq_ATA_ij}))
\begin{align}
  \label{eq_N_special_intermediate}
  \Delta J
  %---
  &\approx
  \quarter\summe{j=1}{m}\left[
  2 \hat{\lambda}^*_j (\matA^T \hat{\matLambda}^* \matA)_{jj}
  + 2 \hat{\lambda}^*_j (\matB^T \check{\matLambda}^* \matB)_{jj} 
  - 2 \hat{\lambda}^{*^2}_j (\matA^T \matA + \matB^T \matB)_{jj}\right]\\[5mm]
  %--
  \nonumber
  &=
  \half\summe{j=1}{m}\left[
    \hat{\lambda}^*_j  (\matA^T \hat{\matLambda}^* \matA)_{jj}
    - \hat{\lambda}^{*^2}_j (\matA^T \matA)_{jj}\right]\\
  &+
  \half\summe{j=1}{m}\left[
    \hat{\lambda}^*_j (\matB^T \check{\matLambda}^* \matB)_{jj}
    - \hat{\lambda}^{*^2}_j (\matB^T \matB)_{jj}\right]\\[5mm]
  %--
  \nonumber
  &=
  \half\summe{j=1}{m} \hat{\lambda}^*_j \left[
    (\matA^T \hat{\matLambda}^* \matA)_{jj}
    - \hat{\lambda}^*_j (\matA^T \matA)_{jj}\right]\\
  &+
  \half\summe{j=1}{m} \hat{\lambda}^*_j \left[
    (\matB^T \check{\matLambda}^* \matB)_{jj}
    - \hat{\lambda}^*_j (\matB^T \matB)_{jj}\right]\\[5mm]
  %--
  \nonumber
  &=
  \half\summe{j=1}{m} \hat{\lambda}^*_j \left[
    \summe{k=1}{m} A_{kj}^2 \hat{\lambda}^*_k
    - \hat{\lambda}^*_j \summe{k=1}{m} A_{kj}^2\right]\\
  &+
  \half\summe{j=1}{m} \hat{\lambda}^*_j \left[
    \summe{k=1}{n-m} B_{kj}^2 \check{\lambda}^*_j
    - \hat{\lambda}^*_j \summe{k=1}{n-m} B_{kj}^2\right]\\[5mm]
  %--
  &=
  \half\summe{j=1}{m} \hat{\lambda}^*_j 
  \summe{k=1}{m} A_{kj}^2 (\hat{\lambda}^*_k - \hat{\lambda}^*_j)
  +
  \half\summe{j=1}{m} \hat{\lambda}^*_j
  \summe{k=1}{n-m} B_{kj}^2 (\check{\lambda}^*_k - \hat{\lambda}^*_j)
\end{align}
Since $\matA$ and $\matB$ are independent perturbations, we can
analyze the terms separately. To confirm a maximum, we have to show
that $\Delta J < 0$ for arbitrary non-zero perturbations.

For the terms containing $\matB$, this is achieved iff
$\check{\lambda}^*_k < \hat{\lambda}^*_j$ for all $k \in [1,n-m]$ and
$j \in [1,m]$, thus the permutation described by $\matP$ sorts the
eigenvectors such that the associated eigenvalues of the first $m$
eigenvectors are larger than those of the last $n-m$ eigenvectors.

For the terms containing $\matA$, we split the sum into two halves,
consider the skew-symmetry through $A_{jk}^2 = A_{kj}^2$ and $A_{jj} =
0$ (this is necessary since the elements of $\matA$ are not
independent), exchange indices in the second sum, and fuse the two
sums:
\begin{align}
  \nonumber
  &
  \summe{j=1}{m} \summe{k=1}{m}
  A_{kj}^2 (\hat{\lambda}^*_k - \hat{\lambda}^*_j) \hat{\lambda}^*_j\\
  &=
  \underbrace{\summe{j=1}{m} \summe{k=j+1}{m}
    A_{kj}^2 (\hat{\lambda}^*_k - \hat{\lambda}^*_j) \hat{\lambda}^*_j}_{j < k}
  +
  \underbrace{\summe{k=1}{m} \summe{j=k+1}{m}
    A_{kj}^2 (\hat{\lambda}^*_k - \hat{\lambda}^*_j) \hat{\lambda}^*_j}_{j > k}\\
  &=
  \summe{j=1}{m} \summe{k=j+1}{m}
  A_{kj}^2 (\hat{\lambda}^*_k - \hat{\lambda}^*_j) \hat{\lambda}^*_j
  +
  \summe{k=1}{m} \summe{j=k+1}{m}
  A_{jk}^2 (\hat{\lambda}^*_k - \hat{\lambda}^*_j) \hat{\lambda}^*_j\\
  &=
  \summe{j=1}{m} \summe{k=j+1}{m}
  A_{kj}^2 (\hat{\lambda}^*_k - \hat{\lambda}^*_j) \hat{\lambda}^*_j
  +
  \summe{j=1}{m} \summe{k=j+1}{m}
  A_{kj}^2 (\hat{\lambda}^*_j - \hat{\lambda}^*_k) \hat{\lambda}^*_k\\
  &=
  \summe{j=1}{m} \summe{k=j+1}{m}
  A_{kj}^2 (\hat{\lambda}^*_k - \hat{\lambda}^*_j) \hat{\lambda}^*_j
  -
  \summe{j=1}{m} \summe{k=j+1}{m}
  A_{kj}^2 (\hat{\lambda}^*_k - \hat{\lambda}^*_j) \hat{\lambda}^*_k\\
  &=
  - \summe{j=1}{m} \summe{k=j+1}{m}
  A_{kj}^2 (\hat{\lambda}^*_k - \hat{\lambda}^*_j)^2.
\end{align}
We see that this part is always negative for non-zero perturbations
$\matA$, since we assumed $\hat{\lambda}^*_k \neq \hat{\lambda}^*_j$
for $j \neq k$ (note that the double sum does not include terms with
$j = k$). In contrast to case TwJ, there is no special order imposed
on the principal eigenvectors.

In summary, we have a maximum if the above condition (derived from the
$\matB$ terms) is fulfilled, otherwise we have a saddle point or a
minimum since we can find directions (choices of $\matA$ and $\matB$)
where $\Delta J > 0$.

\subsubsection{General Case: Elements May Not Be Pairwise Different}
%'''''''''''''''''''''''''''''''''''''''''''''''''''''''''''''''''''
\label{sec_N_general_case}

We first look at the second term of (\ref{eq_N_deltaJ}) where we apply
(\ref{eq_second_term_prep}):
\begin{align}
  \nonumber
  &\vecwnull_j^T \matC \vecwnull_j\\
  &= \vece_j^T \matWnull^T \matC \matWnull \vece_j\\
  &= \vece_j^T \matH \vece_j \\
  &= \matH_{jj} \eqqcolon h_j.
\end{align}
For the squared expression we obtain:
\begin{equation}
  \label{eq_N_general_fixed}
  (\vecwnull_j^T \matC \vecwnull_j)^2 = h_j^2.
\end{equation}
We now look at the first term of (\ref{eq_N_deltaJ}) where we apply
(\ref{eq_first_term_prep}):
\begin{align}
  &\vecw_j^T \matC \vecw_j\\
  &= \vece_j^T \matW^T \matC \matW \vece_j\\
  &= \vece_j^T \matF^T \matH \matF \vece_j
  +
  \vece_j^T \matB^T \check{\matLambda}^* \matB \vece_j.
\end{align}
For the squared expression we obtain:
\begin{align}
  \nonumber
  &(\vecw_j^T \matC \vecw_j)^2\\
  &=
  (\vece_j^T \matF^T \matH \matF \vece_j
  +
  \vece_j^T \matB^T \check{\matLambda}^* \matB \vece_j)^2\\
  \label{eq_N_squared_terms2}
  &=
  (\vece_j^T \matF^T \matH \matF \vece_j)^2
  +
  2 (\vece_j^T \matF^T \matH \matF \vece_j)
  (\vece_j^T \matB^T \check{\matLambda}^* \matB \vece_j)
  +
  (\vece_j^T \matB^T \check{\matLambda}^* \matB \vece_j)^2.
\end{align}
The first expression of (\ref{eq_N_squared_terms2}) is further
processed using (\ref{eq_FTLF_prep}) by:
\begin{align}
  \nonumber
  & \vece_j^T \matF^T \matH \matF \vece_j\\[5mm]
  %---
  \nonumber
  &\approx
  \vece_j^T \matH \vece_j
  + \vece_j^T \matA^T \matH \vece_j
  - \half \vece_j^T (\matA^T \matA + \matB^T \matB) \matH \vece_j\\
  &+
  \vece_j^T\matH \matA \vece_j
  + \vece_j^T \matA^T \matH \matA \vece_j
  - \half \vece_j^T \matH
  (\matA^T \matA + \matB^T \matB) \vece_j\\[5mm]
  %---
  \label{eq_N_first_term2}
  &=
  h_j
  + 2 (\matA^T \matH)_{jj}
  + (\matA^T \matH \matA)_{jj}
  - [(\matA^T \matA + \matB^T \matB) \matH]_{jj}
\end{align}
where we applied \lemmaref{\ref{lemma_square_square_ii}} to fuse pairs
of terms. Note that in contrast to section \ref{sec_N_special_case},
we can't eliminate the linear terms at this stage. We proceed by
squaring this term, including only terms up to second order:
\begin{align}
  \nonumber
  & (\vece_j^T \matF^T \matH \matF \vece_j)^2\\
  \nonumber
  & \approx
  h_j^2\\
  \nonumber
  &+ 4 h_j (\matA^T \matH)_{jj}
  + 2 h_j (\matA^T \matH \matA)_{jj}
  - 2 h_j [(\matA^T \matA + \matB^T \matB) \matH]_{jj}\\
  &+ 4 [(\matA^T \matH)_{jj}]^2.
\end{align}
Including only terms up to second order, the second expression of (\ref{eq_N_squared_terms2}) yields
\begin{equation}
  2 (\vece_j^T \matF^T \matH \matF \vece_j)
  (\vece_j^T \matB^T \check{\matLambda}^* \matB \vece_j)
  \approx
  2 h_j (\matB^T \check{\matLambda}^* \matB)_{jj}.
\end{equation}
The third expression of (\ref{eq_N_squared_terms2}) only contains
forth-order terms and can therefore be omitted. We obtain
\begin{align}
  \nonumber
  &(\vecw_j^T \matC \vecw_j)^2\\
  \nonumber
  &\approx
  h_j^2\\
  \nonumber
  &+ 4 h_j (\matA^T \matH)_{jj}
  + 2 h_j (\matA^T \matH \matA)_{jj}
  - 2 h_j [(\matA^T \matA + \matB^T \matB) \matH]_{jj}\\
  \nonumber
  &+ 4 [(\matA^T \matH)_{jj}]^2\\
  \label{eq_N_general_variable}
  &+ 2 h_j (\matB^T \check{\matLambda}^* \matB)_{jj}.
\end{align}
We insert (\ref{eq_N_general_variable}) and (\ref{eq_N_general_fixed})
into (\ref{eq_N_deltaJ}):
\begin{align}
  \nonumber
  \Delta J
  & \approx
  \summe{j=1}{m} h_j (\matA^T \matH)_{jj}
  + \half \summe{j=1}{m} h_j (\matA^T \matH \matA)_{jj}
  - \half \summe{j=1}{m} h_j [(\matA^T \matA + \matB^T \matB) \matH]_{jj}\\
  &+ \summe{j=1}{m} [(\matA^T \matH)_{jj}]^2
  + \half \summe{j=1}{m} h_j (\matB^T \check{\matLambda}^* \matB)_{jj}.
\end{align}
The linear term
\begin{equation}
 \summe{j=1}{m} h_j (\matA^T \matH)_{jj}
\end{equation}
should disappear at the critical points. We first consider that $\matH
= \matP^{*T} \matH^* \matP^*$ where $\matH^*$ is block-diagonal, see
equation (\ref{eq_H_Hs}). We express $\matP^*_{ij} =
\delta_{i,\pi*(j)}$ by the permutation $\pi^*(i)$ and apply
(\ref{eq_PTAP_ii}) from the proof \lemmaref{\ref{lemma_perm_dg}}. We
can then write the linear term as
\begin{align}
   \summe{j=1}{m} h_j (\matA^T \matH)_{jj}
   &=
   \summe{j=1}{m} \matH_{jj} (\matA^T \matH)_{jj}\\
   &=
   \summe{j=1}{m}
   (\matP^{*T} \matH^* \matP^*)_{jj}
   (\matA^T \matP^{*T} \matH^* \matP^*)_{jj}\\
   &=
   \summe{j=1}{m}
   (\matP^{*T} \matH^* \matP^*)_{jj}
   (\underbrace{\matP^{*T}\matP^*}_{\matI} \matA^T \matP^{*T} \matH^* \matP^*)_{jj}\\
   &=
   \summe{j=1}{m}
   \matH^*_{\pi^*(j),\pi^*(j)}
   (\underbrace{\matP^* \matA^T \matP^{*T}}_{\matA^*} \matH^*)_{\pi^*(j),\pi^*(j)}\\   
   &=
   \summe{j=1}{m}
   \matH^*_{jj}
   (\matA^* \matH^*)_{jj} 
\end{align}
where we could omit the permutation since it just affects the order of
the elements in the sum, not the sum itself. $\matA^*$ is another
arbitrary skew-symmetric matrix and therefore just a different
parametrization of the tangent direction.

To show that the linear term disappears, we write $\matA^*$ and $\matH^*$
as $k$ blocks (see section \ref{sec_N2_constraint}):
\begin{align}
  \nonumber
  &\matA^* \matH^*\\
  &=
  \pmat{
    \matA^{*'}_{11} & \matA^{*'}_{12} & \ldots & \matA^{*'}_{1k}\\
    \matA^{*'}_{21} & \matA^{*'}_{22} & \ldots & \matA^{*'}_{2k}\\
    \vdots        & \vdots        & \ddots & \vdots\\
    \matA^{*'}_{k1} & \matA^{*'}_{k2} & \ldots & \matA^{*'}_{kk}
  }
  \pmat{
    \matH^{*'}_1  & \matNull    & \ldots & \matNull\\
    \matNull     & \matH^{*'}_2 & \ldots & \matNull\\
    \vdots       & \vdots      & \ddots & \vdots\\
    \matNull     & \matNull    & \ldots & \matH^{*'}_k
  }\\
  \label{eq_N_matA_matH}
  &=
  \pmat{
    \matA^{*'}_{11} \matH^{*'}_1 & * & \ldots & *\\
    * & \matA^{*'}_{22} \matH^{*'}_2 & \ldots & *\\
    \vdots    & \vdots    & \ddots & \vdots\\
    * & * & \ldots & \matA^{*'}_{kk} \matH^{*'}_k
  }
\end{align}
where elements marked by the symbol `$*$' are not of interest. We can
continue block-wise (where $s_l$ is the size of block $l$) and by
considering constraint (\ref{eq_N2_constraint_H}):
\begin{align}
  \nonumber
  &
  \summe{j=1}{m} \matH^*_{jj} (\matA^* \matH^*)_{jj}\\
  &=
  \summe{l=1}{k} \summe{i=1}{s_l}
  \left(\matH^{*'}_l\right)_{ii} \left(\matA^{*'}_{ll} \matH^{*'}_{ll}\right)_{ii}\\
  &=
  \summe{l=1}{k} \summe{i=1}{s_l}
  \overline{d}^{*'}_l \left(\matA^{*'}_{ll} \matH^{*'}_{ll}\right)_{ii}\\
  &=
  \summe{l=1}{k} \overline{d}^{*'}_l
  \summe{i=1}{s_l} \left(\matA^{*'}_{ll} \matH^{*'}_{ll}\right)_{ii}\\
  &=
  \summe{l=1}{k} \overline{d}^{*'}_l
  \tr\{\matA^{*'}_{ll} \matH^{*'}_{ll}\}\\
  &=
  0.
\end{align}
The last step is motivated as follows: We know that $\matA^{*'}_{ll}$
is skew-symmetric and $\matH^{*'}_{ll}$ is symmetric (since $\matA^*$
and $\matH^*$ are skew-symmetric and symmetric, respectively). From
\lemmaref{\ref{lemma_tr_Askew_Bsymm}} we know that
$\tr\{\matA^{*'}_{ll} \matH^{*'}_{ll}\} = 0$; therefore the entire
linear term disappears. This also confirms that we actually are at a
critical point.

We are left with
\begin{align}
  \nonumber
  \Delta J
  &\approx
  \half \summe{j=1}{m} h_j \left\{
  (\matA^T \matH \matA)_{jj}
  - [(\matA^T \matA + \matB^T \matB) \matH]_{jj}
  + (\matB^T \check{\matLambda}^* \matB)_{jj}
  \right\}\\
  \label{eq_N_general_intermediate}
  &+ \summe{j=1}{m} [(\matA^T \matH)_{jj}]^2
\end{align}
Before we proceed, we check whether this intermediate solution
coincides with the intermediate solution
(\ref{eq_N_special_intermediate}) for special case from section
\ref{sec_N_special_case} where we have $\matU_m = \matI_m$, $\matH =
\hat{\matLambda}^*$, and $h_j = \hat{\lambda}^*_j$. We obtain
\begin{align}
  \nonumber
  \Delta J
  &\approx
  \half \summe{j=1}{m} \hat{\lambda}^*_j \left\{
  (\matA^T \hat{\matLambda}^* \matA)_{jj}
  - [(\matA^T \matA + \matB^T \matB) \hat{\matLambda}^*]_{jj}
  + (\matB^T \check{\matLambda}^* \matB)_{jj}
  \right\}\\
  &+ \summe{j=1}{m} [(\matA^T \hat{\matLambda}^*)_{jj}]^2.
\end{align}
We apply \lemmaref{\ref{lemma_AD_ii}} to the second term of the first
sum and \lemmaref{\ref{lemma_AskewD_ii}} to the second sum (which
disappears). This gives
\begin{align}
  \nonumber
  \Delta J
  &\approx
  \half \summe{j=1}{m} \hat{\lambda}^*_j \left\{
  (\matA^T \hat{\matLambda}^* \matA)_{jj}
  - (\matA^T \matA + \matB^T \matB)_{jj} \hat{\lambda}^*_j
  + (\matB^T \check{\matLambda}^* \matB)_{jj}
  \right\}
\end{align}
which coincides with (\ref{eq_N_special_intermediate}).

We return to the general case (\ref{eq_N_general_intermediate}). To
demonstrate that the critical points is not a maximum (and therefore a
saddle point or a minimum), we have to show that $\Delta J > 0$ for
some direction parameters $\matA$ and $\matB$. We try $\matB =
\matNull$ and get
\begin{align}
  \Delta J
  &\approx
  \half \summe{j=1}{m} h_j \left\{
  (\matA^T \matH \matA)_{jj}
  - (\matA^T \matA \matH)_{jj}
  \right\}\\
  &+
  \summe{j=1}{m} [(\matA^T \matH)_{jj}]^2\\[5mm]
  %-----
  &=
  \half \summe{j=1}{m} \matH_{jj}
  \left\{
  (\matA^T \matH \matA)_{jj}
  - (\matA^T \matA \matH)_{jj}
  \right\}\\
  &+
  \summe{j=1}{m} [(\matA^T \matH)_{jj}]^2\\[5mm]
  %-----
  &=
  \half \summe{j=1}{m} (\matP^{*T}\matH^*\matP^*)_{jj}
  \left\{
  (\matA^T \matP^{*T} \matH^* \matP^* \matA)_{jj}
  - (\matA^T \matA \matP^{*T} \matH \matP^*)_{jj}
  \right\}\\
  &+
  \summe{j=1}{m} [(\matA^T \matP^{*T} \matH \matP^*)_{jj}]^2\\[5mm]
  %-----
  &=
  \half \summe{j=1}{m} (\matP^{*T}\matH^*\matP^*)_{jj}
  \left\{
  (\underbrace{\matP^{*T}\matP^*} \matA^T \matP^{*T} \matH^* \matP^* \matA
  \underbrace{\matP^{*T}\matP^*})_{jj}
  \right\}\\
  &-
  \half \summe{j=1}{m} (\matP^{*T}\matH^*\matP^*)_{jj}
  \left\{
  (\underbrace{\matP^{*T}\matP^*} \matA^T
  \underbrace{\matP^{*T}\matP^*} \matA \matP^{*T} \matH^* \matP^*)_{jj}
  \right\}\\
  &+
  \summe{j=1}{m} [(\underbrace{\matP^{*T}\matP^*}
    \matA^T\matP^{*T} \matH \matP^*)_{jj}]^2\\[5mm]
  %-----
  &=
  \half \summe{j=1}{m} \matH^*_{\pi*(j),\pi^*(j)}
  \left\{
  (\matP^* \matA^T \matP^{*T} \matH^* \matP^* \matA\matP^{*T})_{\pi*(j),\pi^*(j)}
  \right\}\\
  &-
  \half \summe{j=1}{m} \matH^*_{\pi*(j),\pi^*(j)}
  \left\{
  (\matP^* \matA^T \matP^{*T}\matP^* \matA \matP^{*T} \matH^*)_{\pi*(j),\pi^*(j)}
  \right\}\\
  &+
  \summe{j=1}{m} [(\matP^*\matA^T\matP^{*T}\matH^*)_{\pi*(j),\pi^*(j)}]^2\\[5mm]
  %-----
  &=
  \half \summe{j=1}{m} \matH^*_{jj}
  \left\{
  (\underbrace{\matP^* \matA^T \matP^{*T}} \matH^*
  \underbrace{\matP^* \matA \matP^{*T}})_{jj}
  \right\}\\
  &-
  \half \summe{j=1}{m} \matH^*_{jj}
  \left\{
  (\underbrace{\matP^* \matA^T \matP^{*T}}
  \underbrace{\matP^* \matA \matP^{*T}} \matH^*)_{jj}
  \right\}\\
  &+
  \summe{j=1}{m} [(\underbrace{\matP^*\matA^T\matP^{*T}}\matH^*)_{jj}]^2\\[5mm]
  %-----
  \label{eq_N_saddle_sum1}
  &=
  \half \summe{j=1}{m} \matH^*_{jj}
  \left\{
  (\matA^{*T} \matH^* \matA^*)_{jj}
  \right\}\\
  \label{eq_N_saddle_sum2}
  &-
  \half \summe{j=1}{m} \matH^*_{jj}
  \left\{
  (\matA^{*T} \matA^* \matH^*)_{jj}
  \right\}\\
  \label{eq_N_saddle_sum3}
  &+
  \summe{j=1}{m} [(\matA^{*T}\matH^*)_{jj}]^2
\end{align}
where we inserted and expanded identity matrices (marked by braces),
considered that the permuted indices do not affect the value of the
sum, and introduced a different, but also arbitrary skew-symmetric
parameter $\matA^*$ (also marked by braces). We now try perturbations
where $\matA^*$ is block-diagonal with the same shape of blocks as in
$\matH^*$. For the term (\ref{eq_N_saddle_sum1}) we obtain,
considering the skew-symmetry of $\matA^*$:
\begingroup
\setlength\arraycolsep{3pt}
\begin{align}
  \nonumber
  &\matA^{*T} \matH^* \matA^*\\
  &=
  -\matA^* \matH^* \matA^*\\
  &=
  -\pmat{
    \matA^{*'}_{11} & \matNull   & \ldots & \matNull\\
    \matNull   & \matA^{*'}_{22} & \ldots & \matNull\\
    \vdots     & \vdots     & \ddots & \matNull\\
    \matNull   & \matNull   & \ldots & \matA^{*'}_{kk}
  }
  \pmat{
    \matH^{*'}_{11} & \matNull   & \ldots & \matNull\\
    \matNull   & \matH^{*'}_{22} & \ldots & \matNull\\
    \vdots     & \vdots     & \ddots & \matNull\\
    \matNull   & \matNull   & \ldots & \matH^{*'}_{kk}
  }
  \pmat{
    \matA^{*'}_{11} & \matNull   & \ldots & \matNull\\
    \matNull   & \matA^{*'}_{22} & \ldots & \matNull\\
    \vdots     & \vdots     & \ddots & \matNull\\
    \matNull   & \matNull   & \ldots & \matA^{*'}_{kk}
  }\\
  &=
  -\pmat{
    \matA^{*'}_{11}\matH^{*'}_{11}\matA^{*'}_{11} & \matNull   & \ldots & \matNull\\
    \matNull   & \matA^{*'}_{22}\matH^{*'}_{22}\matA^{*'}_{22} & \ldots & \matNull\\
    \vdots     & \vdots     & \ddots & \matNull\\
    \matNull   & \matNull   & \ldots & \matA^{*'}_{kk}\matH^{*'}_{kk}\matA^{*'}_{kk}
  }.
\end{align}
\endgroup
The term (\ref{eq_N_saddle_sum2}) is treated in the same way:
\begingroup
\setlength\arraycolsep{3pt}
\begin{align}
  \nonumber
  &\matA^{*T} \matA^* \matH^*\\
  &=
  -\matA^* \matA^* \matH^*\\
  &=
  -\pmat{
    \matA^{*'}_{11} & \matNull   & \ldots & \matNull\\
    \matNull   & \matA^{*'}_{22} & \ldots & \matNull\\
    \vdots     & \vdots     & \ddots & \matNull\\
    \matNull   & \matNull   & \ldots & \matA^{*'}_{kk}
  }
  \pmat{
    \matA^{*'}_{11} & \matNull   & \ldots & \matNull\\
    \matNull   & \matA^{*'}_{22} & \ldots & \matNull\\
    \vdots     & \vdots     & \ddots & \matNull\\
    \matNull   & \matNull   & \ldots & \matA^{*'}_{kk}
  }
  \pmat{
    \matH^{*'}_{11} & \matNull   & \ldots & \matNull\\
    \matNull   & \matH^{*'}_{22} & \ldots & \matNull\\
    \vdots     & \vdots     & \ddots & \matNull\\
    \matNull   & \matNull   & \ldots & \matH^{*'}_{kk}
  }\\
  &=
  -\pmat{
    \matA^{*'}_{11}\matA^{*'}_{11}\matH^{*'}_{11} & \matNull   & \ldots & \matNull\\
    \matNull   & \matA^{*'}_{22}\matA^{*'}_{22}\matH^{*'}_{22} & \ldots & \matNull\\
    \vdots     & \vdots     & \ddots & \matNull\\
    \matNull   & \matNull   & \ldots & \matA^{*'}_{kk}\matA^{*'}_{kk}\matH^{*'}_{kk}
  }.
\end{align}
\endgroup
Again, we continue block-wise (where $s_l$ is the size of block $l$)
and consider constraint (\ref{eq_N2_constraint_H}). For the term
(\ref{eq_N_saddle_sum1}) we get
\begin{align}
  \nonumber
  &
  \summe{j=1}{m} \matH^*_{jj} (\matA^{*T} \matH^* \matA^*)_{jj}\\
  &=
  -\summe{j=1}{m} \matH^*_{jj} (\matA^* \matH^* \matA^*)_{jj}\\
  &=
  -\summe{l=1}{k}
  \summe{i=1}{s_l}
  \left(\matH^{*'}_l\right)_{ii}
  \left(\matA^{*'}_{ll} \matH^{*'}_{ll} \matA^{*'}_{ll}\right)_{ii}\\
  &=
  -\summe{l=1}{k}
  \overline{d}^{*'}_l
  \summe{i=1}{s_l}
  \left(\matA^{*'}_{ll} \matH^{*'}_{ll} \matA^{*'}_{ll}\right)_{ii}\\
  &=
  -\summe{l=1}{k}
  \overline{d}^{*'}_l
  \tr\{\matA^{*'}_{ll} \matH^{*'}_{ll} \matA^{*'}_{ll}\}.
\end{align}
For the term (\ref{eq_N_saddle_sum2}) we get
\begin{align}
  \nonumber
  &
  \summe{j=1}{m} \matH^*_{jj} (\matA^{*T} \matA^* \matH^*)_{jj}\\
  &=
  -\summe{j=1}{m} \matH^*_{jj} (\matA^* \matA^* \matH^*)_{jj}\\
  &=
  -\summe{l=1}{k}
  \summe{i=1}{s_l}
  \left(\matH^{*'}_l\right)_{ii}
  \left(\matA^{*'}_{ll} \matA^{*'}_{ll} \matH^{*'}_{ll}\right)_{ii}\\
  &=
  -\summe{l=1}{k}
  \overline{d}^{*'}_l
  \summe{i=1}{s_l}
  \left(\matA^{*'}_{ll} \matA^{*'}_{ll} \matH^{*'}_{ll}\right)_{ii}\\
  &=
  -\summe{l=1}{k}
  \overline{d}^{*'}_l
  \tr\{\matA^{*'}_{ll} \matA^{*'}_{ll} \matH^{*'}_{ll}\}.
\end{align}
Since the trace is invariant under cyclic permutations we see that
\begin{equation}
  \tr\{\matA^{*'}_{ll} \matH^{*'}_{ll} \matA^{*'}_{ll}\}
  =
  \tr\{\matA^{*'}_{ll} \matA^{*'}_{ll} \matH^{*'}_{ll}\},
\end{equation}
therefore the sums (\ref{eq_N_saddle_sum1}) and
(\ref{eq_N_saddle_sum2}) cancel each other out and we are left with
the sum (\ref{eq_N_saddle_sum3})
\begin{align}
  \label{eq_N_saddle_sum3a}
  \Delta J
  \approx
  \summe{j=1}{m} [(\matA^{*T} \matH^*)_{jj}]^2.
\end{align}
In the following we only look at all cases which are not already
covered by the special case treated in section
\ref{sec_N_special_case}. For these cases we know that $\matH^*$ has
non-zero off-diagonal elements (see section
\ref{sec_N2_constraint}). $\matH^*$ is also symmetric. We can
therefore apply \lemmaref{\ref{lemma_skew_symm_diag_nonzero}},
according to which we can find a skew-symmetric $\matA^*$ such that
there is at least one diagonal element in $\matA^{*T} \matH^*$ which
is non-zero. Since the term is squared in (\ref{eq_N_saddle_sum3a}),
we can conclude $\Delta J > 0$. $\matH^*$ implicitly describes the
different critical points. The analysis above shows we can find small
steps on the Stiefel manifold (parametrized by $\matA^*$) away from all
these critical points under which the objective function
increases. The critical points can therefore be only saddle points or
minima, but not maxima.

\subsubsection{Summary}
%''''''''''''''''''''''

Taken together, the special case in section \ref{sec_N_special_case}
and the general case in section \ref{sec_N_general_case} (without the
special cases) show that the novel objective function
(\ref{eq_objfct_new}) on the Stiefel manifold has only maxima at the
principal eigenvectors. All other critical points are either saddle
points or minima since we can find directions (choices of $\matA$ and
$\matB$) where $\Delta J > 0$.

%###########################################################################
%###########################################################################
%###########################################################################

\section{Derivation of Symmetric PCA Learning Rules}\label{sec_derivation}
%===================================================

In the following we focus only on special cases TwJ and N which in the
second variant (TwJ2, N2) promise to lead to true PCA (rather then
PSA) learning rules. Special case TwC has weight-vector lengths
different from unity in the fixed points, but otherwise behaves
similarly to TwJ; so we do not repeat the derivation for this case.

We suggest two ways for the derivation of learning rules: On the first
way (``short form''), we informally replace fixed-point constants
($\matWnull$, $\matDnull$) by variables ($\matW$, $\matD$). Obviously,
the differential equations obtained in this way have the given
fixed points. On the second way (``long form''), we turn the Lagrange
multipliers from fixed-point constants ($\matBnull$) to variables
($\matB$), insert them into the objective function, and determine the
gradient, as done by \cite{nn_Chatterjee00} (but for a different
constraint term).

\subsection{Short Form: Derivation from Fixed-Point Equations}
\label{sec_derivation_short}
%-------------------------------------------------------------

For special case TwJ, we only look at the second variant (TwJ2). We
turn (\ref{eq_TwJ2}) into the differential equation
\begin{eqnarray}\label{eq_ode_TwJ2S}
  \mbox{TwJ2S: } &&
  \tau \matWdot = \matC \matW \matTheta - \matW \matTheta \matW^T \matC \matW.
\end{eqnarray}
which is rule (15a) suggested by \cite{nn_Xu93}. In our nomenclature,
we refer to this rule as TwJ2S (``S'' for ``short form'').

For special case N2, we turn (\ref{eq_N2}) into the differential
equation
\begin{eqnarray}\label{eq_ode_N2S} 
  \mbox{N2S: } &&
  \tau \matWdot = \matC \matW \matD - \matW \matD \matW^T \matC \matW
  \;\;\mbox{with}\;\;
  \matD = \Diag{j=1}{m}\{\vecw_j^T \matC \vecw_j\}.
\end{eqnarray}
We refer to this rule as N2S. We see the structural similarity of the
two learning rules, TwJ2S and N2S. Learning rule (\ref{eq_ode_TwJ2S})
requires the introduction of an additional diagonal matrix $\matTheta$
with pairwise different elements to break the symmetry and turn a PSA
into a PCA rule. In learning rule (\ref{eq_ode_N2S}), the diagonal
matrix $\matD$ automatically appears in the derivation from the novel
objective function (\ref{eq_objfct_new}). In the fixed points, $\matD$
also has pairwise different entries (as has $\matTheta$).

\subsection{Long Form: Derivation by Insertion into Objective Function}
\label{sec_derivation_long}
%----------------------------------------------------------------------

We defined the modified objective function (\ref{eq_Jstar_gen}), which
for $\varOmega_j = 1$ (given in the special cases TwJ and N) turns
into
\begin{equation}\label{eq_Jstar_long}
  J^*
  =
  J
  + \half \summe{j=1}{m} \summe{k=1}{m}
  \half (\beta_{jk} + \beta_{kj}) (\vecw_j^T \vecw_k - \delta_{jk}),
\end{equation}
insert an expression for $\half (\beta_{jk} + \beta_{kj})$, and
compute the gradient. We start at (\ref{eq_fp_pca_vec}), to which we
apply the same non-equivalent transformation, but for individual
$\half (\betanull_{lj} + \betanull_{lj})$:
\begin{eqnarray}
  \vecmnull_l
  &=&
  -\summe{j=1}{m}
  \half (\betanull_{jl} + \betanull_{lj}) \vecwnull_j\\
  \vecwnull_k^T \vecmnull_l
  &=&
  -\summe{j=1}{m}
  \half (\betanull_{jl} + \betanull_{lj}) \vecwnull_k^T \vecwnull_j\\
  \vecwnull_k^T \vecmnull_l
  &=& -\summe{j=1}{m}
  \half (\betanull_{jl} + \betanull_{lj}) \delta_{jk}\\
  \vecwnull_k^T \vecmnull_l
  &=& -\half (\betanull_{kl} + \betanull_{lk}).
\end{eqnarray}
Informally replacing fixed-point constants ($\vecwnull_k$,
$\betanull_{lk}$) by variables ($\vecw_k$, $\beta_{lk}$), we get
\begin{equation}\label{eq_half_betakl_betalk}
\half (\beta_{kl} + \beta_{lk}) = -\vecw_k^T \vecm_l.
\end{equation}
Again we apparently have two different ways to proceed as in section
\ref{sec_fp_constrained_opt}. Here this manifests as two choices we
have for inserting (\ref{eq_half_betakl_betalk}) into
(\ref{eq_Jstar_long}). If we equate $(l,k)$ from
(\ref{eq_half_betakl_betalk}) with $(j,k)$ from (\ref{eq_Jstar_long}),
we get
\begin{equation}\label{eq_lk_jk}
\half (\beta_{kj} + \beta_{jk}) = -\vecw_k^T \vecm_j.
\end{equation}
If we equate $(k,l)$ from (\ref{eq_half_betakl_betalk}) with $(j,k)$
from (\ref{eq_Jstar_long}), we get
\begin{equation}\label{eq_kl_jk}
\half (\beta_{jk} + \beta_{kj}) = -\vecw_j^T \vecm_k.
\end{equation}
We could insert either of these equations into
(\ref{eq_Jstar_long}). However, by exchanging the summation indices
$j$ and $k$ in (\ref{eq_Jstar_long}) and re-arranging the sums we
arrive at exactly the same constraint term, thus cases
(\ref{eq_lk_jk}) and (\ref{eq_kl_jk}) coincide. We therefore only
proceed with (\ref{eq_lk_jk}).

\subsubsection{General Derivation}
%'''''''''''''''''''''''''''''''''

We insert (\ref{eq_lk_jk}) into $J^*$ and simplify:
\begin{eqnarray}
  J^*
  &=&
  J
  + \half \summe{j=1}{m} \summe{k=1}{m}
  \half (\beta_{jk} + \beta_{kj}) (\vecw_j^T \vecw_k - \delta_{jk})\\
  J^*
  &=&
  J
  - \half \summe{j=1}{m} \summe{k=1}{m}
  (\vecw_k^T \vecm_j) (\vecw_j^T \vecw_k - \delta_{jk})\\
  J^*
  &=&
  J
  + \half \summe{j=1}{m} \summe{k=1}{m}
  \vecw_k^T \vecm_j \delta_{jk}
  -\half \summe{j=1}{m} \summe{k=1}{m}
  (\vecw_k^T \vecm_j) (\vecw_j^T \vecw_k)\\
  J^*
  &=&
  J
  + \half \summe{j=1}{m}
  \vecw_j^T \vecm_j
  -\half \summe{j=1}{m} \summe{k=1}{m}
  (\vecw_k^T \vecm_j) (\vecw_j^T \vecw_k).
\end{eqnarray}
We now compute the gradient, applying \lemmaref{\ref{lemma_scalprod_deriv}}
to compute vector derivatives of scalar products:
\begin{eqnarray}
  \ddf{J^*}{\vecw_l}  
  &=&
  \vecm_l\nonumber\\
  &+&
  \half \summe{j=1}{m}
  \left(\delta_{jl} \vecm_j + \ddf{\vecm_j}{\vecw_l} \vecw_j\right)\nonumber\\
  &-&
  \half \summe{j=1}{m} \summe{k=1}{m}
  (\vecw_k^T \vecm_j)
  (\delta_{jl} \vecw_k + \delta_{kl} \vecw_j)\nonumber\\
  &-&
  \half \summe{j=1}{m} \summe{k=1}{m}
  (\vecw_j^T \vecw_k)
  \left(\delta_{kl} \vecm_j + \ddf{\vecm_j}{\vecw_l} \vecw_k\right)\\[5mm]
  %--------------
  &=&
  \vecm_l\nonumber\\
  &+&
  \half \summe{j=1}{m}
  \left(\delta_{jl} \vecm_j + \matH_j \delta_{jl} \vecw_j\right)\nonumber\\
  &-&
  \half \summe{k=1}{m}
  (\vecw_k^T \vecm_l) \vecw_k\nonumber\\
  &-&
  \half \summe{j=1}{m}
  (\vecw_l^T \vecm_j) \vecw_j\nonumber\\
  &-&
  \half \summe{j=1}{m} \summe{k=1}{m}
  (\vecw_j^T \vecw_k)
  \left(\delta_{kl} \vecm_j + \matH_j \delta_{jl} \vecw_k\right)\\[5mm]
  %--------------
  &=&
  \vecm_l\nonumber\\
  &+&
  \half (\vecm_l + \matH_l \vecw_l)\nonumber\\
  &-&
  \half \summe{k=1}{m}
  (\vecw_k^T \vecm_l) \vecw_k\nonumber\\
  &-&
  \half \summe{j=1}{m}
  (\vecw_l^T \vecm_j) \vecw_j\nonumber\\
  &-&
  \half \summe{j=1}{m}
  (\vecw_j^T \vecw_l) \vecm_j\nonumber\\ 
  &-&
  \half \summe{k=1}{m}
  (\vecw_l^T \vecw_k) \matH_l \vecw_k\\[5mm]
  %--------------
  &=&
  \frac{3}{2} \vecm_l + \half \matH_l \vecw_l\nonumber\\
  &-&
  \half \summe{k=1}{m}
  \vecw_k \vecw_k^T \vecm_l\nonumber\\
  &-&
  \half \summe{j=1}{m}
  \vecw_j \vecm_j^T \vecw_l\nonumber\\
  &-&
  \half \summe{j=1}{m}
  \vecm_j \vecw_j^T \vecw_l\nonumber\\ 
  &-&
  \half \summe{k=1}{m}
  \matH_l \vecw_k \vecw_k^T \vecw_l\\[5mm]
  %--------------
  &=&
  \frac{3}{2} \vecm_l + \half \matH_l \vecw_l\nonumber\\
  &-&
  \half \matW \matW^T \vecm_l\nonumber\\
  &-&
  \half \summe{j=1}{m}
  (\vecw_j \vecm_j^T +  \vecm_j \vecw_j^T) \vecw_l\nonumber\\
  &-&
  \half \matH_l \matW \matW^T \vecw_l\label{eq_long_generic_first}.
\end{eqnarray}

\subsubsection{Derivation for Special Case TwJ}
%''''''''''''''''''''''''''''''''''''''''''''''

For special case TwJ, we have $\vecm_l = \matC \vecw_l \theta_l$ and
$\matH_l = \matC \theta_l$ (see section \ref{sec_tradof}). If we
insert this into (\ref{eq_long_generic_first}), we obtain
\begin{eqnarray}
  \ddf{J^*}{\vecw_l}
  &=&
  \frac{3}{2} \vecm_l + \half \matH_l \vecw_l\nonumber\\
  &-&
  \half \matW \matW^T \vecm_l\nonumber\\
  &-&
  \half \summe{j=1}{m}
  (\vecw_j \vecm_j^T +  \vecm_j \vecw_j^T) \vecw_l\nonumber\\
  &-&
  \half \matH_l \matW \matW^T \vecw_l\\[5mm]
  %-------------
  &=&
  \frac{3}{2} \matC \vecw_l \theta_l + \half \matC \vecw_l \theta_l\nonumber\\
  &-&
  \half \matW \matW^T \matC \vecw_l \theta_l\nonumber\\
  &-&
  \half \summe{j=1}{m}
  (\vecw_j \theta_j \vecw_j^T \matC
  + \matC \vecw_j \theta_j \vecw_j^T) \vecw_l\nonumber\\
  &-&
  \half \matC \matW \matW^T \vecw_l \theta_l\\[5mm]
  %-------------
  &=&
  2 \matC \vecw_l \theta_l\nonumber\\
  &-&
  \half \matW \matW^T \matC \vecw_l \theta_l\nonumber\\
  &-&
  \half
  (\matW \matTheta \matW^T \matC
  + \matC \matW \matTheta \matW^T) \vecw_l\nonumber\\
  &-&
  \half \matC \matW \matW^T \vecw_l \theta_l
\end{eqnarray}
In matrix form we get the following learning rule with $\tau\matWdot =
2 \ddf{J^*}{\matW}$:
\begin{eqnarray}
  \mbox{TwJL: } &&\nonumber\\
  \tau \matWdot
  &=& 
  4 \matC \matW \matTheta\label{eq_ode_TwJL}\\
  &-&
   (\matW \matW^T \matC \matW \matTheta
  + \matW \matTheta \matW^T \matC \matW
  + \matC \matW \matTheta \matW^T \matW
  + \matC \matW \matW^T \matW \matTheta)\nonumber.
\end{eqnarray}
This learning rule (to which we refer as ``TwJL'' where ``L''
indicates the ``long form'') shows similarities to rule (15b)
suggested by \cite{nn_Xu93},
\begin{equation}\label{eq_ode_Xu93_15b}
  \tau \matWdot
  =
  2 \matC \matW \matTheta
  - (\matW \matTheta \matW^T \matC \matW
    +\matC \matW \matTheta \matW^T \matW),
\end{equation}
sharing the second and third negative term. Interestingly, combining
$2\matC\matW\matTheta$ with the remaining first and fourth negative
term alone would not lead to a PCA rule but to a PSA rule since, in
the fixed-point equation, $\matTheta$ could be removed from all
terms. However, these terms are useful in the learning rule to push
the weight matrix towards the principal subspace and/or towards
orthonormality. This is revealed if we transform (\ref{eq_ode_TwJL})
into special case T by choosing $\matTheta = \matI_m$ which gives
\begin{equation}\label{eq_ode_T}
  \mbox{TL:}\quad
  \tau \matWdot
  =
  4 \matC \matW\nonumber
  - (2 \matW \matW^T \matC \matW + 2 \matC \matW \matW^T \matW).
\end{equation}
This rule (to which we refer as ``TL'') is identical to the LMSER rule
introduced by \citet[their eqn. (10b)]{nn_Xu93} which is known to be a
subspace rule. If we only take the first and fourth negative term from
(\ref{eq_ode_TwJL}) combined with the positive term (only with
coefficient $2$) and factor out $\matTheta$, we also obtain the TL /
LMSER rule.

\subsubsection{Derivation for Special Case N}
%''''''''''''''''''''''''''''''''''''''''''''

For special case N, we have $\vecm_l = \matC \vecw_l [\vecw_l^T \matC
  \vecw_l]$ and $\matH_l = \matC\vecw_l\vecw_l^T\matC + \matC
[\vecw_l^T\matC\vecw_l]$ (see section \ref{sec_newof}); scalar
expressions are indicated by square brackets below. If we insert this
into (\ref{eq_long_generic_first}), we obtain
\begin{eqnarray}
  \ddf{J^*}{\vecw_l}
  &=&
  \frac{3}{2} \vecm_l + \half \matH_l \vecw_l\nonumber\\
  &-&
  \half \matW \matW^T \vecm_l\nonumber\\
  &-&
  \half \summe{j=1}{m}
  (\vecw_j \vecm_j^T +  \vecm_j \vecw_j^T) \vecw_l\nonumber\\
  &-&
  \half \matH_l \matW \matW^T \vecw_l\\[5mm]
  %----------------  
  &=&
  \frac{3}{2} \matC \vecw_l [\vecw_l^T \matC \vecw_l]\nonumber\\
  &+&
  \half (\matC\vecw_l\vecw_l^T\matC + \matC [\vecw_l^T\matC\vecw_l]) \vecw_l
  \nonumber\\
  &-&
  \half \matW \matW^T \matC \vecw_l [\vecw_l^T \matC \vecw_l]\nonumber\\
  &-&
  \half \summe{j=1}{m}
  (\vecw_j [\vecw_j^T \matC \vecw_j] \vecw_j^T \matC 
  +  \matC \vecw_j [\vecw_j^T \matC \vecw_j] \vecw_j^T) \vecw_l\nonumber\\
  &-&
  \half (\matC\vecw_l\vecw_l^T\matC +
  \matC [\vecw_l^T\matC\vecw_l]) \matW \matW^T \vecw_l\\[5mm]
  %-----------------
  &=&
  \frac{5}{2} \matC \vecw_l [\vecw_l^T \matC \vecw_l]\nonumber\\
  &-&
  \half \matW \matW^T \matC \vecw_l [\vecw_l^T \matC \vecw_l]\nonumber\\
  &-&
  \half \matW \matD \matW^T \matC \vecw_l\nonumber\\
  &-&
  \half \matC \matW \matD \matW^T \vecw_l\nonumber\\
  &-&
  \half \matC \vecw_l [\vecw_l^T \matC \matW \matW^T \vecw_l]\nonumber\\
   &-&
  \half \matC \matW \matW^T \vecw_l [\vecw_l^T\matC\vecw_l].
\end{eqnarray}
With $\tau\matWdot = 2 \ddf{J^*}{\matW}$ we get the learning rule in
matrix form:
\begin{eqnarray}
  \mbox{NL: } &&\nonumber\\
  \tau \matWdot
  &=&
  5 \matC \matW \matD\nonumber\\
  &-&  (\matW \matW^T \matC \matW \matD + \matW \matD \matW^T \matC \matW
  \nonumber\\
  &&   +\matC \matW \matD \matW^T \matW + \matC \matW \matD^*
  +\matC \matW \matW^T \matW \matD).
  \label{eq_ode_NL}
\end{eqnarray}
where
\begin{eqnarray}
  \label{eq_matD}
  \matD &=& \Diag{j=1}{m}\{\vecw_j^T \matC \vecw_j\}\\
  \label{eq_matDstar}
  \matD^* &=& \Diag{j=1}{m}\{\vecw_j^T \matC \matW \matW^T \vecw_j\}.
\end{eqnarray}
We refer to this learning rule as NL. We see similarities with
(\ref{eq_ode_TwJL}), except for the term $\matC\matW\matD^*$;
whether this term is required to produce PCA behavior instead of PSA
behavior needs to be analyzed.

\subsection{Derivation from Gradient on Stiefel Manifolds}
\label{sec_derivation_stiefel}
%---------------------------------------------------------

Differential equations (learning rules) can also be derived from the
gradient on Stiefel manifolds. Two different approaches are known (see
\lemmaref{\ref{lemma_stiefel_gradients}}). In the ``embedded metric'',
  we can write (\ref{eq_stiefel_gradient_embedded}) as a differential
  equation
\begin{equation}\label{eq_stiefel_gradient_embedded_J_W}
    \tau \matWdot
    = \matG_J(\matW)
    - \half \matW (\matW^T \matG_J(\matW) + \matG^T_J(\matW) \matW).
  \end{equation}
In the ``canonical metric'', we can write
(\ref{eq_stiefel_gradient_canonical}) as a differential equation
\begin{equation}\label{eq_stiefel_gradient_canonical_J_W}
  \tau \matWdot
  = \matG_J(\matW)
  - \matW \matG^T_J(\matW) \matW.
\end{equation}
In both equations $ \matG_J(\matW)$ denotes the gradient of the
objective function (without consideration of the Stiefel constraint).

\subsubsection{Derivation for Traditional Objective Function}
%''''''''''''''''''''''''''''''''''''''''''''''''''''''''''''

For the traditional objective function, we have with (\ref{eq_tradof_matM}):
\begin{equation}
  \matG_J(\matW) = \ddf{J}{\matW} = \matM =
  \matC \matW \matTheta.
\end{equation}
Note that we have to assume the constraint $\matW^T\matW = \matOmega =
\matI$ here to stay on the Stiefel manifold. For the ``embedded
metric'' (\ref{eq_stiefel_gradient_embedded_J_W}) we get
\begin{equation}\label{eq_ode_TSE}
  \mbox{TSE:} \quad
  \tau \matWdot = \matC \matW \matTheta
  - \half (\matW \matW^T \matC \matW \matTheta
  + \matW \matTheta \matW^T \matC \matW)
\end{equation}
where ``TSE'' indicates ``traditional, Stiefel, embedded''. For the
``canonical metric'' (\ref{eq_stiefel_gradient_canonical_J_W}) we get
\begin{equation}\label{eq_ode_TSC}
  \mbox{TSC:} \quad
  \tau \matWdot = \matC \matW \matTheta
  - \matW \matTheta \matW^T \matC \matW
\end{equation}
where ``TSC'' indicates ``traditional, Stiefel, canonical''.

\subsubsection{Derivation for Novel Objective Function}
%''''''''''''''''''''''''''''''''''''''''''''''''''''''

For the novel objective function, we have with (\ref{eq_newof_matM}):
\begin{equation}
  \matG_J(\matW) = \ddf{J}{\matW} = \matM =
  \matC \matW \matD, \quad \matD = \Diag{j=1}{m}\{\vecw_j^T \matC \vecw_j\}.
\end{equation}
For the ``embedded metric'' (\ref{eq_stiefel_gradient_embedded_J_W})
we get
\begin{equation}\label{eq_ode_NSE}
  \mbox{NSE:} \quad
  \tau \matWdot = \matC \matW \matD
  - \half (\matW \matW^T \matC \matW \matD
  + \matW \matD \matW^T \matC \matW)
\end{equation}
where ``NSE'' indicates ``novel, Stiefel, embedded''. For the
``canonical metric'' (\ref{eq_stiefel_gradient_canonical_J_W}) we get
\begin{equation}\label{eq_ode_NSC}
  \mbox{NSC:} \quad
  \tau \matWdot = \matC \matW \matD
  - \matW \matD \matW^T \matC \matW
\end{equation}
where ``NSC'' indicates ``novel, Stiefel, canonical''.

\subsection{Comparison of Terms}\label{sec_comparison_terms}
%--------------------------------

Table \ref{tab_comparison_terms} gives an overview of the negative
terms appearing in the different differential equations. We see that
$\matW\matD\matW^T\matC\matW$ appears in all PCA rules (but not in the
two PSA rules) so this is the crucial PCA term. The terms
$\matW\matW^T\matC\matW\matD$ and $\matC\matW\matW^T\matW\matD$ (each
alone or together, but without other negative terms) will not lead to
PCA rules since $\matD$ appears at the last position in all terms and
could therefore be factored out from the entire differential equation,
e.g.
\begin{align}
  \tau \matWdot
  &=
  2 \matC\matW\matD -\matW\matW^T\matC\matW\matD -\matC\matW\matW^T\matW\matD\\
  &=
  (2 \matC\matW - \matW\matW^T\matC\matW - \matC\matW\matW^T\matW) \matD
\end{align}
where the fixed points are the same as that of the LMSER rule which
performs PSA. They are apparently useful (but not required if the PCA
term is present) to push the state towards the principal subspace. The
term $\matC\matW\matD\matW^T\matW$ alone will not produce PCA
behavior, since
\begin{align}
  \tau \matWdot
  &=
  \matC\matW\matD - \matC\matW\matD\matW^T\matW\\
  &=
  \matC\matW\matD (\matI - \matW^T \matW)
\end{align}
so this equation has fixed points on the entire Stiefel manifold. It
is apparently just pushing the state towards the Stiefel manifold. The
role of the term $\matC\matW\matD^*$ has to be explored.

\newcommand{\mctwo}[1]{\multicolumn{2}{c|}{#1}}

\begin{table}[t]
\begin{center}
  \caption{Overview of negative terms (number $n_t$) which are present
    in addition to the positive term $n_t \matC\matW\matD$ in the
    different differential equations. $\matD$ is either taken from
    (\ref{eq_matD}) or represents $\matTheta$. $\matD^*$ is taken from
    (\ref{eq_matDstar}). Rules above the double line perform PCA. The
    differential equations are TwJ2S (\ref{eq_ode_TwJ2S}), N2S
    (\ref{eq_ode_N2S}), TwJL (\ref{eq_ode_TwJL}), NL
    (\ref{eq_ode_NL}), TSE (\ref{eq_ode_TSE}), NSE (\ref{eq_ode_NSE}),
    TSC (\ref{eq_ode_TSC}), NSC (\ref{eq_ode_NSC}). Rules Xu93~(15a)
    and Xu93~(15b) are taken from \cite{nn_Xu93}. Rules below the
    double line perform PSA, not PCA.  TwJ1S could be derived from
    (\ref{eq_TwJ1}), N1S from (\ref{eq_N1}). Oja Subspace was
    described by \cite{nn_Oja89}, LMSER by \cite{nn_Xu93}; in these
    rules we have $\matD = \matI$, so the first two terms coincide as
    well as the second two terms.}
\label{tab_comparison_terms}
\begin{tabular}{|l||c|c|c|c|c|c|}\hline
  & $n_t$
  & {\scriptsize $\matW\matW^T\matC\matW\matD$}
  & {\scriptsize $\matW\matD\matW^T\matC\matW$}
  & {\scriptsize $\matC\matW\matD\matW^T\matW$}
  & {\scriptsize $\matC\matW\matW^T\matW\matD$}
  & {\scriptsize $\matC\matW\matD^*$}\\\hline\hline
  TwJ2S, N2S     & 1 &   & x &   &   &   \\\hline
  TwJL           & 4 & x & x & x & x &   \\\hline
  NL             & 5 & x & x & x & x & x \\\hline
  TSE, NSE       & 2 & x & x &   &   &   \\\hline
  TSC, NSC       & 1 &   & x &   &   &   \\\hline
  Xu93 (15a)     & 1 &   & x &   &   &   \\\hline
  Xu93 (15b)     & 2 &   & x & x &   &   \\\hline\hline
  (TwJ1S), (N1S) & 1 & x &   &   &   &   \\\hline
  Oja Subspace   & 1 & \mctwo{x} & \mctwo{}& \\\hline
  LMSER          & 2 & \mctwo{x} & \mctwo{x}& \\\hline\hline
  role           &   & PSA & PCA & constraint & PSA & ? \\\hline
\end{tabular}
\end{center}
\end{table}

%#############################################################################
%#############################################################################
%#############################################################################

\section{Unresolved Issues}
%==========================

\subsection{Major Issues}
%------------------------

The derivation of fixed points in section \ref{sec_fp} relies on a
somewhat awkward application of \lemmaref{\ref{lemma_blockdiag_evec}}
in reverse direction. There may be a more elegant solution. Moreover,
parts of the analysis of the different special cases can probably be
fused in a more general treatment.

As discussed in section \ref{sec_variants_discussion}, we presently
have no explanation why we obtain the ``uninteresting'' special cases
TwJ1, TwC1, and N1 and the ``interesting'' special cases TwJ2, TwC2,
and N2 from different choices of the Lagrange parameters. It seems to
be crucial to exploit their symmetry, but it is unknown why this is
the case. The influence of the non-equivalent transformation required
to isolate the Lagrange parameters on the appearance of spurious
solutions is also not clear.

The proof of \lemmaref{\ref{lemma_RTDR_ii_b_i}} is weak and needs to
be improved. However, the lemma is not used in a critical step but
only in the analysis of the solution space of the fixed points in
section \ref{sec_fp}.

We currently only analyze the behavior of the constrained objective
functions at the critical points (section \ref{sec_behavior}). It is
not clear how the insights gained from this can be transferred to the
convergence behavior of the learning rules (see section
\ref{sec_behavior_intro}).

It needs to be explored how the different learning rules behave under
deviations from the Stiefel manifold. Even if a step starts on the
Stiefel manifold, all learning rules probably move away from the
Stiefel manifold in the step. Whether some rules contain terms which
reduce this deviation, or whether generally a back-projection is
required (exact as in \lemmaref{\ref{lemma_stiefel_svd}} or
approximated for small steps as in
\lemmaref{\ref{lemma_stiefel_proj_appr}}) has to be explored.

Learning rule (15b) by \cite{nn_Xu93} could not be derived by one of
the three methods in section \ref{sec_derivation}. Rule TwJL contains
the two terms of this rule, but two additional terms. This indicates
that there is yet another way to derive learning rules, but it is
unclear whether this is possible in the Lagrange-multiplier scheme
described here.

\subsection{Minor Issues}
%------------------------

The overall fixed-point analysis in section \ref{sec_overall_fp} is
inconclusive for special case N, but it is doubtful whether this can
be resolved since additional fixed points outside the principal
eigenvectors are known to exist (section \ref{sec_N2}), but are
irrelevant for the learning rules since they are probably not
attractors (section \ref{sec_N}).

In the analysis of the solutions of the different special cases in
section \ref{sec_fp} we found that some fixed points are spurious
solutions. In particular, we looked at the rotation matrices contained
in the solutions and argued that the critical function is only maximal
for certain choices of those matrices. In all other cases, the
rotation matrices can be modified to increase the value of the
critical function, thus we can't be at a fixed point. There may be a
gap in this argumentation since it is not clear whether a {\em
  continuous} modification towards larger values is
possible. Moreover, spurious solutions could also manifest themselves
in the permutation matrices; this was not analyzed.

The space of solutions of the orthogonal block matrices $\matU_l$
which fulfill the constraint for special case N2 (see section
\ref{sec_N2_constraint}) is not known. We only know that Hadamard
matrices (with a factor) are one possible choice.

In section \ref{sec_N2_test}, we test the solution for special case
N2. While the block-diagonal shape of $\matU$ needs to be taken into
account, the solution surprisingly holds without considering the
additional constraint from section \ref{sec_N2_constraint}.

The role of the term $\matC \matW \matD^*$ which appears in the
learning rule NL (\ref{eq_ode_NL}) needs to be explored (see section
\ref{sec_comparison_terms}).

Throughout the work, diagonal sign matrices and permutation matrices
often appear together in products. It may be possible to simplify this
by defining signed permutation matrices. Alternatively, there may be a
consistent way where diagonal sign matrices are omitted completely and
standard permutation matrices are used. Moreover, the fusion of
diagonal sign matrices and of permutation matrices should be treated
once in a lemma instead of multiple times in section \ref{sec_fp}.

%#############################################################################
%#############################################################################
%#############################################################################

\section{Conclusion}\label{sec_conclusion}
%===================

We could demonstrate that symmetric learning rules for principal
component analysis can be obtained without resorting to fixed weight
factor matrices in the objective function. To accomplish this, a novel
objective function had to be introduced which contains squared terms
compared to the traditional (non-weighted) function. The novel
objective function has a complex set of critical points which not only
includes principal eigenvectors. However, the analysis of the behavior
of the objective function at the critical points shows that local
maxima only occur at the principal eigenvectors. This allows to derive
novel symmetrical learning rules from the objective function which
should exhibit PCA behavior (confirmed in preliminary
simulations). Learning rules derived from the traditional objective
function require {\em fixed} diagonal weight-factor matrices to
perform PCA instead of just PSA. In the novel learning rules, {\em
  variable} diagonal matrices appear in the same place without any
additional assumptions. There is a major difference in the behavior of
the learning rules: Rules with fixed weight factors produce a fixed
order of the principal eigenvectors with respect to the
corresponding eigenvalues whereas rules derived from the novel
objective function can converge towards any permutation of the
principal eigenvectors, depending on the initial weight
matrix.\footnote{Note that in the averaged rules considered here, the
  learning process is deterministic; for a given initial weight
  matrix, a specific permutation is approached. In online learning,
  data vectors are typically presented in random order, so we expect
  that any permutation could be approached, even with the same initial
  weight matrix.}

In section \ref{sec_fp}, fixed-point equations are obtained by
eliminating the Lagrange multipliers, and in section
\ref{sec_derivation}, essentially the same technique is used to derive
learning rules. We currently do not see an alternative way to
determine the fixed points. However, learning rules could also be
obtained from the Lagrange-Newton framework \cite[see
  e.g.][sec. 5.2]{own_Moeller20}. In this framework, the Lagrange
multipliers are not eliminated, but treated as parameters in the
numerical optimization. A Newton descent (instead of a gradient
method) is required in this case, since the solutions of the Lagrange
equations are typically saddle points. We think that the novel
objective function may be a promising starting point for the
derivation of coupled learning rules
\cite[]{own_Moeller04a,own_Kaiser10} from the Lagrange-Newton
framework. In coupled rules, eigenvectors and eigenvalues are
estimated simultaneously. They are a solution to the speed-stability
problem which probably also affects all learning rules derived in this
work. For the simple single-component PCA analyzer mentioned by
\cite{own_Moeller20}, Lagrange multipliers coincide with the
eigenvalue estimates; whether this is also the case for symmetrical
learning rules has to be explored.

The behavior of the learning rules, particularly convergence speed, is
difficult to predict from the analysis in this work. One could argue
that the rules derived from the novel objective function may converge
faster: Since no order is imposed, each weight vector can converge
to the {\em closest} eigenvector (but this speed advantage is not
confirmed in preliminary simulations). In the opposite direction, the
additional (unstable) fixed points of the novel rules could slow down
the learning. It is also unclear how the choice of fixed weight
factors in the traditional rules affects convergence speed: If
$\matTheta$ contains a good guess of the true eigenvalues (which
$\matD$ represents in the stable fixed points), the traditional rules
may converge faster than the novel rules since the latter first have
to approach a suitable matrix $\matD$. Moreover, the role of the
additional PSA and back-projection terms in the longer learning rules
is difficult to anticipate. Preliminary simulations show that the
rules with more terms converge more slowly than those with more terms,
at least with the same learning rate (which, however, may not be
optimal for all rules). More work has to be invested into simulations
to explore these influences.

Our preliminary simulations reveal one major disadvantage of the novel
learning rules. Both the old and the novel learning rules show slower
convergence if two of the $m$ principal eigenvalues of the covariance
matrix are close to each other. This is due to the fact that, in the
extreme case of identical eigenvalues, the eigenvectors are not
unique. The closer we get to this extreme case, the slower the
convergence to the eigenvectors. The novel learning rules seem suffer
from an additional disadvantage in this situation. While the old
learning rules have fixed weight-factor matrices ($\matTheta$), these
are replaced by the variable diagonal matrices ($\matD$) in the novel
learning rules. Close to the stable fixed points, the entries of
$\matD$ are estimates of the eigenvalues. If the eigenvalues are close
to each other, so will be the diagonal entries of $\matD$. Therefore
the ``symmetry-breaking'' effect of $\matD$ is reduced, approaching
PSA instead of PCA behavior, which ultimately leads to particularly
slow convergence. This observation relates to the existence of
additional fixed points described in section \ref{sec_N2}. For the
case of two nearby eigenvalues, the novel learning rule operates in
the vicinity of these fixed points, thus the progress is slow.

%#############################################################################
%#############################################################################
%#############################################################################

\section*{Acknowledgments}
%=========================
\label{sec_acknowledgments}
\addcontentsline{toc}{section}{\nameref{sec_acknowledgments}}

I'm very grateful for help by Axel Könies and Pierre-Antoine Absil,
and for the proofs provided by Math Stack-Exchange users 'Joppy',
'lcv', and 'user8675309'; their contributions are mentioned in the
text. My derivations particularly profited from techniques described
by \cite{nn_Xu93}.

%#############################################################################
%#############################################################################
%#############################################################################

%\newpage
% homes <-> home
\trarxiv{%
  \bibliographystyle{/home/moeller/bst/plainnatsfnnm}
  \bibliography{/home/moeller/bib/nn17,/home/moeller/bib/own14}

\begin{thebibliography}{21}
\expandafter\ifx\csname natexlab\endcsname\relax\def\natexlab#1{#1}\fi

\bibitem[Abadir and Magnus(2005)]{nn_Abadir05}
K.~M. Abadir and J.~R. Magnus.
\newblock {\em Matrix Algebra}.
\newblock Cambridge University Press, 2005.

\bibitem[Absil et~al.(2008)Absil, Mahony, and Sepulchre]{nn_Absil08}
P.-A. Absil, R.~Mahony, and R.~Sepulchre.
\newblock {\em Optimization Algorithms on Matrix Manifolds}.
\newblock Princeton University Press, 2008.

\bibitem[Absil and Malick(2012)]{nn_Absil12}
P.-A. Absil and J.~Malick.
\newblock Projection-like retractions on matrix manifolds.
\newblock {\em SIAM Journal on Optimization}, 22\penalty0 (1):\penalty0
  135--158, 2012.

\bibitem[Bhattacharya and Bhattacharya(2012)]{nn_Bhattacharya_12}
A.~Bhattacharya and R.~Bhattacharya.
\newblock {\em Nonparametric Inference on Manifolds: With Applications to Shape
  Spaces}.
\newblock Institute of Mathematical Statistics Monographs. Cambridge University
  Press, 2012.

\bibitem[Chatterjee et~al.(2000)Chatterjee, Kang, and
  Roychowdhury]{nn_Chatterjee00}
C.~Chatterjee, Z.~Kang, and V.~P. Roychowdhury.
\newblock Algorithms for accelerated convergence of adaptive {PCA}.
\newblock {\em IEEE Transactions on Neural Networks}, 11\penalty0 (2):\penalty0
  338--355, 2000.

\bibitem[Diamantaras and Kung(1996)]{nn_Diamantaras96}
K.~I. Diamantaras and S.~Y. Kung.
\newblock {\em Principal Component Neural Networks. Theory and Applications}.
\newblock John Wiley \& Sons, 1996.

\bibitem[Edelman et~al.(1998)Edelman, Arias, and Smith]{nn_Edelman98}
A.~Edelman, T.~A. Arias, and S.~T. Smith.
\newblock The geometry of algorithms with orthogonality constraints.
\newblock {\em SIAM Journal on Matrix Analysis and Applications}, 20\penalty0
  (2):\penalty0 303--353, 1998.

\bibitem[Gentle(2017)]{nn_Gentle17}
J.~A. Gentle.
\newblock {\em Matrix Algebra. Theory, Computations and Applications in
  Statistics}.
\newblock Springer, 2nd edition, 2017.

\bibitem[Golub and {van Loan}(1996)]{nn_Golub96}
G.~H. Golub and C.~F. {van Loan}.
\newblock {\em Matrix Computations}.
\newblock Johns Hopkins University Press, Baltimore and London, 3rd edition,
  1996.

\bibitem[Horn and Johnson(1999)]{nn_Horn99}
R.~A. Horn and C.~R. Johnson.
\newblock {\em Matrix Analysis}.
\newblock Cambridge University Press, 1999.

\bibitem[Kaiser et~al.(2010)Kaiser, Schenck, and M\"oller]{own_Kaiser10}
A.~Kaiser, W.~Schenck, and R.~M\"oller.
\newblock Coupled singular value decomposition of a cross covariance matrix.
\newblock {\em International Journal of Neural Systems}, 20\penalty0
  (4):\penalty0 293--318, 2010.

\bibitem[Kong et~al.(2017)Kong, Hu, and Duan]{nn_Kong17}
X.~Kong, C.~Hu, and Z.~Duan.
\newblock {\em Principal Component Analysis Networks and Algorithms}.
\newblock Springer Singapore / Science Press Beijing, 2017.

\bibitem[M\"oller(2020)]{own_Moeller20}
R.~M\"oller.
\newblock Derivation of coupled {PCA} and {SVD} learning rules from a {N}ewton
  zero-finding framework.
\newblock {\em arXiv:2003.11456}, 2020.

\bibitem[M\"oller and K\"onies(2004)]{own_Moeller04a}
R.~M\"oller and A.~K\"onies.
\newblock Coupled principal component analysis.
\newblock {\em IEEE Transactions on Neural Networks}, 15\penalty0 (1):\penalty0
  214--222, 2004.

\bibitem[Oja(1982)]{nn_Oja82}
E.~Oja.
\newblock A simplified neuron model as principal component analyzer.
\newblock {\em Journal of Mathematical Biology}, 15:\penalty0 267--273, 1982.

\bibitem[Oja(1989)]{nn_Oja89}
E.~Oja.
\newblock Neural networks, principal components, and subspaces.
\newblock {\em International Journal of Neural Systems}, 1\penalty0
  (1):\penalty0 61--68, 1989.

\bibitem[Oja(1992)]{nn_Oja92a}
E.~Oja.
\newblock Principal components, minor components, and linear neural networks.
\newblock {\em Neural Networks}, 5\penalty0 (6):\penalty0 927--935, 1992.

\bibitem[Oja et~al.(1992{\natexlab{a}})Oja, Ogawa, and Wangviwattana]{nn_Oja92}
E.~Oja, H.~Ogawa, and J.~Wangviwattana.
\newblock Principal component analysis by homogeneous neural networks,
  \uppercase{P}art \uppercase{I}: \uppercase{T}he weighted subspace criterion.
\newblock {\em IEICE Transactions on Information and Systems}, E75-D\penalty0
  (3):\penalty0 366--375, 1992{\natexlab{a}}.

\bibitem[Oja et~al.(1992{\natexlab{b}})Oja, Ogawa, and
  Wangviwattana]{nn_Oja92b}
E.~Oja, H.~Ogawa, and J.~Wangviwattana.
\newblock Principal component analysis by homogeneous neural networks,
  \uppercase{P}art \uppercase{II}: \uppercase{A}nalysis and extensions of the
  learning algorithms.
\newblock {\em IEICE Transactions on Information and Systems}, E75-D\penalty0
  (3):\penalty0 376--381, 1992{\natexlab{b}}.

\bibitem[Sanger(1989)]{nn_Sanger89}
T.~D. Sanger.
\newblock Optimal unsupervised learning in a single-layer linear feedforward
  neural network.
\newblock {\em Neural Networks}, 2\penalty0 (6):\penalty0 459--473, 1989.

\bibitem[Xu(1993)]{nn_Xu93}
L.~Xu.
\newblock Least mean square error reconstruction principle for self-organizing
  neural nets.
\newblock {\em Neural Networks}, 6:\penalty0 627--648, 1993.

\end{thebibliography}
}{%

}

%#############################################################################
%#############################################################################
%#############################################################################

\section*{Changes}
%=================
\label{sec_changes}
\addcontentsline{toc}{section}{\nameref{sec_changes}}

20 May 2019: Started report.\\
24 May 2020: First version for arXiv.\\
27 May 2020: Added paragraph to section \ref{sec_conclusion} on a
disadvantage of the novel rules.

%#############################################################################
%#############################################################################
%#############################################################################

\appendix

%#############################################################################
% Lemmata
%#############################################################################

\section{Lemmata}
%================

\subsection{Diagonal and Orthogonal Matrices}
%--------------------------------------------

\begin{lemma}\footnote{Used by \citet[p.633]{nn_Xu93}.
    Proof provided here for convenience.}\label{lemma_perm_diag} Let\/
  $\matP$ be a permutation matrix and\/ $\matD$ a diagonal
  matrix. Then $\matP^T \matD \matP$ is a diagonal matrix with
  permuted diagonal elements.
\end{lemma}

\begin{proof}
  Let $\pi(i)$ be the permuted index to index $i$. Then
  \begin{eqnarray}
    \label{eq_perm_diag_P}
    (\matP)_{ij}   &=& \delta_{i,\pi(j)}\\
    \label{eq_perm_diag_PT}
    (\matP^T)_{ij} &=& \delta_{j,\pi(i)}\\
    (\matD)_{ij}   &=& d_i \delta_{ij}.
  \end{eqnarray}
  We form the product $\matP^T\matD$
  \begin{eqnarray}
    (\matP^T\matD)_{ij}
    &=& \summe{k}{} (\matP^T)_{ik} (\matD)_{kj}\\
    &=& \summe{k}{} \delta_{k,\pi(i)} d_k \delta_{kj}\\
    &=& \delta_{j,\pi(i)} d_j
  \end{eqnarray}
  and then multiply by $\matP$:
  \begin{eqnarray}
    ((\matP^T\matD)\matP)_{ij}
    &=& \summe{k}{} (\matP^T\matD)_{ik} (\matP)_{kj}\\
    &=& \summe{k}{} \delta_{k,\pi(i)} d_k \delta_{k,\pi(j)}\\
    &=& d_{\pi(i)} \delta_{\pi(i),\pi(j)}.
  \end{eqnarray}
  We see from $\delta_{\pi(i),\pi(j)}$ that all off-diagonal elements
  are zero (since $\pi(i) \neq \pi(j)$ for $i \neq j$). The diagonal
  elements are $d_{\pi(i)}$ and thus the permuted diagonal entries of
  $\matD$.
\end{proof}

\lemmasep
%-----------------------------------------------------------------------------

\begin{lemma}\label{lemma_perm_diag_sign}
  \lemmaref{\ref{lemma_perm_diag}} can be extended to signed
  permutation matrices: Let\/ $\matP$ be a permutation matrix,
  $\matXi$ be a diagonal sign matrix (with entries $\pm 1$), and\/
  $\matD$ a diagonal matrix. We form a signed permutation matrix
  $\matP' = \matXi \matP$. Then $\matP'^T \matD \matP'$ is a diagonal
  matrix with permuted diagonal elements.
\end{lemma}

\begin{proof}
  The signed permutation matrix is
  \begin{eqnarray}
    (\matP')_{ij}
    &=&
    (\matXi \matP)_{ij}\\
    &=&
    \summe{k}{} (\matXi)_{ik} (\matP)_{kj}\\
    &=&
    \summe{k}{} \xi_i \delta_{ik} \delta_{k,\pi(j)}\\
    &=&
    \xi_i \delta_{i,\pi{j}}.
  \end{eqnarray}  
  Thus we get the following modified expressions for the signed
  permutation matrix:
  \begin{eqnarray}
    (\matP')_{ij}   &=& \xi_i \delta_{i,\pi(j)}\\
    (\matP'^T)_{ij} &=& \xi_j \delta_{j,\pi(i)}.
  \end{eqnarray}
  We form the product $\matP'^T\matD$
  \begin{eqnarray}
    (\matP'^T\matD)_{ij}
    &=& \summe{k}{} (\matP'^T)_{ik} (\matD)_{kj}\\
    &=& \summe{k}{} \xi_k \delta_{k,\pi(i)} d_k \delta_{kj}\\
    &=& \xi_j \delta_{j,\pi(i)} d_j
  \end{eqnarray}
  and then multiply by $\matP$:
  \begin{eqnarray}
    ((\matP'^T\matD)\matP')_{ij}
    &=& \summe{k}{} (\matP'^T\matD)_{ik} (\matP')_{kj}\\
    &=& \summe{k}{} \xi_k \delta_{k,\pi(i)} d_k \xi_k \delta_{k,\pi(j)}\\
    &=& \summe{k}{} \xi^2_k \delta_{k,\pi(i)} d_k \delta_{k,\pi(j)}\\
    &=& \summe{k}{} \delta_{k,\pi(i)} d_k \delta_{k,\pi(j)}\\
    &=& d_{\pi(i)} \delta_{\pi(i),\pi(j)}.
  \end{eqnarray}
  The conclusions are the same as for \lemmaref{\ref{lemma_perm_diag}}.
\end{proof}

\lemmasep
%-----------------------------------------------------------------------------

\begin{lemma}\label{lemma_perm_dg}
Let $\matA$ be matrix of size $n \times n$, and $\matP$ a permutation
matrix of the same size. Then
\begin{equation}
\dg\{\matP^T \matA \matP\} = \matP^T \dg\{\matA\} \matP.
\end{equation}
\end{lemma}

\begin{proof}
  We use (\ref{eq_perm_diag_P}) and (\ref{eq_perm_diag_PT}) and show
  that the left side
  \begin{align}
    (\matP^T\matA)_{ij}
    &=
    \summe{k=1}{n} (\matP^T)_{ik} \matA_{kj}
    =
    \summe{k=1}{n} \delta_{k,\pi(i)} \matA_{kj}
    =
    \matA_{\pi(i),j}\\
    \label{eq_PTAP_ii}
    (\matP^T\matA\matP)_{ii}
    &=
    \summe{j=1}{n} (\matP^T\matA)_{ij} P_{ji}
    =
    \summe{j=1}{n} \matA_{\pi(i),j} \delta_{j,\pi(i)}
    =
    \matA_{\pi(i),\pi(i)}\\
    (\dg\{\matP^T\matA\matP\})_{ij}
    &=
    (\matP^T\matA\matP)_{ii} \delta_{ij}
    =
    \matA_{\pi(i),\pi(i)} \delta_{ij}
  \end{align}
  coincides with the right side:
  \begin{align}
    (\matP^T \dg\{\matA\})_{ik}
    &=
    \summe{l=1}{n} (\matP^T)_{il} (\dg\{\matA\})_{lk}
    =
    \summe{l=1}{n} \delta_{l,\pi(i)} {\matA}_{lk} \delta_{l,k}
    =
    \delta_{k,\pi(i)} {\matA}_{kk}\\
    (\matP^T \dg\{\matA\} \matP)_{ij}
    &=
    \summe{k=1}{n} (\matP^T \dg\{\matA\})_{ik} \matP_{kj}
    =
    \summe{k=1}{n} \delta_{k,\pi(i)} \matA_{kk} \delta_{k,\pi(j)}\\
    &=
    \matA_{\pi(i),\pi(i)} \delta_{\pi(i),\pi(j)}
    =
    \matA_{\pi(i),\pi(i)} \delta_{ij}
  \end{align}
where the last step holds since $\pi(i) = \pi(j)$ if and only if $i =
j$.
\end{proof}

\lemmasep
%-----------------------------------------------------------------------------

\begin{lemma}\label{lemma_orthosim_diag}
  Let\/ $\matD$ be a diagonal matrix with non-zero and pairwise
  different entries. Let $\matQ$ be an orthogonal matrix. If\/
  $\matQ^T \matD \matQ = \matD'$ is diagonal then $\matQ = \matP'$
  where $\matP' = \matXi \matP$ is a signed permutation matrix
  ($\matXi$ is a diagonal sign matrix with entries $\pm 1$, $\matP$ is
  a permutation matrix). Note that, according to
  \lemmaref{\ref{lemma_perm_diag_sign}}, this solution implies that
  $\matD'$ contains the same diagonal elements as $\matD$ but in
  permuted order.
\end{lemma}

\begin{proof}\footnote{From \url{https://math.stackexchange.com/questions/3365505}, proof kindly provided by user 'Joppy'}
  Statement 1: The eigenvectors of the diagonal matrix $\matD$ with
  pairwise different, non-zero entries are the basis vectors
  $\pm\vece_i$, $i=1,\ldots,n$. Proof: The characteristic equation
  $\det\{\matD - \lambda \matI\} = 0$ leads to $\lambda_i = d_i$. The
  eigenvector equation $(\matD - \lambda_i \matI) \vecv_i = \vecnull$
  becomes $(\matD - d_i \matI) \vecv_i = \vecnull$ which leads to the
  solution $(\vecv_i)_i = \pm 1$ and $(\vecv_i)_j = 0$, $\forall j\neq
  i$, and thus $\vecv_i = \pm\vece_i$ which concludes the proof of
  statement 1.

  Statement 2: Assume that $\matQ$ is an invertible matrix
  ($\matQ^{-1} = \matQ^T$ is a special case). If $\matQ^{-1} \matD
  \matQ$ is diagonal, then the columns of $\matQ$ are the eigenvectors
  of $\matD$. Proof: We express $\matQ$ by its columns $\matQ =
  \pmat{\vecq_1 & \ldots & \vecq_n}$. Since $\matQ$ is invertible, the
  columns of $\matQ$ span the entire $\mathbb{R}^n$ (since the $n$
  columns must be linearly independent\footnote{See
    e.g. \url{https://math.stackexchange.com/questions/1058555}}). Thus
  any vector can be expressed as a superposition of the column
  vectors, also the vector $\matD \vecq_i$:
  \begin{equation}\label{eq_orthosim_diag_super}
    \matD \vecq_i = \sum_{j=1}^n a_{ij} \vecq_j.
  \end{equation}
  We look at column $i$ of $\matQ^{-1}\matD\matQ$:
  \begin{eqnarray}
    (\matQ^{-1}\matD\matQ)\vece_i
    &=& \matQ^{-1}\matD\vecq_i\\
    &=& \matQ^{-1}\sum_{j=1}^n a_{ij} \vecq_j\\
    &=& \sum_{j=1}^n a_{ij} \vece_j.
  \end{eqnarray}
  Since we assume that $\matQ^{-1}\matD\matQ$ is diagonal (i.e. equals
  a diagonal matrix $\matD'$ with elements $d'_i$), we know that
  \begin{eqnarray}
    (\matQ^{-1}\matD\matQ)\vece_i &=& d'_i \vece_i\\
    \sum_{j=1}^n a_{ij} \vece_j &=& d'_i \vece_i
  \end{eqnarray}
  thus $a_{ii} = d'_i$ and $a_{ij} = 0$, $\forall j \neq i$, or
  $a_{ij} = \delta_{ij} d'_i$. Therefore
  (\ref{eq_orthosim_diag_super}) turns into
  \begin{equation}
    \matD \vecq_i = d'_i \vecq_i,\quad i = 1,\ldots,n.
  \end{equation}
  We see that the vectors $\vecq_i$, $i=1,\ldots,n$, are the
  eigenvectors of $\matD$, which concludes the proof of statement
  2.

  According to statement 1, the eigenvectors of $\matD$ are the basis
  vectors $\pm \vece_j$, and according to statement 2, the
  eigenvectors of $\matD$ are the columns $\vecq_i$. The assignment
  of $j$ to $i$ can be any permutation $\pi$:
  \begin{equation}
    \vecq_i = \pm\vece_{\pi(i)}
  \end{equation}
  which concludes the proof of the Lemma.
\end{proof}  

\lemmasep
%-----------------------------------------------------------------------------

\begin{lemma}\label{lemma_RTDR_Ds}
  Let\/ $\matD$ be a diagonal matrix with pairwise different diagonal
  entries and\/ $\matR$ an orthogonal matrix. The equation $\matR^T
  \matD \matR$ is a diagonal matrix ($\matD'$) if and only if\/ $\matR
  = \matXi \matP$ where $\matP$ is a permutation matrix and $\matXi$ a
  diagonal sign matrix.
\end{lemma}

\begin{proof}
  \lemmaref{\ref{lemma_orthosim_diag}} demonstrates: If $\matR^T \matD
  \matR = \matD'$ then $\matR = \matXi
  \matP$. \lemmaref{\ref{lemma_perm_diag_sign}} demonstrates: If
  $\matR = \matXi \matP$, then $\matR^T \matD \matR = \matD'$ where
  $\matD'$ is a diagonal matrix (with permuted elements from
  $\matD$). Combining these two directions leads to the
  ``if-and-only-if'' statement.
\end{proof}

\lemmasep
%-----------------------------------------------------------------------------

\begin{lemma}\label{lemma_commute_diag}
  A matrix commuting with a diagonal matrix with pairwise different
  diagonal entries is diagonal: Let\/ $\matD = \diag_{i=1}^{n}\{d_i\}$
  with $d_i \neq d_j$ for $i \neq j$. If the $n \times n$ matrix
  $\matA$ commutes with $\matD$, i.e. $\matA \matD = \matD \matA$,
  then $\matA$ is a diagonal matrix.
\end{lemma}

\begin{proof}\footnote{Slightly modified from \url{https://yutsumura.com} 'A Matrix Commuting With a Diagonal Matrix with Distinct Entries is Diagonal'.}
  The elements of $\matD$ can be expressed as $d_{ij} = \delta_{ij}
  d_i$. We compare entries $(i,j)$ on both sides:
  \begin{eqnarray*}
    (\matA\matD)_{ij}
    &=&
    \summe{k=1}{n} a_{ik} d_{kj}
    =
    \summe{k=1}{n} a_{ik} \delta_{kj} d_k
    =
    a_{ij} d_j\\
    (\matD\matA)_{ij}
    &=&
    \summe{k=1}{n} d_{ik} a_{kj}
    =
    \summe{k=1}{n} \delta_{ik} d_i a_{kj}
    = a_{ij} d_i
  \end{eqnarray*}
  We obtain $a_{ij} d_j = a_{ij} d_i$ or $a_{ij} (d_j - d_i) =
  0$. Since $d_i \neq d_j$ for $i \neq j$, we have $a_{ij} = 0$ for $i
  \neq j$; therefore $\matA$ is diagonal.
\end{proof}

\lemmasep
%-----------------------------------------------------------------------------

\begin{lemma}\label{lemma_commute_blockdiag}
  A matrix commuting with a diagonal matrix where some diagonal
  entries coincide is a block-diagonal matrix: Assume that $\matD =
  \diag_{i=1}^{n}\{d_i\}$ has $k$ distinct diagonal entries $d'_1
  \ldots d'_k$ with $k \leq n$. The diagonal entries are assumed to be
  ordered contiguously (note that here index $p$ of\/ $\matI_p$
  indicates the block index, not the size of the unit matrix):
  \begin{equation}
    \matD
    =
    \pmat{
      d'_1\matI_1 & \matNull   & \ldots & \matNull\\
      \matNull    & d'_2\matI_2 & \ldots & \matNull\\
      \vdots      & \vdots     & \ddots & \vdots\\
      \matNull    & \matNull   & \ldots & d'_k \matI_k}
  \end{equation}
If the $n \times n$ matrix $\matA$ commutes with $\matD$, i.e. $\matA
\matD = \matD \matA$, then $\matA$ is a block-diagonal matrix with
block sizes and locations as in $\matD$ and arbitrary blocks.
\end{lemma}

\begin{proof}\footnote{\label{fn_commute_blockdiag}Slightly modified from ``Diagonalization by a unitary similarity transformation'', \url{http://scipp.ucsc.edu/~haber/archives/physics116A06/diag.pdf}.}
  From the proof of \lemmaref{\ref{lemma_commute_diag}} we consider
  \begin{eqnarray*}
    (\matA\matD)_{ij} &=& a_{ij} d_j\\
    (\matD\matA)_{ij} &=& a_{ij} d_i
  \end{eqnarray*}
  and thus $a_{ij} (d_j - d_i) = 0$. For $d_j \neq d_i$ we get $a_{ij}
  = 0$, but for $d_j = d_i$, $a_{ij}$ is arbitrary. We can write
  $\matA$ as a matrix of $k \times k$ blocks $\matA'_{pq}$ (with $p,q
  \in [1,k]$)
  \begin{equation}
    \matA
    =
    \pmat{
      \matA'_{11} & \matA'_{12} & \ldots & \matA'_{1k}\\
      \matA'_{21} & \matA'_{22} & \ldots & \matA'_{2k}\\
      \vdots     & \vdots     & \ddots & \vdots\\
      \matA'_{k1} & \matA'_{k2} & \ldots & \matA'_{kk}}.
  \end{equation}
  For $p \neq q$ we have $\matA'_{pq} = \matNull$ since $d'_p \neq
  d'_q$. For $p = q$, $\matA'_{pq} = \matA'_{pp}$ is arbitrary. We
  conclude that $\matA$ is block-diagonal:
  \begin{equation}
    \matA
    =
    \pmat{
      \matA'_{11} & \matNull   & \ldots & \matNull\\
      \matNull   & \matA'_{22} & \ldots & \matNull\\
      \vdots     & \vdots     & \ddots & \vdots\\
      \matNull   & \matNull   & \ldots & \matA'_{kk}}.
  \end{equation}
  That the blocks are arbitrary can be shown by considering that the
  matrix products $\matA \matD$ and $\matD \matA$ are block-diagonal:
  \begin{eqnarray}
    \matA \matD &=& \matD \matA\\
    (\matA \matD)_{pp} &=& (\matD \matA)_{pp}\\
    \matA'_{pp} \matD'_{pp} &=& \matD'_{pp} \matA'_{pp}\\
    \matA'_{pp} d'_p \matI_p &=& d'_p \matI_p \matA'_{pp}\\
    \matA_{pp} &=& \matA_{pp}.
  \end{eqnarray}
\end{proof}

\lemmasep
%-----------------------------------------------------------------------------

\begin{lemma}\label{lemma_diagonalization_blockdiag}
  Let\/ $\matA$ be a symmetric, block-diagonal matrix of size $n
  \times n$ with $k$ blocks:
  \begin{equation}
    \matA
    =
    \pmat{
      \matA'_1  & \matNull   & \ldots & \matNull\\
      \matNull  & \matA'_2   & \ldots & \matNull\\
      \vdots    & \vdots     & \ddots & \vdots\\
      \matNull  & \matNull   & \ldots & \matA'_k}.
  \end{equation}
  Then the spectral decomposition of $\matA$
  (\lemmaref{\ref{lemma_spectral}}) is given by
  \begin{equation}
    \matA = \matU \matD \matU^T
  \end{equation}
  where $\matD = \diag_{i=1}^{n}\{d_i\}$ and $\matU$ is a
  block-diagonal orthogonal matrix
  \begin{equation}
    \matU
    =
    \pmat{
      \matU'_1  & \matNull   & \ldots & \matNull\\
      \matNull  & \matU'_2   & \ldots & \matNull\\
      \vdots    & \vdots     & \ddots & \vdots\\
      \matNull  & \matNull   & \ldots & \matU'_k}
  \end{equation}
  where each block $\matU'_i$ is orthogonal as well.
  
  If\/ $\matA$ has pairwise different eigenvalues, the spectral
  decomposition is unique up to the signs of the column vectors of\/
  $\matU$.
\end{lemma}

\begin{proof}\footnote{See footnote \ref{fn_commute_blockdiag}.}
  According to \lemmaref{\ref{lemma_spectral}}, each block of $\matA$
  can be spectrally decomposed individually as $\matA'_i = \matU'_i
  \matD'_i \matU'^T_i$. Uniqueness can be shown by applying the
  uniqueness statement from \lemmaref{\ref{lemma_spectral}} to the
  individual blocks.
\end{proof}
 
\lemmasep
%-----------------------------------------------------------------------------

\begin{lemma}\label{lemma_blockdiag_diag}
  Let\/ $\matU$ be a block-diagonal orthogonal matrix, and\/ $\matD$
  be a diagonal matrix where some entries coincide (which are sorted
  contiguously as show below). Assume that the $k$ blocks of\/ $\matU$
  have the same dimension as the $k$ identity matrices
  in $\matD$. Then
  \begin{equation}
    \underbrace{\pmat{
      \matU'^T_1  & \matNull    & \ldots & \matNull\\
      \matNull   & \matU'^T_2   & \ldots & \matNull\\
      \vdots     & \vdots      & \ddots & \vdots\\
      \matNull   & \matNull    & \ldots & \matU'^T_k}}_{\matU^T}
    \underbrace{\pmat{
        d'_1\matI_1 & \matNull   & \ldots & \matNull\\
        \matNull    & d'_2\matI_2 & \ldots & \matNull\\
        \vdots      & \vdots     & \ddots & \vdots\\
        \matNull    & \matNull   & \ldots & d'_k \matI_k}}_{\matD}
    \underbrace{\pmat{
        \matU'_1  & \matNull   & \ldots & \matNull\\
        \matNull  & \matU'_2   & \ldots & \matNull\\
        \vdots    & \vdots     & \ddots & \vdots\\
        \matNull  & \matNull   & \ldots & \matU'_k}}_{\matU}
    =
    \matD.
  \end{equation}
\end{lemma}

\begin{proof}
  The left-hand side is block-diagonal and block $i$ is $\matU'^T_i
  d'_i \matI_i \matU'_i = d'_i \matI_i$, leading to a diagonal matrix
  which is identical to $\matD$.
\end{proof}

\lemmasep
%-----------------------------------------------------------------------------

\begin{lemma}\label{lemma_orthotrans_diag}
  If\/ $\matR$ is an orthogonal matrix and\/ $\matD$ a diagonal matrix
  with pairwise different diagonal entries, the equation
  \begin{equation}
    \matR^T \matD \matR = \matD
  \end{equation}
  is fulfilled if and only if\/ $\matR = \matXi$ where $\matXi$ is a
  diagonal sign matrix ($\xi_i = \pm 1$).
\end{lemma}

\begin{proof}
  Follows from \lemmaref{\ref{lemma_commute_diag}}: Since $\matD \matR =
  \matR \matD$, $\matD$ and $\matR$ are commuting matrices, thus
  $\matR$ is diagonal. Now diagonal sign matrices $\matXi$ are the
  only diagonal orthogonal matrices: $\matXi^T\matXi = \matXi^2 =
  \matI$ implies $\xi_i^2 = 1$ and thus $\xi_i = \pm 1$.
\end{proof}

\lemmasep
%-----------------------------------------------------------------------------

\begin{lemma}\label{lemma_ortho_diag_ortho}
  Assume that $\matQ$ and $\matR$ are orthogonal matrices and\/
  $\matD$ a diagonal matrix with pairwise different diagonal
  entries. If\/ $\matQ^T \matD \matQ = \matR^T \matD \matR$, then
  $\matQ = \matXi \matR$.
\end{lemma}

\begin{proof}
  We obtain $\matR \matQ^T \matD \matQ \matR^T = \matD$. $\matS =
  \matQ \matR^T$ is also orthogonal. According to
  \lemmaref{\ref{lemma_orthotrans_diag}}, $\matS^T \matD \matS =
  \matD$ is only fulfilled for $\matS = \matXi$. Therefore $\matQ =
  \matXi \matR$.
\end{proof}

\lemmasep
%-----------------------------------------------------------------------------

\begin{lemma}\label{lemma_RTDR_ii_sqr}\footnote{Note
    that for compactness we write $\matA_{ij}^2 \coloneqq
    (\matA_{ij})^2$, so the expression refers to the squared element
    $(i,j)$ of the matrix $\matA$, not to the element $(i,j)$ of the
    matrix $\matA^2$.}
  Let\/ $\matD$ be a diagonal matrix with pairwise different diagonal
  entries and\/ $\matR$ an orthogonal matrix, both of dimension
  $n$. In the following, all sums run up to index $n$. The expression
  \begin{equation}
    \summe{i}{}\left(\matR^T \matD \matR \right)_{ii}^2
  \end{equation}
  is maximal for $\matR = \matXi \matP$ where $\matP$ is a permutation
  matrix. The maximum is $\sum_{i}(\matD)_{ii}^2$.
\end{lemma}

\begin{proof}
  Since the (squared) Frobenius norm is invariant under a
  similarity transformation (here as the special case of an orthogonal
  transformation), we have
  \begin{equation}
    \summe{i}{} \summe{j}{} (\matR^T\matD\matR)_{ij}^2
    =
    \summe{i}{} \summe{j}{} \matD_{ij}^2.
  \end{equation}
  Focusing on diagonal entries we can write
  \begin{equation}
    \summe{i}{} (\matR^T\matD\matR)_{ii}^2
    + \summe{i}{} \summe{j \neq i}{} (\matR^T\matD\matR)_{ij}^2
    =
    \summe{i}{} \matD_{ii}^2
  \end{equation}
  and therefore obtain the inequality
  \begin{equation}
    \summe{i}{} (\matR^T\matD\matR)_{ii}^2 \leq \summe{i}{} \matD_{ii}^2.
  \end{equation}
  To analyze in which cases both terms in the inequality are equal
  (and thus the left-hand side is maximal), we look at
  \begin{equation}
    \summe{i}{} \summe{j \neq i}{} (\matR^T\matD\matR)_{ij}^2 = 0.
  \end{equation}
  A sum of squared terms is zero if and only if all terms are
  zero. Here the squared terms are all off-diagonal elements of
  $\matR^T\matD\matR$, therefore $\matR^T\matD\matR$ is a diagonal
  matrix $\matD'$,
  \begin{equation}
    \matR^T \matD \matR = \matD',
  \end{equation}
  and thus $\matR = \matXi \matP$ according to
  \lemmaref{\ref{lemma_RTDR_Ds}}.
\end{proof}

\lemmasep
%-----------------------------------------------------------------------------

\begin{lemma}\label{lemma_RTDR_ii}
  Let\/ $\matR$ be an orthogonal matrix with row vectors $\vecr_i$
  and\/ $\matD$ a diagonal matrix. Then $(\matR^T \matD \matR)_{ii} =
  \vecr_i^T \matD \vecr_i$.
\end{lemma}

\begin{proof}
  Let vector $\vece_i$ denote a unit vector with element $1$ at
  position $i$. Since $\matR \vece_i = \vecr_i$ we get
  \begin{equation}
    (\matR^T \matD \matR)_{ii}
    = \vece_i^T \matR^T \matD \matR \vece_i
    = \vecr_i^T \matD \vecr_i.
  \end{equation}
\end{proof}

\lemmasep
%-----------------------------------------------------------------------------

\begin{lemma}\label{lemma_RTDR_diag}
  Let\/ $\matR$ be an orthogonal matrix with row vectors $\vecr_i$
  and\/ $\matD$ a diagonal matrix with elements $d_i \in
  [\check{d},\hat{d}]$. Then for the diagonal elements of\/
  $\matR^T\matD\matR$ we have $(\matR^T\matD\matR)_{ii} \in
  [\check{d},\hat{d}]$.
\end{lemma}

\begin{proof}
  \lemmaref{\ref{lemma_RTDR_ii}} leads to $(\matR^T\matD\matR)_{ii} =
  \vecr_i^T \matD \vecr_i$. Since $\matD$ is diagonal, its eigenvalues
  are the diagonal elements. By applying the Rayleigh-Ritz Theorem
  \cite[][sec.~4.2.2]{nn_Horn99} we immediately see that $\vecr_i^T
  \matD \vecr_i \in [\check{d},\hat{d}]$.
\end{proof}

\lemmasep
%-----------------------------------------------------------------------------

\begin{lemma}\label{lemma_RTDR_ii_b_i}
  Let\/ $\matD$ be a diagonal matrix with pairwise different and
  sorted elements ($d_1 > \ldots > d_n$). Let\/ $\vecb$ be a vector
  with pairwise different and sorted elements ($b_1 > \ldots >
  b_n$). Let\/ $\matR$ be an orthogonal matrix with column vectors
  $\vecr_i$. Then
  \begin{equation}\label{eq_RTDR_ii_b_i}
    \summe{i}{} (\matR^T \matD \matR)_{ii} b_i
  \end{equation}
  is maximal for $\matR = \matXi$.
\end{lemma}

\begin{proof}\footnote{The proof has some informal components
    and needs to be improved.}
  We sequentially determine $\vecr_1$ to $\vecr_n$ such that
  (\ref{eq_RTDR_ii_b_i}) is maximized. This sequential approach is
  justified by the fact that $\tr\{\matR^T\matD\matR\}$ is invariant
  under changes of $\matR$ (as the trace is invariant under an
  orthogonal transformation); it has the constant value $\sum_i
  d_i$. Thus maximizing a single diagonal element
  $(\matR^T\matD\matR)_{ii}$ will only redistribute the sum over the
  remaining diagonal elements. Since $\matD$ is diagonal, it has the
  eigenvalues $d_1\ldots d_n$. Since $\matD$ has pairwise different
  elements, eigenvalue $d_i$ belongs to eigenvector $\vecxxi_i$, where
  $\vecxxi_j$ is a unit vector with element $\pm 1$ at position
  $j$. According to \lemmaref{\ref{lemma_RTDR_ii}}, we have
  $(\matR^T\matD\matR)_{ii} = \vecr_i^T\matD\vecr_i$. We see that to
  maximize expression (\ref{eq_RTDR_ii_b_i}), the largest value $b_1$
  should be paired with the largest value of $\vecr_1^T \matD
  \vecr_1$. We know from the Rayleigh-Ritz Theorem
  \cite[][sec.~4.2.2]{nn_Horn99} that $\vecr_1^T\matD\vecr_1$ becomes
  maximal (value $d_1$) at the principal eigenvector $\vecr_1 =
  \vecxxi_1$. Choosing $\vecr_2$ perpendicular to $\vecr_1$, we pair
  the second-largest value $b_2$ with the largest value of $\vecr_2^T
  \matD \vecr_2$ (value $d_2$) which is found at $\vecr_2 = \vecxxi_2$
  Formally this can be shown by a deflation procedure, see
  e.g. \citet[sec.~2.2.8]{nn_Diamantaras96}. We continue this
  procedure up to $\vecr_n$ and obtain $\vecr_i = \vecxxi_i,\forall
  i=1\ldots n$ and thus $\matR = \matXi =
  (\vecxxi_1,\ldots,\vecxxi_n)$.
\end{proof}

\lemmasep
%-----------------------------------------------------------------------------

\begin{lemma}\label{lemma_DP}
  Let\/ $\matD$ be a diagonal matrix and $\matP$ a permutation matrix
  of the same size. Then $\matD \matP = \matP \matD^*$ where $\matD^*
  = \matP^T \matD \matP$ is diagonal and contains the same elements as
  $\matD$ but in permuted order.
\end{lemma}

\begin{proof}\footnote{From \url{https://math.stackexchange.com/questions/197243}, provided by user Jakub Konieczny}
  \begin{equation}
    \matD\matP = (\underbrace{\matP\matP^T}_{\matI}) \matD\matP =
    \matP(\matP^T\matD\matP) = \matP\matD^*
  \end{equation}
  based on \lemmaref{\ref{lemma_perm_diag}}.
\end{proof}

\lemmasep
%-----------------------------------------------------------------------------

\trarxiv{
  % Unused?
  \begin{lemma}\label{lemma_ortho_sign}
    Let\/ $\matQ$ be an orthogonal matrix and\/ $\matXi$ be a diagonal
    sign matrix. Then $\matQ' = \matQ \matXi$ or $\matQ' = \matXi \matQ$
    is also orthogonal.
  \end{lemma}
  
  \begin{proof}
    Both $\matQ$ and $\matXi$ are orthogonal matrices, so their product
    (in arbitrary order) is also orthogonal.
  \end{proof}
}{}

\lemmasep
%-----------------------------------------------------------------------------

\begin{lemma}\label{lemma_sorted_permute}
Let $\matA$ and\/ $\matB$ be diagonal matrices with pairwise
different, sorted, strictly positive diagonal elements (from vector
$\veca$ with $a_1 > a_2 > \ldots > a_n$ and from vector $\vecb$ with
$b_1 > b_2 > \ldots > b_n$, respectively). Let $\matA'$ (respectively
$\veca'$) be a permuted version of $\matA$ (respectively $\veca$):
$\matA' = \matP^T \matA \matP$ where $\matP$ is a permutation matrix
(\lemmaref{\ref{lemma_perm_diag}}). Then the expression $t =
\tr\{\matA' \matB\} = \veca'^T \vecb$ is maximal for $\matP = \matI$
(i.e. no permutation).
\end{lemma}

\begin{proof}
Assume that two elements of $\matA'$ are exchanged in their order
relative to $\matA$: $a'_k = a_j$ and $a'_l = a_i$ for $k < l$ and $i
< j$ with $a_i > a_j$. Then exchanging these two elements (while
keeping the remaining elements unchanged) will increase $t$:
\begin{eqnarray}
  a'_k b_k + a'_l b_l &=& a_j b_k + a_i b_l\\
  a'_l b_k + a'_k b_l &=& a_i b_k + a_j b_l\\
  a_i b_k + a_j b_l &>& a_j b_k + a_i b_l\\
  (a_i b_k + a_j b_l) - (a_j b_k + a_i b_l) &>& 0\\
  (a_i - a_j) b_k + (a_j - a_i) b_l &>& 0\\
  \underbrace{(a_i - a_j)}_{> 0}
  \underbrace{(b_k - b_l)}_{> 0} &>& 0
\end{eqnarray}
We can repeat this procedure for all elements exchanged in their order
in $\matA'$, changing them back to the order which they had in
$\matA$, leading to $\matP = \matI$.
\end{proof}

%############################################################################

\subsection{Derivatives}
%-----------------------

\begin{lemma}\label{lemma_scalar_vector_deriv}
  Assume that $a = a(\vecc)$ and\/ $\vecb = \vecb(\vecc)$. Then
  \begin{equation}
    \ddf{(a \vecb)}{\vecc}
    =
    \vecb \left(\ddf{a}{\vecc}\right)^T + a \ddf{\vecb}{\vecc}
  \end{equation}
  where
  \begin{equation}
    \ddf{\vecb}{\vecc} = \pmat{\ddf{b_j}{c_i}}_{ji}.
  \end{equation}
\end{lemma}

\begin{proof}
  We order the terms of the result matrix as in the usual definition
  of a Jacobian matrix --- the function index changes over the rows
  ($j$), the variable index over the columns ($i$):
  \begin{equation}
    \ddf{(a \vecb)}{\vecc} = \pmat{\ddf{(ab_j)}{c_i}}_{ji}.
  \end{equation}
  Individual elements of this matrix are determined by applying the
  product rule:
  \begin{equation}
    \ddf{(ab_j)}{c_i} = b_j \ddf{a}{c_i} + a \ddf{b_j}{c_i}.
  \end{equation}
  Returning to matrix-vector notation we obtain
\begin{equation}
  \pmat{b_j \ddf{a}{c_i} + a \ddf{b_j}{c_i}}_{ji}
  =
  \vecb \left(\ddf{a}{\vecc}\right)^T + a \ddf{\vecb}{\vecc}
\end{equation}
which is the expression provided in the lemma.
\end{proof}

\lemmasep
%-----------------------------------------------------------------------------

\begin{lemma}\label{lemma_scalprod_deriv}
  Assume that $\veca = \veca(\vecc)$ and $\vecb = \vecb(\vecc)$. Then
  \begin{equation}
    \ddf{(\veca^T\vecb)}{\vecc}
    =
    \ddf{\veca}{\vecc} \vecb + \ddf{\vecb}{\vecc} \veca
  \end{equation}
  where
  \begin{align}
    \ddf{\veca}{\vecc} &= \pmat{\ddf{a_j}{c_i}}_{ji}\\
    \ddf{\vecb}{\vecc} &= \pmat{\ddf{b_j}{c_i}}_{ji}.
  \end{align}
\end{lemma}

\begin{proof}
  Matrix terms are ordered as in
  \lemmaref{\ref{lemma_scalar_vector_deriv}}. We have
  \begin{equation}
    \ddf{(\veca^T\vecb)}{\vecc}
    =
    \pmat{\ddf{(\veca^T\vecb)}{c_i}}_i
  \end{equation}
  of which individual elements are determined by
  \begin{eqnarray}
    \ddf{(\veca^T\vecb)}{c_i}
    &=& \ddf{\sum_j a_j b_j}{c_i}\\
    &=& \sum_j \ddf{a_j}{c_i} b_j + \sum_j \ddf{b_j}{c_i} a_j\\
    &=&
    \left(\ddf{\veca}{c_i}\right)^T \vecb +
    \left(\ddf{\vecb}{c_i}\right)^T \veca.
  \end{eqnarray}
  This leads to
  \begin{eqnarray}
    \ddf{(\veca^T\vecb)}{\vecc}
    &=&
    \pmat{
      \left(\ddf{\veca}{c_i}\right)^T \vecb +
      \left(\ddf{\vecb}{c_i}\right)^T \veca}_i\\
    &=&
    \pmat{\ddf{a_j}{c_i}}_{ji} \vecb +
    \pmat{\ddf{b_j}{c_i}}_{ji} \veca\\
    &=&
    \ddf{\veca}{\vecc} \vecb +
    \ddf{\vecb}{\vecc} \veca.
  \end{eqnarray}
\end{proof}

%############################################################################

\subsection{Trace}
%-----------------

\begin{lemma}\label{lemma_tr_AB}\footnote{\url{https://en.wikipedia.org/wiki/Trace_(linear_algebra)}}
  \begin{equation}
    \tr\{\matA\matB\} = \summe{i}{}\summe{j}{} A_{ij} B_{ji}
  \end{equation}
\end{lemma}

\begin{proof}
  \begin{eqnarray}
    (\matA\matB)_{ij} &=& \summe{k}{} A_{ik} B_{kj}\\
    (\matA\matB)_{ii} &=& \summe{k}{} A_{ik} B_{ki}\\
    \tr\{\matA\matB\} &=& \summe{i}{} (\matA\matB)_{ii}
    = \summe{i}{}\summe{j}{} A_{ij} B_{ji}.
  \end{eqnarray}
\end{proof}

\lemmasep
%-----------------------------------------------------------------------------

\begin{lemma}\label{lemma_tr_AD}
  Let\/ $\matA$ be a square matrix and $\matD$ be a diagonal matrix
  (with diagonal elements $d_i$), both of the same dimension. Then
  \begin{equation}
    \tr\{\matA\matD\} = \summe{i}{} (\matA)_{ii} d_i.
  \end{equation}
\end{lemma}

\begin{proof}
  Using \lemmaref{\ref{lemma_tr_AB}} and $\matD = (d_i
  \delta_{ij})_{ij}$ we get
  \begin{equation}
    \tr\{\matA\matD\}
    =
    \summe{i}{}\summe{j}{} A_{ij} D_{ji}
    =
    \summe{i}{}\summe{j}{} A_{ij} d_j \delta_{ji}
    =
    \summe{i}{}A_{ii} d_i
    =
    \summe{i}{}(\matA)_{ii} d_i.
  \end{equation}
\end{proof}

\lemmasep
%-----------------------------------------------------------------------------

\begin{lemma}\label{lemma_tr_ATDAOmega}
  Let $\matA$ be any $n\times m$ matrix, and
  $\matD=\diag_{i=1}^n\{d_i\}$ and
  $\matOmega=\diag_{i=1}^m\{\varOmega_i\}$ diagonal matrices. Then
  \begin{equation}
    \tr\left\{\matA^T\matD\matA\matOmega\right\}
    =
    \summe{i=1}{m}\summe{k=1}{n} A^2_{ki} d_k \varOmega_i.
  \end{equation}
\end{lemma}

\begin{proof}
  \begin{align}
    (\matA^T\matD)_{ij}
    &= \summe{k=1}{n} A_{ki} D_{kj}
    = \summe{k=1}{n} A_{ki} d_j \delta_{kj}
    = A_{ji} d_j\\
    \label{eq_ATDA_ij}
    \left((\matA^T\matD)\matA\right)_{ij}
    &= \summe{k=1}{n} (\matA^T\matD)_{ik} A_{kj}
    = \summe{k=1}{n} A_{ki} d_k A_{kj}\\
    \left((\matA^T\matD\matA)\matOmega\right)_{ii}
    &= \summe{l=1}{m} \left((\matA^T\matD)\matA\right)_{il} \Omega_{li}\\
    &= \summe{l=1}{m} \left((\matA^T\matD)\matA\right)_{il}
    \varOmega_i \delta_{li}
    = \left((\matA^T\matD)\matA\right)_{ii} \varOmega_i\\
    &= \summe{k=1}{n} A_{ki} d_k A_{ki} \varOmega_i
    = \summe{k=1}{n} A^2_{ki} d_k \varOmega_i\\
    \tr\left\{\matA^T\matD\matA\matOmega\right\}
    &= \summe{i=1}{m} \left((\matA^T\matD\matA)\matOmega\right)_{ii}
    = \summe{i=1}{m} \summe{k=1}{n} A^2_{ki} d_k \varOmega_i.
  \end{align}
\end{proof}

\lemmasep
%-----------------------------------------------------------------------------

\begin{lemma}\label{lemma_tr_ATAD}
  Let $\matA$ be any $n \times m$ matrix and $\matD =
  \diag_{i=1}^{m}\{d_i\}$ a diagonal matrix. Then
  \begin{equation}
    \tr\left\{\matA^T \matA \matD\right\}
    =
    \summe{i=1}{m} \summe{k=1}{n} A^2_{ki} d_i.
  \end{equation}
\end{lemma}

\begin{proof}
  \begin{align}
    \label{eq_ATA_ij}
    \left(\matA^T\matA\right)_{ij}
    &= \summe{k=1}{n} A_{ki} A_{kj}\\
    \left((\matA^T\matA)\matD\right)_{ii}
    &= \summe{l=1}{m} (\matA^T\matA)_{il} D_{li}
    = \summe{l=1}{m} (\matA^T\matA)_{il} d_i \delta_{li}
    = (\matA^T\matA)_{ii} d_i\\
    &= \summe{k=1}{n} A^2_{ki} d_i\\
    \tr\left\{\matA^T \matA \matD\right\}
    &= \summe{i=1}{m} \left((\matA^T\matA)\matD\right)_{ii}
    =  \summe{i=1}{m} \summe{k=1}{n} A^2_{ki} d_i.
  \end{align}
\end{proof}
  
\lemmasep
%-----------------------------------------------------------------------------

\begin{lemma}\label{lemma_tr_AskewD}
  Let $\matA$ be any skew-symmetric matrix and $\matD$ be a diagonal matrix of
  the same size. Then
  \begin{equation}
    \tr\{\matA \matD\} = 0.
  \end{equation}
\end{lemma}

\begin{proof}
  Follows from \lemmaref{\ref{lemma_tr_AD}} since $(\matA)_{ii} = 0$.
\end{proof}

\lemmasep
%-----------------------------------------------------------------------------

\begin{lemma}\label{lemma_tr_Askew_Bsymm}
  Let $\matA$ be any skew-symmetric matrix ($\matA^T = -\matA$) and
  $\matB$ any symmetric matrix ($\matB^T = \matB$), both of size $n
  \times n$. Then
  \begin{equation}
    \tr\{\matA \matB\} = 0.
  \end{equation}
\end{lemma}

\begin{proof}
  We split the double sum, apply $\matA_{ij} = -\matA_{ji}$,
  $\matA_{ii} = 0$, $\matB_{ij} = \matB_{ji}$, and exchange indices:
  \begin{eqnarray}
    \tr\{\matA \matB\}
    &=&
    \summe{i=1}{n} (\matA \matB)_{ii}\\
    &=&
    \summe{i=1}{n} \summe{j=1}{n} \matA_{ij} \matB_{ji}\\
    &=&
    \summe{i=1}{n} \summe{j=i+1}{n} \matA_{ij} \matB_{ji}
    +
    \summe{j=1}{n} \summe{i=j+1}{n} \matA_{ij} \matB_{ji}
    +
    \summe{i=1}{n} \matA_{ii} \matB_{ii}\\
    &=&
    \summe{i=1}{n} \summe{j=i+1}{n} \matA_{ij} \matB_{ji}
    -
    \summe{j=1}{n} \summe{i=j+1}{n} \matA_{ji} \matB_{ij}\\
     &=&
    \summe{i=1}{n} \summe{j=i+1}{n} \matA_{ij} \matB_{ji}
    -
    \summe{i=1}{n} \summe{j=i+1}{n} \matA_{ij} \matB_{ji}\\
    &=&
    0.
  \end{eqnarray}
\end{proof}

%############################################################################

\subsection{Matrix Functions}
%----------------------------

% also \url{https://de.mathworks.com/help/symbolic/funm.html}
\begin{lemma}
  Let\/ $\matX$ be a diagonalizable $n \times n$ matrix such that\/
  $\matX = \matV \matD \matV^{-1}$. Note that this decomposition can
  be accomplished if $\matA$ has pairwise different eigenvalues
  \cite[][p.171]{nn_Abadir05} or if $\matA$ is symmetric
  \cite[][p.177]{nn_Abadir05}. Let\/ $\matF(\matX)$ be a matrix-valued
  function of matrix $\matA$. Let $f(x)$ be the scalar version of\/
  $\matF(\matX)$ for $n=1$. Then the Taylor expansion of
  $\matF(\matX)$ can be determined from the Taylor expansion of $f(x)$
  (here at the point $0$) by using the coefficients of the scalar
  series for the matrix series,\footnote{Taylor series:
    \url{https://en.wikipedia.org/wiki/Taylor_series}.} see
  \citet[][ch.9]{nn_Abadir05}, \citet[][ch.11,p.565]{nn_Golub96},
  \citet[][p.35]{nn_Diamantaras96}, and
  \citet[][p.152]{nn_Gentle17}. We start from the definition of
  $\matF(\matX)$:
  \begin{align}
    \matF(\matX)
    &\coloneqq
    \matV \diag\{f(d_1), \ldots, f(d_n)\} \matV^{-1}\\
    &=
    \matV
    \diag\left\{
    \summe{i=0}{\infty}c_i d_1^i, \ldots,
    \summe{i=0}{\infty}c_i d_n^i
    \right\}
    \matV^{-1}\\
    &=
    \summe{i=0}{\infty}
    \matV
    \diag\left\{c_i d_1^i, \ldots, c_i d_n^i\right\}
    \matV^{-1}\\
    &=
    \summe{i=0}{\infty}
    c_i
    \matV
    \matD^i
    \matV^{-1}\\
    &=
    \summe{i=0}{\infty}
    c_i
    \matX^i
    \end{align}
  with coefficients
  \begin{equation}
    c_i = \frac{f^{(i)}(0)}{i!}
  \end{equation}

  where the last step is motivated by
  \begin{equation}
    \matX^i
    = \underbrace{\matX \matX \ldots \matX}_{i\times}
    = \underbrace{(\matV\matD\matV^{-1}) (\matV\matD\matV^{-1})
      \ldots (\matV\matD\matV^{-1})}_{i\times}
    = \matV\matD^i\matV^{-1}.
  \end{equation}
  We provide the following examples\footnote{Taylor series calculator:
    \url{https://www.symbolab.com/solver/taylor-maclaurin-series-calculator}.}:
  \begin{align}
    f(x)
    =
    (1+x)^{-1}
    &=
    1 - x + x^2 - x^3 + \ldots\\
    \matF(\matX)
    =
    (\matI + \matX)^{-1}
    &=
    \matI - \matX + \matX^2 - \matX^3 + \ldots\\
    f(x)
    =
    (1+x)^{-\half}
    &=
    1 - \frac{1}{2} x + \frac{3}{8} x^2 - \frac{5}{16} x^3 + \ldots\\
    \label{eq_inv_sqrt_IplusX}
    \matF(\matX)
    =
    (\matI + \matX)^{-\half}
    &=
    \matI-\frac{1}{2}\matX+\frac{3}{8}\matX^2-\frac{5}{16}\matX^3 + \ldots
  \end{align}
  For $\nu \in \mathbb{R}$, we have \cite[][p.244]{nn_Abadir05}
  \begin{equation}
    \matF(\matX) = (\matI + \matX)^\nu
    = \summe{i=0}{\infty} c^\nu_i \matX^i
  \end{equation}
  with binomial coefficients (coefficients of the $i$-th derivative of
  $x^\nu$)
  \begin{equation}
    c^\nu_i = \pmat{\nu\\i} \coloneqq \frac{\produkt{j=0}{i-1} (\nu - j)}{i!}
  \end{equation}
  where $\prod_{j=0}^{-1} \coloneqq 1$ and $0! \coloneqq 1$.
\end{lemma}

%############################################################################

\subsection{Stiefel Manifold}
%----------------------------

\begin{lemma}\label{lemma_stiefel_tangent}
  Let $\matX$ be an $n \times m$ matrix ($m \leq n$) on the Stiefel
  manifold, thus $\matX^T \matX = \matI$. Let $\matDelta$ be a tangent
  direction ($n \times m$ matrix) of the Stiefel manifold at $\matX$. Then
  $\matX^T \matDelta$ is skew-symmetric:
  \begin{equation}\label{eq_stiefel_tangent_skew}
    \matX^T \matDelta + \matDelta^T \matX = \matNull
  \end{equation}
  \cite[][p.307]{nn_Edelman98}. Note that this implies
  \begin{equation}\label{eq_stiefel_tangent_perp}
    \tr\{\matX^T \matDelta\} = 0.
  \end{equation}
\end{lemma}

\lemmasep
%-----------------------------------------------------------------------------

\begin{lemma}\label{lemma_stiefel_tangent_para}
  The tangent directions $\matDelta$ ($n \times m$ matrix) on a
  Stiefel manifold at $\matX$ ($n \times m$ matrix, $m \leq n$,
  $\matX^T \matX = \matI_m$) are given by the parameterized expression
  \begin{equation}\label{eq_stiefel_tangent_para}
    \matDelta = \matX \matA + \matX_\perp \matB
  \end{equation}
  where parameter $\matA$ is any skew-symmetric $m \times m$ matrix
  ($\matA^T = -\matA$), parameter $\matB$ is any $(n-m) \times m$
  matrix, and $\matX_\perp$ is any $n \times (n-m)$ matrix
  complementing $\matX$ such that
  \begin{equation}
    (\matX | \matX_\perp) (\matX | \matX_\perp)^T
    = \matX \matX^T + \matX_\perp \matX^T_\perp = \matI_n
  \end{equation}
  \cite[][p.308]{nn_Edelman98}. Note that also
  \begin{align}
    (\matX | \matX_\perp)^T (\matX | \matX_\perp) &= \matI_n\\
    \matX^T_\perp \matX_\perp &= \matI_{n-m}\\
    \matX^T \matX_\perp &= \matNull_{m,n-m}.
  \end{align}
\end{lemma}

\lemmasep
%-----------------------------------------------------------------------------

\begin{lemma}\label{lemma_stiefel_tangent_proj}
  Let $\matX$ and $\matZ$ be $n \times m$ matrices ($m \leq n$). A
  Stiefel manifold is defined by $\matX^T \matX = \matI$. The
  projection $\matP_{\matX}(\matZ)$ of the matrix $\matZ$ onto the
  tangent space of the Stiefel manifold at $\matX$ is given by the
  tangent directions ($n \times m$ matrices)
  \begin{equation}\label{eq_stiefel_tangent}
    \matDelta
    = \matP_{\matX}(\matZ)
    = \matZ - \half \matX (\matX^T \matZ + \matZ^T \matX)
  \end{equation}
  \cite[][p.81]{nn_Absil08}.
\end{lemma}

\lemmasep
%-----------------------------------------------------------------------------

\newcommand{\matGSt}{\matG^{\mbox{\scriptsize St}}}
\begin{lemma}\label{lemma_stiefel_gradients}
  Let $\matX$ be an $n \times m$ matrix ($m \leq n$) on the Stiefel
  manifold, thus $\matX^T \matX = \matI$. Let $f(\matZ)$ be a
  scalar-valued function where $\matZ$ is an $n \times m$ matrix, and
  let $\matG_f(\matZ)$ denote the gradient of $f$. The gradient of $f$
  on the Stiefel manifold at $\matX$, denoted by $\matGSt_f(\matX)$,
  is defined to be a tangent vector of the Stiefel manifold at
  $\matX$. Two different versions have been suggested. In the
  ``embedded metric'' \cite[][p.81]{nn_Absil08}, equation
  (\ref{eq_stiefel_tangent}) from
  \lemmaref{\ref{lemma_stiefel_tangent_proj}} is applied to $\matZ =
  \matG_f$:
  \begin{equation}\label{eq_stiefel_gradient_embedded}
    \matGSt_f(\matX)
    = \matG_f(\matX)
    - \half \matX (\matX^T \matG_f(\matX) + \matG^T_f(\matX) \matX).
  \end{equation}
  In the ``canonical metric'', the gradient is given by
  \begin{equation}\label{eq_stiefel_gradient_canonical}
    \matGSt_f(\matX)
    = \matG_f(\matX)
    - \matX \matG^T_f(\matX) \matX
  \end{equation}
  \cite[][p.318]{nn_Edelman98}.
\end{lemma}

\lemmasep
%-----------------------------------------------------------------------------

\begin{lemma}\label{lemma_stiefel_svd}
  Let $\matXnull$ be an $n \times m$ matrix on a Stiefel manifold,
  thus $\matXnull^T \matXnull = \matI$. If such an $\matXnull$ can be
  found for a given arbitrary point $\matX$ and the condition $\|\matX
  - \matXnull\|_F < 1$ (Frobenius norm) holds, then the projection
  of\/ $\matX$ onto the Stiefel manifold exists, is unique, and is
  given by
  \begin{equation}\label{eq_stiefel_svd}
    \matW = \matP_{\matXnull}(\matX)
    = \matU \matV^T = \summe{i=1}{m} \vecu_i \vecv_i^T
  \end{equation}
  where $\matU$ and $\matV$ are obtained from the singular value
  decomposition of\/ $\matX$ 
  \begin{equation}
    \matX = \matU \matD \matV^T = \summe{i=1}{m} \vecu_i d_i \vecv_i^T
  \end{equation}
  with $\matU$ of size $n \times m$, $\matD$ of size $m \times m$ and
  $\matV$ of size $m \times m$
  \cite[][p.144]{nn_Absil12}\footnote{Erratum from
    \url{https://sites.uclouvain.be/absil/2010.038}
    applied.}.

  Expression (\ref{eq_stiefel_svd}) minimizes the Frobenius distance
  $\|\matX-\matW\|_F$.\footnote{See Theorem 2.1 and proof in
    \url{https://people.wou.edu/~beavers/Talks/LowdinJointMeetings0107.pdf}. The method is also known as 'L\"owdin (Symmetric) Orthogonalization'.}

  We also have\footnote{I'm grateful to P.-A. Absil for pointing this
    out and for referring to \cite{nn_Bhattacharya_12} (personal
    communication).}
  \begin{equation}\label{eq_stiefel_svd_mod}
    \matW = \matP_{\matXnull}(\matX) = \matX (\matX^T \matX)^{-\half}
  \end{equation}
  \cite[analyzed by][Theorem 10.2]{nn_Bhattacharya_12}.
\end{lemma}

\begin{proof}
  To prove the relation between (\ref{eq_stiefel_svd}) and
  (\ref{eq_stiefel_svd_mod}), we write the SVD of $\matX$ as
  \begin{equation}
    \matX = \matU \matD \matV^T
    = \matU (\matV^T \matV) \matD \matV^T
    = (\matU \matV^T) (\matV \matD \matV^T)
    = \matW (\matV \matD \matV^T).
  \end{equation}
  We also see that
  \begin{equation}
    \matX^T\matX
    = (\matV\matD\matU^T) (\matU\matD\matV^T)
    = \matV\matD^2\matV^T
  \end{equation}
  and thus
  \begin{equation}
    (\matX^T \matX)^\half
    = \matV\matD\matV^T
  \end{equation}
  since $(\matX^T\matX)^\half(\matX^T\matX)^\half =
  (\matV\matD\matV^T)(\matV\matD\matV^T) = \matV\matD^2\matV^T =
  \matX^T\matX$. From $\matX = \matW (\matV\matD\matV^T) = \matW
  (\matX^T\matX)^\half$ we finally get (\ref{eq_stiefel_svd_mod}).
  It is obvious that $\matW$ produced by (\ref{eq_stiefel_svd_mod})
  lies on the Stiefel manifold:
  \begin{align}
    \matW^T \matW
    &=
    \left[(\matX^T \matX)^{-\half} \matX^T\right]
    \left[\matX (\matX^T \matX)^{-\half}\right]\\
    &=
    (\matX^T \matX)^{-\half} (\matX^T\matX) (\matX^T \matX)^{-\half}\\
    &=
    \matI.
  \end{align}
\end{proof}

\lemmasep
%-----------------------------------------------------------------------------

\begin{lemma}\label{lemma_stiefel_proj_appr}
  Let $\matXnull$ be an $n \times m$ matrix on the Stiefel manifold,
  thus $\matXnull^T \matXnull = \matI$. The matrix $\matX = \matXnull
  + \matDelta$ is obtained by adding a tangent direction $\matDelta$
  to $\matXnull$. Then the projection of $\matX$ onto the Stiefel
  manifold can be approximated for small absolute values in the
  elements of $\matDelta$ as
  \begin{equation}\label{eq_stiefel_proj_appr}
    \matW
    = \matP_{\matXnull}(\matX)
    = \matX (\matX^T \matX)^{-\half}
    \approx \matX - \half \matXnull \matDelta^T \matDelta.
  \end{equation}
  This expression may be useful for the iterative update of learning
  rules for weight matrices: At the present point $\matXnull$ on the
  Stiefel manifold we compute the change of the weight matrix
  $\matDelta$ and from that a point $\matX$ outside the manifold. The
  next weight matrix is then obtained by projecting $\matX$
  approximately back to the Stiefel manifold using
  (\ref{eq_stiefel_proj_appr}) from
  \lemmaref{\ref{lemma_stiefel_proj_appr}}. While the projection back
  to the manifold is not exact, it may improve the performance over a
  learning rule which is just following the gradient (taken from the
  tangent space at $\matXnull$).
\end{lemma}

\begin{proof}
  We start from (\ref{eq_stiefel_svd_mod}) from
  \lemmaref{\ref{lemma_stiefel_svd}}, apply
  (\ref{eq_stiefel_tangent_skew}) from
  \lemmaref{\ref{lemma_stiefel_tangent}}, approximate up to linear
  terms using (\ref{eq_inv_sqrt_IplusX}), and omit third-order terms
  (since these are dominated by the second-order terms which appear in
  the equation):
  \begin{align}
    \matW
    &=
    \matX
    \left(\matX^T \matX\right)^{-\half}\\
    &=
    (\matXnull + \matDelta)
    \left[(\matXnull + \matDelta)^T (\matXnull + \matDelta)\right]^{-\half}\\
    &=
    (\matXnull + \matDelta)
    \left(\matI + \matDelta^T \matXnull + \matXnull^T \matDelta
    + \matDelta^T \matDelta\right)^{-\half}\\
    &=
    (\matXnull + \matDelta)
    \left(\matI + \matDelta^T \matDelta\right)^{-\half}\\
    &\approx
    (\matXnull + \matDelta)
    \left(\matI - \half \matDelta^T \matDelta\right)\\
    &\approx
    (\matXnull + \matDelta) - \half \matXnull \matDelta^T \matDelta\\
    &=
    \matX - \half \matXnull \matDelta^T \matDelta.
  \end{align}
  Note that the quadratic terms should not be omitted in the last
  approximation step, since otherwise $\matW$ would lie on the tangent
  space.
\end{proof}

\lemmasep
%-----------------------------------------------------------------------------

\begin{lemma}\label{lemma_stiefel_proj_appr_tangent}
  If we insert the parametrization of the tangent space of the Stiefel
  manifold from (\ref{eq_stiefel_tangent_para}) from
  \lemmaref{\ref{lemma_stiefel_tangent_para}} into the approximation
  of the projection to the Stiefel manifold from
  (\ref{eq_stiefel_proj_appr}) we obtain
  \begin{equation}\label{eq_stiefel_proj_appr_tangent}
    \matW
    \approx
    (\matXnull + \matXnull \matA + \matXnull_\perp \matB)
    - \half \matXnull (\matA^T \matA + \matB^T \matB).
  \end{equation}
  This equation may be useful for analyzing how a function on the
  Stiefel manifold behaves in the vicinity of a critical point, for
  small $\matDelta$ and accordingly small $\matA$ and $\matB$.
\end{lemma}

\begin{proof}
  We take the parametrization of the tangent space and consider the
  properties of $\matXnull$ and $\matXnull_\perp$, both from
  \lemmaref{\ref{lemma_stiefel_tangent_para}}:
  \begin{align}
    \matW
    &\approx
    \matX - \half \matXnull \matDelta^T \matDelta\\
    &=
    (\matXnull + \matDelta) - \half \matXnull \matDelta^T \matDelta\\
    &=
    (\matXnull + \matXnull \matA + \matXnull_\perp \matB)
    - \half \matXnull
    (\matXnull \matA + \matXnull_\perp \matB)^T
    (\matXnull \matA + \matXnull_\perp \matB)\\
    &=
    (\matXnull + \matXnull \matA + \matXnull_\perp \matB)
    - \half \matXnull (\matA^T \matA + \matB^T \matB).
  \end{align}
\end{proof}

%#############################################################################

\subsection{Further Lemmata}
%---------------------------

\begin{lemma}\label{lemma_ev_xi}
  If\/ $\matV$ contains the eigenvectors of $\matC$ in its columns
  such that $\matC \matV = \matV \matLambda$, then $\matV' = \matV
  \matXi$ is also an eigenvector matrix of $\matC$.
\end{lemma}

\begin{proof}
  $\matC \vecv'_i = \vecv'_i \lambda_i \Leftrightarrow \matC \vecv_i
  \xi_i = \vecv_i \xi_i \lambda_i \Leftrightarrow \matC \vecv_i =
  \vecv_i \lambda_i$.  Essentially we see that eigenvectors are only
  determined up to their signs (for the real-valued case).
\end{proof}

\lemmasep
%-----------------------------------------------------------------------------

\begin{lemma}\label{lemma_spectral}
  Let\/ $\matA$ be a symmetric matrix with eigenvalues $\lambda_i$
  collected in $\matLambda = \diag_{i=1}^{n}\{\lambda_i\}$ and
  corresponding eigenvectors $\vecx_i$ collected in the columns of\/
  $\matX$. Then
  \begin{equation}
    \matA
    = \matX \matLambda \matX^T
    = \summe{i=1}{n} \lambda_i \vecx_i \vecx_i^T. 
  \end{equation}
  is the spectral decomposition of $\matA$. This decomposition is
  unique up to the sign of the eigenvectors (see
  \lemmaref{\ref{lemma_ev_xi}}) if the eigenvalues are pairwise
  different, i.e. $\lambda_i \neq \lambda_j$ for $i \neq
  j$.\footnote{See e.g.  \citet[][p.155]{nn_Gentle17},
    \citet[][p.36]{nn_Diamantaras96} and
    \url{https://nlp.stanford.edu/IR-book/html/htmledition/matrix-decompositions-1.html}.}
  Note that in this case $\matX$ is orthogonal, i.e. $\matX^T\matX =
  \matX\matX^T=\matI$.\footnote{See
    e.g. \url{http://www.quandt.com/papers/basicmatrixtheorems.pdf}.}
\end{lemma}

\lemmasep
%-----------------------------------------------------------------------------

\begin{lemma}\label{lemma_blockdiag_evec}
  Let\/ $\matM$ be a symmetric block-diagonal $n \times n$ matrix
  \begin{equation}
    \matM = \pmat{\matA & \matNull\\\matNull & \matB}
  \end{equation}
  where $\matA$ is a symmetric $m \times m$ and\/ $\matB$ a symmetric
  $k \times k$ matrix with $n = m + k$. Assume that $\matM$ has the
  set of eigenvalues $\{\lambda_1, \ldots, \lambda_n\}$ which are
  pairwise different. Further assume an order of the eigenvalues
  such that $\lambda_1,\ldots,\lambda_m$ are eigenvalues of\/ $\matA$
  while $\lambda_{m+1},\ldots,\lambda_n$ are eigenvalues of\/
  $\matB$. Let $\lambda_i$ be the element $i$ on the diagonal of the
  diagonal matrix $\matLambda$. The spectral decomposition of\/
  $\matM$ is given by $\matM = \matE \matLambda \matE^T$. Then the
  eigenvector matrix $\matE$ (with eigenvectors in its columns) of\/
  $\matM$ is block-diagonal
  \begin{equation}
  \matE =   \pmat{\matX & \matNull\\\matNull & \matY}
  \end{equation}
  where $\matX$ is an orthogonal $m \times m$ and\/ $\matY$ an
  orthogonal $k \times k$ matrix.
\end{lemma}

\begin{proof}
  The eigenvectors of $\matM$ in the columns of $\matE$,
  \begin{equation}
    \vece_i = \pmat{\vecx_i\\ \vecy_i},
  \end{equation}
  can be determined from
  \begin{equation}
    \pmat{\matA & \matNull\\\matNull & \matB} \pmat{\vecx_i\\\vecy_i}
    =
    \lambda_i \pmat{\vecx_i\\\vecy_i}
  \end{equation}
  which can be split into independent eigen equations
  \begin{eqnarray}
    \matA \vecx_i &=& \lambda_i \vecx_i
    \label{eq_eig_x}
    \\
    \matB \vecy_i &=& \lambda_i \vecy_i.
    \label{eq_eig_y}
  \end{eqnarray}
  It is known that the set of the eigenvalues of the block-diagonal
  matrix $\matM$ is the union of the sets of eigenvalues of its blocks
  $\matA$ and $\matB$, respectively. If $\lambda_i$ is an eigenvalue
  of $\matA$, leading to a non-trivial solution of (\ref{eq_eig_x}),
  it cannot be an eigenvalue of $\matB$ (since all eigenvalues are
  pairwise different), and thus (\ref{eq_eig_y}) only has the trivial
  solution $\vecy_i = \vecnull$. The opposite holds if $\lambda_i$ is
  an eigenvalue of $\matB$. (Trivial overall solutions $\vece_i =
  \vecnull$ are excluded.) Thus we get the corresponding eigenvectors
  of $\matM$ for eigenvalue $\lambda_i$
  \begin{equation}
    \pmat{\vecx_i\\ \vecnull}\;\mbox{for}\; 1\leq i \leq m,\;
    \pmat{\vecnull\\ \vecy_i}\;\mbox{for}\; m+1 \leq i \leq n
  \end{equation}
  and therefore
  \begin{equation}\label{eq_ev_blockdiag}
    \matE
    =
    \pmat{
      \matX & \matNull\\
      \matNull & \matY}
    =
    \pmat{
      \vecx_1 \ldots \vecx_m & \matNull\\
      \matNull & \vecy_{m+1} \ldots \vecy_n}.
  \end{equation}
  Both $\matX$ and $\matY$ are orthogonal matrices since the
  eigenvectors of symmetric matrices are orthogonal.
\end{proof}

\lemmasep
%-----------------------------------------------------------------------------

\begin{lemma}\label{lemma_AD_ii}
  Let $\matA$ be any square matrix of size $n\times n$ and $\matD =
  \diag_{i=1}^{n}\{d_i\}$ be a diagonal matrix of the same size. Then
  \begin{equation}
    (\matA \matD)_{ii} = (\matA)_{ii} d_i.
  \end{equation}
\end{lemma}

\begin{proof}
  \begin{equation}
  (\matA\matD)_{ii}
  = \summe{k}{} A_{ik} D_{ki}
  = \summe{k}{} A_{ik} d_i \delta_{ki}
  = A_{ii} d_i.
  \end{equation}
\end{proof}

\lemmasep
%-----------------------------------------------------------------------------

\begin{lemma}\label{lemma_ATD_DA_ii}
Let $\matA$ be any square matrix of size $n\times n$ and $\matD =
\diag_{i=1}^{n}\{d_i\}$ be a diagonal matrix of the same size. Then
\begin{equation}
(\matA^T\matD)_{ii} = (\matD\matA)_{ii}.
\end{equation}
\end{lemma}

\begin{proof}
  According to \lemmaref{\ref{lemma_AD_ii}}, we have
  $(\matA\matD)_{ii} = A_{ii}d_i$. We also have $(\matA^T\matD)_{ii} =
  \summe{k=1}{n} A_{ki} D_{ki} = \summe{k=1}{n} A_{ki} d_i \delta_{ki}
  = A_{ii} d_i$.
\end{proof}
  
\lemmasep
%-----------------------------------------------------------------------------

\begin{lemma}\label{lemma_AskewD_ii}
  Let $\matA$ be any skew-symmetric matrix of size $n\times n$ and
  $\matD = \diag_{i=1}^{n}\{d_i\}$ be a diagonal matrix of the same
  size. Then
  \begin{equation}
    (\matA \matD)_{ii} = 0.
  \end{equation}
\end{lemma}

\begin{proof}
  Follows from \lemmaref{\ref{lemma_AD_ii}} since $A_{ii} = 0$.
\end{proof}

\lemmasep
%-----------------------------------------------------------------------------

\begin{lemma}\label{lemma_square_square_ii}
  Let $\matA$ and $\matB$ be square matrices of dimension $n$. Then
  \begin{equation}
    (\matA \matB)_{ii}
    = \left([\matA \matB]^T\right)_{ii}
    = \left(\matB^T \matA^T\right)_{ii}.
  \end{equation}
\end{lemma}

\lemmasep
%-----------------------------------------------------------------------------

\trarxiv{
  % Unused?
  \begin{lemma}\label{lemma_symm_symm_ii}
    Let $\matA$ and $\matB$ be symmetric matrices of dimension $n$. Then
    \begin{equation}
      (\matA \matB)_{ii}
      = \left([\matA \matB]^T\right)_{ii}
      = \left(\matB^T \matA^T\right)_{ii}
      = (\matB \matA)_{ii}.
    \end{equation}
  \end{lemma}
}{}

\lemmasep
%-----------------------------------------------------------------------------

\trarxiv{
  % Unused?
  \begin{lemma}\label{lemma_skew_symm_ii}
    Let $\matA$ be a skew-symmetric matrix of dimension $n$,
    i.e. $\matA^T = -\matA$. Let $\matB$ be a symmetric matrix of
    dimension $n$, i.e. $\matB^T = \matB$. Then
    \begin{equation}
      (\matA \matB)_{ii}
      = \left([\matA \matB]^T\right)_{ii}
      = \left(\matB^T \matA^T\right)_{ii}
      = - (\matB \matA)_{ii}.
    \end{equation}
  \end{lemma}
}{}

\lemmasep
%-----------------------------------------------------------------------------

\begin{lemma}\label{lemma_mult_INull}
  Let\/ $\matA_n$ be a square matrix of dimension $n$ of block-diagonal shape:
  \begin{equation}
    \matA_n = \pmat{\hat{\matA}_m&\matNull\\\matNull&\check{\matA}_{n-m}}.
  \end{equation}
  Then
  \begin{equation}
    \matA_n \pmat{\matI_n\\\matNull} = \pmat{\hat{\matA}_n\\\matNull}.
  \end{equation}
  Let\/ $\matB_m$ be a square matrix of dimension $m$. Then
  \begin{equation}
    \pmat{\matI_m\\\matNull} \matA_m = \pmat{\matA_m\\\matNull}.
  \end{equation}
\end{lemma}

\lemmasep
%-----------------------------------------------------------------------------

\begin{lemma}\label{lemma_skew_symm_diag_nonzero}
  Let $\matB$ denote a symmetric matrix ($\matB^T = \matB$) of size $n
  \times n$ which has at least one non-zero off-diagonal element, i.e.
  \begin{equation}\label{eq_B_nonzero_offdiag}
    \exists k,j,\,\, k\neq j: \matB_{kj} \neq 0.
  \end{equation}
  Then for all such matrices $\matB$ we can find a skew-symmetric
  matrix $\matA$ ($\matA^T = -\matA$) of size $n \times n$ such that
  there is a diagonal element $(\matA\matB)_{jj}$ which is non-zero:
  \begin{equation}
    \forall \matB: \exists \matA: \exists j: (\matA\matB)_{jj} \neq 0.
  \end{equation}
\end{lemma}

\begin{proof}\footnote{From \url{https://math.stackexchange.com/questions/3662881}, proof kindly provided by user 'lcv', slightly modified.}
  Fix an index $j$ such that $\exists k$ with $k \neq j$ where
  $\matB_{kj} \neq 0$. Such an index $j$ exists according to
  (\ref{eq_B_nonzero_offdiag}).

  Then form the following matrix:
  \begin{equation}
    \matA = \summe{i=1}{n}
    \left\{
    \sign(B_{ij}) \vece_j \vece_i^T - \sign(B_{ij}) \vece_i \vece_j^T
    \right\}.
  \end{equation}
  Matrix $\matA$ has zero elements except at row $j$ where it has
  elements $\sign(B_{ij})$ and at column $j$ where it has elements
  $-\sign(B_{ij})$. This matrix is skew-symmetric by construction. Its
  elements are
  \begin{equation}
  A_{jk} = \sign(B_{jk}) - \sign(B_{jj}) \delta_{jk},
  \end{equation}
  i.e. row $j$ contains the elements $\sign(B_{jk})$ except at the
  main diagonal where it is zero. Now
  \begin{align}
    (\matA\matB)_{jj}
    &=
    \summe{k=1}{n} A_{jk} B_{kj}\\
    &=
    \summe{k=1}{n} A_{jk} B_{jk}\\
    &=
    \summe{k=1}{n}
    \left\{\sign(B_{jk}) - \sign(B_{jj}) \delta_{jk}\right\} B_{jk}\\
    &=
    \summe{k=1}{n} \sign(B_{jk}) B_{jk}
    -
    \summe{k=1}{n} \sign(B_{jj}) \delta_{jk}  B_{jk}\\
    &=
    \summe{k=1}{n} \sign(B_{jk}) B_{jk}
    -
    \sign(B_{jj}) B_{jj}\\
    &=
    \summe{k=1}{n} |B_{jk}|
    -
    |B_{jj}|\\
    &=
    \summe{k=1,k\neq j}{n} |B_{jk}|\\
    &\neq 0
  \end{align}
  since at least one element of this sum is non-zero according to
  (\ref{eq_B_nonzero_offdiag}). The sum is also strictly positive. We
  applied $\sign(x)\cdot x = |x|$.
\end{proof}

%#############################################################################
% Global Maximima of Objective Functions
%#############################################################################

\trarxiv{%
  \input symmpca-global-maxima.tex
}{}

\end{document}